\documentclass[a4paper,10pt]{article}%

\usepackage[latin1]{inputenc}%
\usepackage{enumerate}%

\usepackage{amsmath}%
\usepackage{amsfonts,amssymb,MnSymbol}%
\usepackage{amsthm}%

\usepackage[all,cmtip]{xy}%
\newdir{ >}{{}*!/-5pt/@{>}}%
\usepackage{array}%
\usepackage[dvipsnames]{xcolor}%
\usepackage{ulem}%
\usepackage{dsfont}%
\usepackage{bbm}%
\usepackage{paralist}%

\usepackage{tikz-cd}%
\usetikzlibrary{patterns}%
\tikzset{%
  treenode/.style = {shape=rectangle, rounded corners,%
                     draw, align=center,%
                     top color=white, bottom color=blue!20},%
  root/.style     = {treenode, font=\Large, bottom color=red!30},%
  env/.style      = {treenode, font=\ttfamily\normalsize},%
  dummy/.style    = {circle,draw,inner sep=0pt,minimum size=1.4mm}%
}%

\usepackage{todonotes}%

\usepackage[inline]{enumitem}%

\newtheorem{theorem}[equation]{Theorem}%
\newtheorem{lemma}[equation]{Lemma}%
\newtheorem{proposition}[equation]{Proposition}%
\newtheorem{corollary}[equation]{Corollary}%

\theoremstyle{definition}%
\newtheorem{definition}[equation]{Definition}%
\newtheorem{example}[equation]{Example}%
\newtheorem{remark}[equation]{Remark}%
\newtheorem{notation}[equation]{Notation}%
\renewcommand{\O}{\ensuremath{\mathcal{O}}}%
\DeclareMathOperator{\colim}{colim}%

\numberwithin{equation}{section}%

\author{Lu\'is Alexandre Pereira}%
\title{Equivariant dendroidal sets}%

\usepackage{tikz}%
\tikzset{%
  treenode/.style = {shape=rectangle, rounded corners,%
                     draw, align=center,%
                     top color=white, bottom color=blue!20},%
  root/.style     = {treenode, font=\Large, bottom color=red!30},%
  env/.style      = {treenode, font=\ttfamily\normalsize},%
  dummy/.style    = {circle,draw,inner sep=0pt,minimum size=2mm}%
}%

\usetikzlibrary[decorations.pathreplacing]

\begin{document}

\maketitle%

\begin{abstract}
	We extend the Cisinski-Moerdijk-Weiss theory of $\infty$-operads to the equivariant setting to obtain a notion of $G$-$\infty$-operads that encode ``equivariant operads with norm maps'' up to homotopy. At the root of this work is the identification of a suitable category of $G$-trees together with a notion of $G$-inner horns capable of encoding the compositions of norm maps.
	
	Additionally, we follow Blumberg and Hill by constructing suitable variants associated to each of the indexing systems  featured in their work.
\end{abstract}

\tableofcontents

\section{Introduction}%

Operads encode a variety of algebraic structures, such as monoids, commutative monoids or (depending on the ambient category) Lie algebras, $E_n$-algebras, etc. Indeed, all such instances can be regarded as categories of algebras for some (fixed) suitable operad.
Informally, an operad $\O$ consists of ``sets/spaces of $n$-ary operations'' $\O(n)$, $n \geq 0$, each of which carries a $\Sigma_n$-action encoding ``reordering the inputs of the operations'', and a suitable notion of ``composition of operations''.

From the homotopy theory point of view, one of the most important classes of operads is certainly that of the $E_\infty$-operads, which are ``up to homotopy'' replacements of the commutative operad $\mathsf{Com}$. More concretely, while algebras for $\mathsf{Com}$ are the usual commutative monoids, the algebras for an $E_\infty$-operad are ``up to homotopy commutative monoids'', where associativity and commutativity are only enforced up to homotopy. Further, $E_{\infty}$-operads $\O$ are characterized by the property that each space $\O(n)$ is a contractible space with a free $\Sigma_n$-action.

This work lies at the intersection of operad theory and equivariant homotopy theory. Briefly, in $G$-equivariant homotopy theory a map of $G$-spaces $X \to Y$ is considered a $G$-weak equivalence only if all the induced fix point maps $X^H \to Y^H$, $H \leq G$ are weak equivalences. Therefore, it is no surprise that the characterization of $G$-$E_\infty$-operads, i.e. $G$-equivariant operads whose algebras are ``$G$-equivariant up to homotopy commutative monoids'' would need to be modified. Indeed, a naive first guess might be that a $G$-operad $\O$ should be called $G$-$E_\infty$ if
\begin{inparaenum}
\item[(i)] each space $\O(n)$ has a free $\Sigma_n$-action and
\item[(ii)] $\O(n)$ is $G$-contractible.
\end{inparaenum}
Accepting this tentative characterization for the moment, such a $G$-operad  is easily produced: simply taking a (non-equivariant) $E_\infty$-operad and giving it a trivial $G$-action yields such an example.
 However, it has long been known \cite{CW91} that such ``$G$-trivial $E_\infty$-operads'' are not the correct replacement for the commutative operad in the equivariant setting.
To see why, we consider the much studied example of $R$ a (strictly) commutative $G$-ring spectrum. For a finite $G$-set $T$ with $n$ elements it is possible to equip 
$R^{\wedge T} \simeq R^{\wedge n}$ with a mixed $G$-action combining the actions on $R$ and $T$. One often writes
$N^T R$ for $R^{\wedge T}$ together with this action and calls it the Hill-Hopkins-Ravenel \textit{norm}. Multiplication then induces \textit{norm maps}
\begin{equation}\label{NORM MAPS}
	N^T R \to R
\end{equation}
satisfying equivariance and associativity conditions.
The flaw of ``$G$-trivial $E_\infty$-operads'' is then that they lack all norm maps (\ref{NORM MAPS}) with $T$ a non-trivial $G$-set (or, after restriction to $H\leq G$, $T$ a non-trivial $H$-set). 

In understanding this issue, note first that though 
 $\O(n)$ has a $G\times \Sigma_n$-action when $\O$ is a $G$-operad,
conditions (i) and (ii) above  actually fail to determine a \textit{unique} $G\times \Sigma_n$-homotopy type. Indeed, (i) implies that $\O(n)^{\Gamma}=\emptyset$ whenever $\Gamma \cap \Sigma_n \neq \**$ while (ii) implies that 
$\O(n)^{\Gamma} \sim \**$ if $\Gamma \leq G$, but these conditions leave out many subgroups $\Gamma \leq G \times \Sigma_n$. Indeed, there are identifications
\begin{equation}
\Gamma \text{ such that } \Gamma \cap \Sigma_n = \**
	\leftrightarrow
\text{graph of } G\geq H \to \Sigma_n
	\leftrightarrow
H\text{-action on } \{1,\cdots,n\}
\end{equation}
and (ii) covers only those $\Gamma$ encoding trivial $H$-actions.
The correct characterization of $G$-$E_{\infty}$-operads is then that:
\begin{inparaenum}
	\item[(i)] $\O(n)$ is $\Sigma_n$-free;
	\item[(ii')] $\O(n)$ is graph-contractible, i.e. $\O(n)^{\Gamma}\sim \**$ for any $\Gamma \leq G \times \Sigma_n$ such that $\Gamma \cap \Sigma_n = \**$.
\end{inparaenum}

A key observation of Blumberg and Hill in \cite{BH15} is that the reason a ``$G$-trivial $E_{\infty}$-operad'' induces only the norm maps for trivial sets is that it satisfies (ii') only for those $\Gamma$ encoding trivial sets. Indeed, their work takes this observation much further.
Motivated by the study of equivariant spectra in incomplete universes, they define a whole lattice of types of $G$-operads, which they dub \textit{$N_{\infty}$-operads}, and which satisfy (ii') only for $\Gamma$ encoding $H$-sets within certain special families. Further, they call such special families \textit{indexing systems}. ``$G$-trivial $E_{\infty}$'' and $G$-$E_{\infty}$ are then the minimum and maximum types of $N_{\infty}$-operads, with the remaining types interpolating in between.

The motivation for this paper (and the larger project it belongs to) is the observation that the closure conditions for the $H$-sets in an indexing system identified in \cite[Def. 3.22]{BH15} admit a nice diagrammatic interpretation (discussed in \S \ref{INDEXSYS SEC}) and that this suggests the possibility of encoding equivariant operads with norm maps (i.e. with operations in $\O(n)^{\Gamma}$) via suitable diagrammatic models.

Indeed, it is well known that composition of operations in an operad can be encoded using tree diagrams and work of Moerdijk and Weiss in \cite{MW08} and follow up work of Cisinski and Moerdijk in \cite{CM11} builds a category $\Omega$ of trees and a model structure on the presheaf category 
$\mathsf{dSet}=\mathsf{Set}^{\Omega^{op}}$
which is shown in the follow up papers \cite{CM13a} and \cite{CM13b} to be Quillen equivalent to the category of
colored simplicial operads.

The role of this paper is to provide the equivariant analogue of the work in \cite{MW08} and \cite{CM11}. We first identify a (non-obvious) category $\Omega_G$ of $G$-trees (introduced in 
$\S \ref{EQUIVTREESINTO SEC}$ and formally defined in 
$\S \ref{EQUIVTREECAT SEC}$) capable of encoding norm maps and their compositions, and then adapt the proofs in \cite{MW08} and \cite{CM11} to prove the existence of a model structure on $\mathsf{dSet}^G$ whose fibrant objects, which we call $G$-$\infty$-operads, are ``up to homotopy $G$-operads with norm maps''. We note that our results are not formal: indeed, while our proofs closely follow those in \cite{MW08} and \cite{CM11} the presence of equivariance often requires significant modifications. Moreover, we note that alternative so called ``genuine'' model structures on $\mathsf{dSet}^G$ built by formal methods (say, by mimicking the definition of the genuine model structure on $\mathsf{Top}^G$) would instead only model $G$-operads \textit{without non trivial norm maps}.

\section*{Acknowledgments}

The current work owes much to Peter Bonventre: the category $\Omega_G$ of $G$-trees, which is at the root of this paper, is a joint discovery with him; the author first heard of the notion of broad posets (due to Weiss) from him; and lastly, this work has also greatly benefited from extensive joint conversations.

I would also like to thank the anonymous referee for his or hers helpful suggestions.


\section{Main results}

Our main result follows. It is the equivariant analogue of \cite[Thm. 2.4]{CM11}.

\begin{theorem}\label{GINFTYOP THM}
	There exists a model structure on $\mathsf{dSet}^G$ such that
	\begin{itemize}
		\item the cofibrations are the $G$-normal monomorphisms;
		\item the fibrant objects are the $G$-$\infty$-operads;
		\item the weak equivalences are the smallest class containing the $G$-inner anodyne extensions, the trivial fibrations and closed under ``$2$-out-of-$3$''.
	\end{itemize}
\end{theorem}

Theorem \ref{GINFTYOP THM} will be proven as the combination of Proposition \ref{DSETGMODEL PROP}, Theorem \ref{FIBRANTOBJ THM} and 
Corollary \ref{WEAKEQUIVCHAR}.

Further, letting $\mathcal{F}$ denote an indexing system (cf. \cite[Def. 3.22]{BH15}, and as reinterpreted in Definition \ref{WEAKINDEXSYS DEF}), we also prove the following more general result.

\begin{theorem}\label{INDEXSYSMAIN THM}
For $\mathcal{F}$ a weak indexing system there
exists a model structure on $\mathsf{dSet}^G$ such that
	\begin{itemize}
		\item the cofibrations are the $\mathcal{F}$-normal monomorphisms;
		\item the fibrant objects are the $\mathcal{F}$-$\infty$-operads;
		\item the weak equivalences are the smallest class containing the $\mathcal{F}$-inner anodyne extensions, the $\mathcal{F}$-trivial fibrations and closed under ``$2$-out-of-$3$''.
	\end{itemize}
\end{theorem}

Theorem \ref{INDEXSYSMAIN THM}
is proven at the end of \S \ref{INDEXSYS SEC}.

\begin{remark}
In the special case where $\mathcal{F}$ is the indexing system containing only the trivial $H$-sets, the model structure given by Theorem \ref{INDEXSYSMAIN THM} coincides with the ``formal genuine model structure'' as built using \cite[Prop. 2.6]{St16} (for the collection of all subgroups $H \leq G$). However, this is not the case for either Theorem \ref{GINFTYOP THM} or indeed the vast majority of instances of Theorem \ref{INDEXSYSMAIN THM}, which are not formal consequences of the existence of the model structure on $\mathsf{dSet}$.
\end{remark}

\section{Outline}

After reviewing the familiar types of trees found elsewhere in the literature (e.g. \cite{MW07}, \cite{MW08}, \cite{CM11}, among others), 
\S \ref{INTROGTREES SEC} provides an introductory look at the new equivariant trees that motivate this paper, focusing on examples. Most notably, each $G$-tree can be represented by two \textit{distinctly} shaped tree diagrams, called the 
\textit{expanded} and \textit{orbital} representations, each capturing different key features.

\S \ref{TREESFORESTS SEC} lays the necessary framework for our work.
Specifically, \S \ref{BROADPOSETS SEC} recalls Weiss' algebraic \textit{broad poset} model \cite{We12} for the category 
$\Omega$ of trees, which we prefer since planar representations of $G$-trees can easily get prohibitively large.
 \S \ref{CATFOR SEC} discusses forests, which play an auxiliary role.
 \S \ref{EQUIVTREECAT SEC} formally introduces the category $\Omega_G$ of $G$-trees.
Lastly, \S \ref{PRESHEAFCAT SEC} introduces 
all the necessary presheaf categories, most notably the category 
$\mathsf{dSet}^G$ featured in Theorems \ref{GINFTYOP THM} and 
\ref{INDEXSYSMAIN THM}.

\S \ref{NORMONANODYNEEXT SEC} discusses the notions of $G$-normal monomorphism and $G$-$\infty$-operad needed to state 
Theorem \ref{GINFTYOP THM}. The former of these is straightforward, but the latter requires the key (and more subtle) notion of $G$-inner horn (Definition \ref{INNERHORNG DEF}).

\S \ref{TENSORPROD SEC} is the technical heart of the paper, extending the key technical results \cite[Prop 9.2]{MW08} and \cite[Thms. 5.2 and 4.2]{CM11} concerning tensor products of dendroidal sets and the dendroidal join to the equivariant setting (Theorems \ref{EXPNPROP THM}, \ref{OUTERIN THM}, \ref{JOINLIFT THM}).

\S \ref{MODELSTR SEC} then finishes the proof of Theorem \ref{GINFTYOP THM} by combining the results of \S \ref{TENSORPROD SEC} with the arguments in the proof of the original non-equivariant result \cite[Thm. 2.4]{CM11}. 

Finally, \S \ref{INDEXSYS SEC} proves Theorem \ref{INDEXSYSMAIN THM} by straightforward generalizations of our arguments to the
framework of general indexing systems.

\section{An introduction to equivariant trees}
\label{INTROGTREES SEC}

\subsection{Planar trees}%

Operads are a tool for studying various types of algebraic structures that possess operations of several arities. More concretely, an operad $\O$ consists of a sequence of sets (or, more importantly to us, spaces/simplicial sets) $\O(n)$, $n \geq 0$ which behave as sets (spaces/simplicial sets) of $n$-ary operations. I.e., one should have 
\textit{composition product} maps (cf. \cite[Def. 1.1]{May72} or \cite[Def. 1.4]{GJ95}) %
\begin{equation}\label{COMPPROD EQ}%
\begin{tikzcd}[row sep =0em]
	\O(k) \times \O(n_1) \times \cdots \times \O(n_k) %
		\ar{r}{\circ}  &
	\O(n_1 + \cdots + n_k) \\
	(\varphi,\psi_1,\cdots,\psi_k) \ar[mapsto]{r} & \varphi(\psi_1,\cdots,\psi_k)%
\end{tikzcd}
\end{equation}%
and an identity $id \in \O(1)$ satisfying suitable associativity and unital conditions.%

A powerful tool for visualizing operadic compositions and their compatibilities is given by tree diagrams. For instance, the tree %
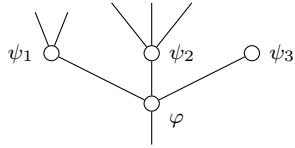
\[%
	\begin{tikzpicture}[grow=up, every node/.style = {font=\footnotesize},level distance = 1.9em]%
	\tikzstyle{level 2}=[sibling distance=3.75em]%
	\tikzstyle{level 3}=[sibling distance=1.5em]%
		\node {}%
			child{node [dummy,label=-10:$\varphi$] {}%
				child{node [dummy,label=right:$\psi_3$] {}}%
				child{node [dummy,label=right:$\psi_2$] {}%
					child{}%
					child{}%
					child{}%
				}%
				child{node [dummy,label=left:$\psi_1$] {}%
					child{node {} }%
					child{node {} }%
				}%
			};%
	\end{tikzpicture}%
\]%
encodes the composition of operations %
$\varphi \in \O(3)$, $\psi_1 \in \O(2)$, $\psi_2 \in \O(3)$ and $\psi_3 \in \O(0)$, %
and one has $\varphi(\psi_1,\psi_2,\psi_3) \in \O(5)$, with arity of the composite given by counting \textit{leaves} (i.e. the edges at the top of the tree, not capped by a circle).%

Alternatively, given the presence of the identity $id \in \O(1)$, one can instead define operads using so called \textit{partial composition products} \cite[Def. 1.16]{MSS02}%
\[%
\begin{tikzcd}[row sep = 0em]
	\O(k) \times \O(n) \ar{r}{\circ_i} &
	\O(n+k-1) \\
	(\varphi,\psi) \ar[mapsto]{r} & \varphi(id, \cdots, id,  \psi, id,\cdots, id)%
\end{tikzcd}
\]%
which are also readily visualized using trees. For example, %
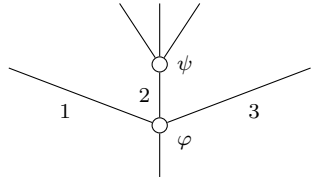
\begin{equation}\label{TREEWITHLABELS EQ}%
	\begin{tikzpicture}[auto,grow=up,every node/.style={font=\footnotesize},level distance=2.3em]%
	\tikzstyle{level 2}=[sibling distance=6em]%
	\tikzstyle{level 3}=[sibling distance=1.5em]%
		\node {}%
			child{node [dummy,label=-10:$\varphi$] {}%
				child{node {}%
				edge from parent node [swap] {$3$}}%
				child{node [dummy,label=right:$\psi$] {}%
					child{}%
					child{}%
					child{}%
				edge from parent node {$2$}}%
				child{node {}%
				edge from parent node {$1$}}%
			};%
	\end{tikzpicture}%
\end{equation}%
encodes the partial composition $\varphi \circ_2 \psi = \varphi(id, \psi, id)$.%

Heuristically, trees encoding iterations of $\circ$ are naturally tiered whereas trees encoding iterations of %
$\circ_i$ operations are not, as exemplified by the following.%
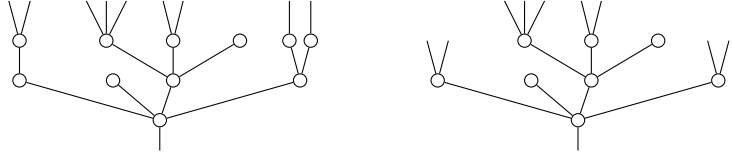
\[%
	\begin{tikzpicture}[grow = up, level distance = 1.5em,dummy/.style={circle,draw,inner sep=0pt,minimum size=1.75mm}]%
		\tikzstyle{level 2}=[sibling distance=3.5em]%
		\tikzstyle{level 3}=[sibling distance=2.5em]%
		\tikzstyle{level 4}=[sibling distance=0.75em]%
			\node at (0,0) {}%
				child{node [dummy] {}%
					child{node [dummy] {}%
						child[sibling distance=0.8em]{node [dummy] {}%
							child %
						}%
						child[sibling distance=0.8em]{node [dummy] {}%
							child%
						}%
					}%
					child[sibling distance=1em]{node [dummy] {}%
						child{node [dummy] {}}%
						child{node [dummy] {}%
							child %
							child %
						}%
						child{node [dummy] {}%
							child %
							child %
							child %
						}%
					}%
					child{node [dummy] {}}%
					child{node [dummy] {}%
						child{node [dummy]{}%
							child[sibling distance=0.8em] %
							child[sibling distance=0.8em] %
						}%
					}%
				};%
			\node at (5.5,0) {}%
				child{node [dummy] {}%
					child{node [dummy] {}%
						child[sibling distance=0.8em] %
						child[sibling distance=0.8em] %
					} %
					child[sibling distance=1em]{node [dummy] {}%
						child{node [dummy] {}}%
						child{node [dummy] {}%
							child %
							child %
						}%
						child{node [dummy] {}%
							child %
							child %
							child %
						}%
					}%
					child{node [dummy] {}}%
					child{node [dummy] {}%
						child[sibling distance=0.8em] %
						child[sibling distance=0.8em] %
					}%
				};%
\end{tikzpicture}%
\]%
In the leftmost tree, encoding an iterated composition of $\circ$, all leaves appear at the same height and the operations, encoded by \textit{nodes} (i.e. the circles) are naturally divided into levels. On the other hand, this fails for the rightmost tree, which encodes iterated compositions of $\circ_i$. Indeed, while the definition %
$\varphi \circ_i \psi = \varphi(id, \cdots, id, \psi, id,\cdots, id)$ would allow us to convert the rightmost tree into a tiered tree by inserting nodes labeled by $id$, there are multiple ways to do so (indeed, the leftmost tree represents one such possibility).%

In practice, the second type of trees seems to be the most convenient and we will henceforth work only with such trees.%

\subsection{Symmetric trees}%

The (planar) tree notation just described is suitable for working with so-called ``non-$\Sigma$ operads''. In many  applications, however, operads possess an additional piece of structure: each set (space/simplicial set) $\O(n)$ has a left action of the symmetric group $\Sigma_n$. Heuristically, the role of this action is to ``change the order of the inputs of an operation'': thinking of $\varphi \in \O(n)$ as an operation %
$x_1 \cdots, x_n \mapsto \varphi(x_1,\cdots,x_n)$ and letting $\sigma \in \Sigma_n$, then %
$\sigma \varphi \in \O(n)$ would correspond to the operation %
$x_1, \cdots, x_n \mapsto \varphi(x_{\sigma(1)},\cdots,x_{\sigma(n)})$.%

When representing compositions on a (symmetric) operad, it thus becomes convenient to think of the edges above a node, which represent the inputs for the operation labeling the node, as not having a fixed order. One immediate drawback of this perspective, however, is that drawing %
a planar representation of such a tree on paper necessarily requires choosing an (arbitrary) order for the input edges of every node. Therefore, it is possible for different planar representations to encode the exact same information. %
For example, the pictures %
	
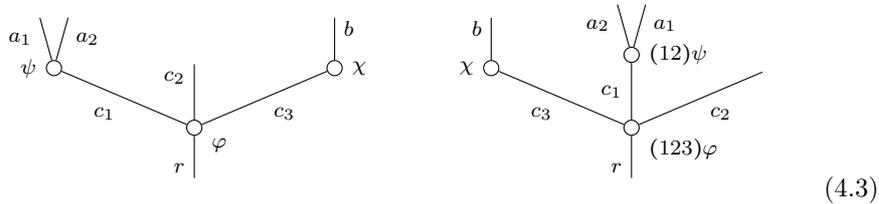
\begin{equation}\label{SYMTREES EQ}%
		\begin{tikzpicture}[auto,grow = up, level distance = 2.25em,every node/.style={font=\footnotesize}]%
		\tikzstyle{level 2}=[sibling distance=5.25em]%
		\tikzstyle{level 3}=[sibling distance=1.25em]%
		\tikzstyle{level 4}=[sibling distance=0.75em]%
			\node at (0,0){}%
				child{node [dummy,label=-10:$\varphi$] {}%
					child{ node [dummy,label=0:$\chi$] {}%
						child{node {}%
						edge from parent node [swap,near end] {$b$}}%
					edge from parent node [swap] {$c_3$}}%
					child[level distance = 2.75em]{node {}%
					edge from parent node [near end]  {$c_2$}}%
					child{ node [dummy,label=180:$\psi$] {}%
						child{node {}%
						edge from parent node [swap, very near end] {$a_2$}}%
						child{node {}%
						edge from parent node [very near end] {$a_1$}}%
					edge from parent node  {$c_1$}}%
				edge from parent node [near start] {$r$}};%
			\node at (5.75,0){}%
				child{node [dummy,label=-10:$(123)\varphi$] {}%
					child{node {}%
					edge from parent node [swap] {$c_2$}}%
					child[level distance = 2.75em]{ node [dummy,label=0:$(12)\psi$] {}%
						child[level distance = 2.25em]{node {}%
						edge from parent node [swap,very near end] {$a_1$}}%
						child[level distance = 2.25em]{node {}%
						edge from parent node [very near end] {$a_2$}}%
					edge from parent node {$c_1$}}%
					child{ node [dummy,label=180:$\chi$] {}%
						child{node {}%
						edge from parent node [near end] {$b$}}%
					edge from parent node {$c_3$}}%
				edge from parent node [near start] {$r$}};%
		\end{tikzpicture}%
	\end{equation}%
display two planar representations of the same tree that encode \textit{the same} composition data. To explain why, we first point out that $a_1, a_2, b, c_1, c_2, c_3$ are simply the \textit{names} of the edges of the tree (needed so as to distinguish different representations of the tree in the plane), by contrast with $\psi,\chi,(123)\varphi$, etc. which are operations in $\O$. %
Next, for a finite set $S$ of size $n$, denote %
$\O(S) = \left(\mathsf{Iso}(\{1,\cdots,n\},S)\times \O(n)\right)_{\Sigma_n}$, %
where the orbits $(-)_{\Sigma_n}$ are defined using the diagonal action, acting on the $\mathsf{Iso}(\{1,\cdots,n\},S)$ component by precomposition with the inverse. %
Note that though $\O(S)$ is of course isomorphic to $\O(n)$, the isomorphism is not canonical, depending on the choice of an isomorphism $\{1,\cdots,n\} \xrightarrow{\simeq} S$. With this convention, both trees in (\ref{SYMTREES EQ}) represent the same instance of a composition
\[
  \O(\{c_1,c_2,c_3\}) \times \O(\{a_1,a_2\}) \times \O(\{b\})%
    \xrightarrow{(-) \circ_{c_1} (-) \circ_{c_3} (-)}
	\O(\{a_1,a_2,c_2,b\}).
\]
The reason for the differing labels on the nodes of the trees is then that different planar representations correspond to different choices of isomorphisms $\{1,\cdots,n\} \xrightarrow{\simeq} S$. %
For example, the leftmost tree uses the identification %
	$\{1,2,3\} \simeq \{c_1,c_2,c_3\}$ %
while the rightmost tree uses %
	$\{1,2,3\} \simeq \{c_3,c_1,c_2\}$) %
so that, for example, the classes %
	$[(\{1,2,3\}\xrightarrow{\simeq}\{c_1,c_2,c_3\},\varphi)]$, %
	$[(\{1,2,3\} \xrightarrow{\simeq} \{c_3,c_1,c_2\},(123)\varphi)]$ %
are in fact the \textit{same} element of $\O({c_1,c_2,c_3})$.%

For reasons to become apparent when we discuss how to encode compositions in equivariant operads using equivariant trees, the fact that the labels in (\ref{SYMTREES EQ}) change depending on the planar representation is rather inconvenient. Since the source of the problem is the use of non canonical isomorphisms $\O(S) \simeq \O(|S|)$, an easy solution is to use labels in $\O(S)$ instead. Thus, denoting by %
	$[\psi] \in \O(\{a_1,a_2\})$, $[\chi] \in \O(\{b\})$, %
	$[\varphi] \in \O(\{c_1,c_2,c_3\})$ %
the classes of the operations in the leftmost tree in (\ref{SYMTREES EQ}), the same information could then be instead encoded by labeling \textit{both} trees with the common labels %
$[\varphi]$, $[\psi]$, $[\chi]$. %
In what follows, we refer to labels of the form $\varphi \in \O(n)$ as \textit{coordinate dependent} labels and to 
$[\varphi] \in \O(S)$ labels as \textit{coordinate free} labels. %

Additionally, an important feature of symmetric trees not present in  planar trees is that they have non trivial ``automorphism groups''. For example, the tree %
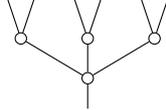
\[%
	\begin{tikzpicture}[grow = up, level distance = 1.5em,dummy/.style={circle,draw,inner sep=0pt,minimum size=1.5mm}]%
	\tikzstyle{level 2}=[sibling distance=2.5em]%
	\tikzstyle{level 3}=[sibling distance=1em]%
		\node{}%
			child{node [dummy] {}%
				child{node [dummy] {}%
					child %
					child %
				}%
				child{node [dummy] {}%
					child %
					child %
				}%
				child{node [dummy] {}%
					child %
					child %
				}%
			};%
	\end{tikzpicture}%
\]%
has automorphism group isomorphic to the wreath product
$\Sigma_3 \wr \Sigma_2 = \Sigma_3 \ltimes (\Sigma_2)^{\times 3}$.%

\begin{remark}\label{COLOROPPERS REM}
	An alternative and more rigorous perspective on (\ref{SYMTREES EQ}) is provided by \cite[\S 3]{MW07}, where it is explained that any tree $T$ has an associated  \textit{colored} operad $\Omega(T)$. Briefly, colored operads generalize operads much in the way that categories generalize monoids: each colored operad $\O$ has a collection of objects, morphism sets $\O(\underline{b};a)$ for an ordered tuple of source objects $\underline{b}=b_1,\cdots, b_n$ and target object $a$,
	units $id_a \in \O(a;a)$, compositions
	\[
	\O(b_1,\cdots,b_n;a) \times
	\O(\underline{c}_1;b_1) \times \cdots \times
	\O(\underline{c}_n;b_n)
	\xrightarrow{\circ}
	\O(\underline{c}_1,\cdots,\underline{c}_n;a),
	\]
and isomorphisms $\O(b_1,\cdots,b_n;a) \xrightarrow{\sigma_{\**}}\O(b_{\sigma^{-1}(1)},\cdots,b_{\sigma^{-1}(n)};a)$ for each $\sigma \in \Sigma_n$
satisfying natural associativity, unitality and symmetry conditions. Note that regular operads are then ``colored operads with a single color''.

$\Omega(T)$ is then the colored operad with objects the set of edges of $T$ and freely generated by morphisms associated to each node. Explicitly, 
for the tree in (\ref{SYMTREES EQ}) these generators are (unique) morphisms $a_1 a_2 \to c_1$, $b \to c_3$ and 
$c_1 c_2 c_3 \to d$
 if using the leftmost planar representation, or generators   
$a_2 a_1 \to c_1$, $b \to c_3$, $c_3 c_1 c_2 \to d$ if using the 
 rightmost planar representation (that the two descriptions coincide follows from the symmetry isomorphisms).

(\ref{SYMTREES EQ}) is then a diagrammatic representation of a morphism $\Omega(T) \to \O$, between the colored operad $\Omega(T)$ and the regular operad $\O$, with the node labels being the image of the associated generators of $\Omega(T)$.
\end{remark}

\subsection{Equivariant trees}\label{EQUIVTREESINTO SEC}%

Throughout the following we fix a finite group $G$ and the term operad will now refer to an operad $\O$ together with a left $G$-action compatible will all the structure. Notably, $\O(n)$ now has a $G \times \Sigma_n$-action, so that from an equivariant homotopy theory perspective it is  natural to consider fixed points
$\O(n)^{\Gamma}$ for $\Gamma \leq G \times \Sigma_n$.
On the other hand, since operad theory often focuses on $\Sigma$-cofibrant operads (i.e. such that $\O(n)$ is $\Sigma_n$-free), it is natural to focus attention on $\Gamma$ such that
$\Gamma \cap \Sigma_n = \**$, since for such $\O$ it is $\O(n)^{\Gamma}=\emptyset$ otherwise. The key identifications
\[
\Gamma \text{ such that } \Gamma \cap \Sigma_n = \**
	\leftrightarrow
\text{graph of } G\geq H \to \Sigma_n
	\leftrightarrow
\text{action of } H\leq G \text{ on } \{1,\cdots,n\}
\]
then hint at a deep connection between $G$-operads and 
$H$-sets that is at the core of Blumberg and Hill's work in 
\cite{BH15}. Briefly, to each $\Sigma$-cofibrant $G$-operad $\O$ they associate the family of those $\Gamma$ such that 
$\O(n)^{\Gamma} \neq \emptyset$ and, in turn, the family of the corresponding $H$-sets, $H\leq G$
\cite[Def. 4.5]{BH15}. They then show that such families satisfy a number of novel closure conditions
\cite[Lemmas 4.10, 4.11, 4.12, 4.15]{BH15}, and dub such a family an \textit{indexing system} \cite[Def. 3.22]{BH15}. Moreover, analyzing their proofs one sees that the key idea is that carefully chosen fixed point conditions on the source of the composition
(\ref{COMPPROD EQ}) induce fixed point conditions on its target (for an explicit example, see (\ref{FIXEDPOINTCOMP EQ}) below). 

The discovery of equivariant trees was the result of an attempt to encode the closure conditions of Blumberg and Hill diagrammatically, and we provide more details on how that works in \S \ref{INDEXSYS SEC}.
For now, however, we focus on examples.

As a first guess, one might attempt to define $G$-trees simply as symmetric trees together with a $G$-action (using the automorphisms mentioned in the previous section).
As it turns out, such ``trees with a $G$-action'' are only a part of what is required, though we will choose such trees as our first examples.

\begin{example}\label{Z4 EX} %
Let $G=\mathbb{Z}_{/4}$. The following are two \textit{equivalent} representations of a symmetric tree $T$ with a $G$-action. %
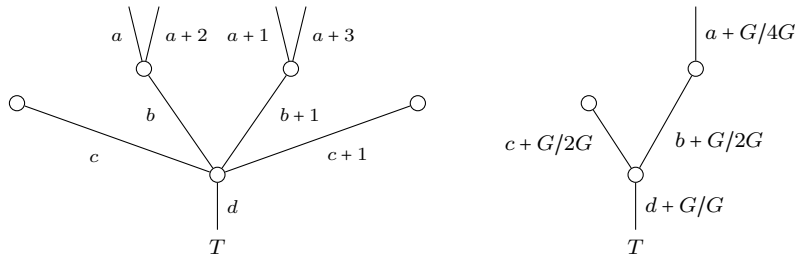
\[%
	\begin{tikzpicture}[auto,grow=up, every node/.style={font=\scriptsize},level distance = 2.7em]%
	\tikzstyle{level 2}=[sibling distance=5em]%
	\tikzstyle{level 3}=[sibling distance=1.3em]%
		\node at (0,0)[{font=\footnotesize}]{$T$}%
			child{node [dummy] {}%
				child{node [dummy]  {}%
				edge from parent node [swap] {$c+1$}}%
				child[sibling distance=5.5em,level distance = 4em]{node [dummy] {}%
					child[level distance = 2.7em]{node {} edge from parent node [swap,near end] {$a+3$}}%
					child[level distance = 2.7em]{node {} edge from parent node[near end] {$a+1$}}%
				edge from parent node [swap,near end] {$b+1$}}%
				child[sibling distance=5.5em,level distance = 4em]{node [dummy] {}%
					child[level distance = 2.7em]{node {} edge from parent node [swap,near end] {$a+2$}}%
					child[level distance = 2.7em]{node {} edge from parent node [near end] {$\phantom{0+}a$}}%
				edge from parent node [near end] {$b$}}%
				child{node [dummy] {}%
				edge from parent node {$c\phantom{||}$} }%
			edge from parent node [swap] {$d$}};%
	\begin{scope}[every node/.style={font=\footnotesize}]
		\node at (5.5,0){$T$}%
			child{node [dummy] {}%
				child[level distance = 4em,sibling distance=4.5em]{node [dummy] {}%
					child[level distance = 2.7em]{node {} edge from parent node [swap] {$a+G/4G$}}%
				edge from parent node [swap] {$b+G/2G$}}%
				child[sibling distance=3.5em]{node [dummy] {} %
				edge from parent node [near end] {$c+G/2G$}}%
			edge from parent node [swap] {$d+G/G$}};%
	\end{scope}%
	\end{tikzpicture}%
\]%
The leftmost representation, which we call the \textit{expanded representation}, is simply a planar representation of the corresponding equivariant tree, together with a naming convention for the edges that reflects the $G$-action. More concretely, $1 \in G$ acts on the tree by sending $a$ to $a+1$, $a+1$ to $a+2$, $b$ to $b+1$, etc (note that implicitly $b+2=b$, $c+2=c$, $d+1=d$). %

The rightmost representation, which we call the \textit{orbital representation}, is obtained from the expanded representation by %
``identifying edges which lie in the same $G$-orbit'', and then labeling the corresponding ``edge orbit'' by the $G$-set of the edges corresponding to it. %
\end{example}%

\begin{example}\label{D6TREE EX}%
Let $G = D_6 = \{e, r, r^2, r^3, r^4, r^5, s, sr, s r^2, s r^3, s r^4, s r^5\}$ denote the hexagonal dihedral group with generators $r,s$ such that %
$r^6=e$, $s^2=e$, $s r s = r^5$.%

Letting $H_1 \geq H_2 \geq H_3$ denote the subgroups %
 $H_1 = \langle r^2,s \rangle$, %
 $H_2 = \langle s \rangle$, %
 $H_3 = \{e\}$ %
one has the following representations of a tree $T$ with $G$-action.%
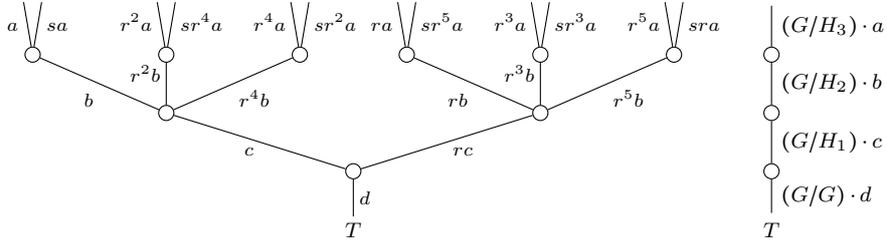
\[%
	\begin{tikzpicture}[auto,grow=up, level distance = 2.2em,every node/.style={font=\scriptsize,inner sep=2pt},dummy/.style    = {circle,draw,inner sep=0pt,minimum size=2mm}]%
	\tikzstyle{level 2}=[sibling distance=14em]%
	\tikzstyle{level 3}=[sibling distance=5em]%
	\tikzstyle{level 4}=[sibling distance=0.75em]%
		\node[{font=\footnotesize}]{$T$}%
			child{node [dummy] {}%
				child{node [dummy] {}%
					child{node [dummy] {}%
						child{node {}%
						edge from parent node [swap, very near end] {$s r a\phantom{1^2}$}}%
						child{node {}%
						edge from parent node [very near end] {$r^5 a$}}%
					edge from parent node [swap] {$r^5 b$}}%
					child{node [dummy] {}%
						child{node {}%
						edge from parent node [swap,very near end] {$s r^3 a$}}%
						child{node {}%
						edge from parent node [very near end] {$r^3 a$}}%
					edge from parent node [near end] {$r^3 b$}}%
					child{node [dummy] {}%
						child{node {}%
						edge from parent node [swap,very near end] {$s r^5 a$}}%
						child{node {}%
						edge from parent node [very near end] {$\phantom{1^2}r a$}}%
					edge from parent node {$\phantom{r^5}r b$}}%
				edge from parent node [swap] {$r c$}}%
				child{node [dummy] {}%
					child{node [dummy] {}%
						child{node {}%
						edge from parent node [swap,very near end] {$s r^2 a$}}%
						child{node {}%
						edge from parent node [very near end] {$r^4 a$}}%
					edge from parent node [swap] {$r^4 b$}}%
					child{node [dummy] {}%
						child{node {}%
						edge from parent node [swap,very near end] {$s r^4 a$}}%
						child{node {}%
						edge from parent node [very near end] {$r^2 a$}}%
					edge from parent node [near end] {$r^2 b$}}%
					child{node [dummy] {}%
						child{node {}%
						edge from parent node [swap,very near end] {$s a\phantom{^2}$}}%
						child{node {}%
						edge from parent node [very near end] {$\phantom{1^2}a$}}%
					edge from parent node {$\phantom{1^1}b$}}%
				edge from parent node {$c$}}%
			edge from parent node [swap] {$d$}};%
	\begin{scope}[auto,every node/.style={font=\footnotesize}]%
		\node at (5.5,0){$T$}%
			child{node [dummy] {}%
				child{node [dummy] {}%
					child{node [dummy] {}%
						child{node {}%
						edge from parent node [swap] {$(G/H_3) \cdot a$}}%
					edge from parent node [swap] {$(G/H_2) \cdot b$}}%
				edge from parent node [swap] {$(G/H_1) \cdot c$}}%
			edge from parent node [swap] {$(G/G) \cdot d$}};%
	\end{scope}%
	\end{tikzpicture}%
\]%
We note that it is implicit in the orbital representation that, for example, %
the assignment $b \mapsto c$ defines a $G$-set map %
$(G/H_2) \cdot b  \to (G/H_1) \cdot c$ (i.e. $H_1 \geq H_2$). %
\end{example}%

We can now ask what the analogue of the node labels in (\ref{TREEWITHLABELS EQ}) are for such $G$-trees. For example,
for $G=\mathbb{Z}_{/4}$
 consider the label $\varphi$ in the leftmost expanded representation of the equivariant \textit{corolla} (i.e. tree with a single node) below.
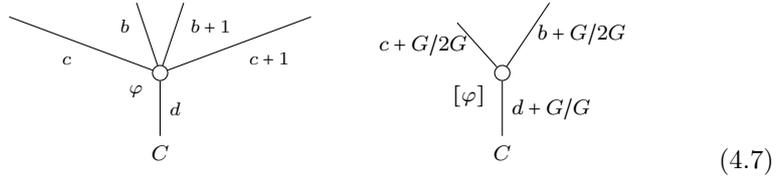
\begin{equation} \label{EQUIVCOROLLA EQ}%
	\begin{tikzpicture}[auto,grow=up, level distance = 3em,every node/.style={font=\scriptsize}]%
	\tikzstyle{level 2}=[sibling distance=4em]%
		\node at (0,0)[{font=\footnotesize}]{$C$}%
			child{node [dummy,label=200:$ \varphi $] {}%
				child[level distance = 2.25em]{node {}%
				edge from parent node [swap] {$c+1$}}%
				child[sibling distance=2em]{node {}%
				edge from parent node [very near end,swap] {$b+1$}}%
				child[sibling distance=2em]{node {}%
				edge from parent node [very near end] {$b$}}%
				child[level distance = 2.25em]{node {}%
				edge from parent node {$\phantom{1+}c$}}%
			edge from parent node [swap] {$d$}};%
	\begin{scope}[every node/.style={font=\footnotesize}]%
		\node at (4.5,0){$C$}%
			child{node [dummy,label=200:$\lbrack \varphi \rbrack$] {}%
				child{node {}%
				edge from parent node [right] {$b+G/2G$}}%
				child[level distance = 2.25em]{node {}%
				edge from parent node [left] {$c+G/2G$}}%
			edge from parent node [right] {$d + G/G$}};%
	\end{scope}%
	\end{tikzpicture}%
\end{equation}%
An immediate answer is provided by Remark
\ref{COLOROPPERS REM}: indeed, the corolla $C$ with a $G$-action will generate a colored operad $\Omega(C)$ with a $G$-action, so that  
(\ref{EQUIVCOROLLA EQ}) should encode a $G$-equivariant map 
$\Omega(C) \to \O$. Unpacking this observation, one should have 
$\varphi \in \O(4)$, but an additional equivariance condition is to be expected. To make this explicit, note first that there are left actions of both $G$ and $\Sigma_4$ on the set of \textit{all} morphisms of $\Omega(C)$ and that these actions commute, assembling to an action of $G\times \Sigma_4$. As concrete examples, $1 \in \mathbb{Z}_{/4}$ sends the morphism 
$c b (b+1) (c+1) \to d$
	to  
$(c+1) (b+1) b c \to d$
while $(124) \in \Sigma_4$ sends 
$c b (b+1) (c+1) \to d$
	to
$(c+1) c (b+1) b \to d$.
Further, one can readily check that the $G\times \Sigma_4$-isotropy of the morphism $c b (b+1) (c+1) \to d$ is precisely the subgroup 
$\Gamma_{\{c,b,b+1,c+1\}}$
 given by the graph of the homomorphism $G \to \Sigma_4$
 encoding the $G$-set $\{c,b,b+1,c+1\}$. And since $\varphi$ is the image of that morphism, we get the sought condition 
 $\varphi \in \O(4)^{\Gamma_{\{c,b,b+1,c+1\}}}$.

We now turn our attention to the orbital representation of $C$ on the right side of (\ref{EQUIVCOROLLA EQ}), which is often preferable both for conceptual reasons and compactness.
Writing $S = \{b,c,b+1,c+1\}$, our node label is now the coordinate free label 
$[\varphi] \in \O(S)$ (indeed, a label in $\O(4)$ can not be used since the orbital notation provides no ordering of $S$).
Further, $\O(S)$ now has \textit{two} \textit{commuting} $G$-actions, one induced by the structural $G$-action on $\O$ and one induced by the $G$-action on $S$. Referring to the combined \textit{diagonal} $G$-action as the \textit{canonical} $G$-action, the equivariance condition is now straightforward: it is simply 
$[\varphi] \in \O(S)^{G}$.

It seems helpful to make the equivalence between the previous two paragraphs explicit. The homomorphism $\rho \colon G \to \Sigma_S$ encoding $S$ induces a semi-direct product
$G \ltimes \Sigma_S$ with multiplication  
$(g,\sigma)(\bar{g},\bar{\sigma})=(g \bar{g},\sigma \rho(g)\bar{\sigma}\rho(g)^{-1})$.
$G \ltimes \Sigma_S$ naturally acts on $S$ with the action of $(g,\sigma)$ given by $\sigma \rho(g)$ and on $\O(S)$ with action given by $\sigma g$ (where $\sigma \in \Sigma_S$ acts via symmetries and $g\in G$ acts via the canonical action).
The isomorphism 
$\tau \colon\{c,b,b+1,c+1\} \xrightarrow{\simeq} \{1,2,3,4\}$
then induces an isomorphism 
$
\tau_{\**} \colon G \ltimes \Sigma_S \xrightarrow{\simeq} G \times \Sigma_4
$
via $\tau_{\**}(g,\sigma)=(g,\tau\sigma \rho(g)\tau^{-1})$.
That the previous two paragraphs are equivalent is then simply the observation that $\tau_{\**}$ sends the subgroup 
$G \leq G \ltimes \Sigma_S$ to 
$\Gamma_{\{c,b,b+1,c+1\}} \leq G \times \Sigma_4$, i.e. that the graph subgroup encodes the canonical action when in coordinate dependent notation.

\vskip 5pt

Now that we know that corollas with $G$-actions encode operations fixed by graphs of full homomorphisms $G\to \Sigma_n$, we turn to the question of how to encode operations fixed by graphs of \textit{partial} homomorphisms $G \geq H \to \Sigma_n$. A natural first guess might be that this role is played by corollas with a $H$-action. However, due to the lack of full $G$-actions this would not quite provide the necessary maps for the category $\Omega_G$ of $G$-trees that we introduce in \S \ref{EQUIVTREECAT SEC}.
The solution is both simple and surprising: one simply
``induces a tree with a $H$-action into a $G$-object''.
We start with an example where $H=\**$.
\begin{example}\label{FREECOR EX}%
Let $G = \mathbb{Z}_{/3}$. The equivariant corolla $C$ with orbital representation given on the right%
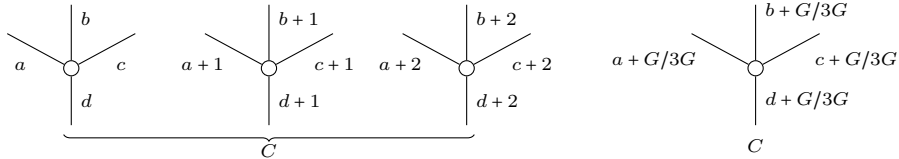
\[%
	\begin{tikzpicture}[auto,grow=up,every node/.style={font=\scriptsize},level distance = 2.5em]%
	\tikzstyle{level 2}=[sibling distance=2.75em]%
		\node at (0,0) {}%
			child{node [dummy] {}%
				child[level distance = 1.5em]{node {}%
				edge from parent node [swap] {$c\phantom{+1}$}}%
				child[level distance = 2.75em]{node {}%
				edge from parent node [swap, near end] {$b$}}%
				child[level distance = 1.5em]{node {}%
				edge from parent node {$\phantom{1+}a$}}%
			edge from parent node [swap] {$d$}};%
		\node at (2.6,0) {}%
			child{node [dummy] {}%
				child[level distance = 1.5em]{node {}%
				edge from parent node [swap] {$c+1$}}%
				child[level distance = 2.75em]{node {}%
				edge from parent node [swap, near end] {$b+1$}}%
				child[level distance = 1.5em]{node {}%
				edge from parent node {$a+1$}}%
			edge from parent node [swap] {$d+1$}};%
		\node at (5.2,0) {}%
			child{node [dummy] {}%
				child[level distance = 1.5em]{node {}%
				edge from parent node [swap] {$c+2$}}%
				child[level distance = 2.75em]{node {}%
				edge from parent node [swap, near end] {$b+2$}}%
				child[level distance = 1.5em]{node {}%
				edge from parent node {$a+2$}}%
			edge from parent node [swap] {$d+2$}};%
		\node at (9,0) {}%
			child{node [dummy] {}%
				child[level distance = 1.5em]{node {} %
				edge from parent node [near end,swap] {$c + G/3G$}}%
				child[level distance = 2.75em]{node {}%
				edge from parent node [very near end,swap] {$b + G/3G$}}%
				child[level distance = 1.5em]{node {}%
				edge from parent node [near end] {$a + G/3G$}}%
			edge from parent node [swap] {$d + G/3G$}};%
		\draw[decorate,decoration={brace,amplitude=2.5pt}] (5.3,0) -- (-0.1,0) node[midway]{$C$}; %
		\node at (9,-0.15) {$C$};
	\end{tikzpicture}%
\]%
has expanded representation given by the union of the three (non-equivariant) corollas on the left. For clarity, we stress that we refer to the three trees on the left together as a forming a single $\mathbb{Z}_{/3}$-tree.
The legitimacy of this nomenclature is born out of the role such $G$-trees play in the theory, though for now we simply point out that at least the orbital representation is a ``honest'' tree (for further discussion, see \S \ref{EQUIVTREECAT SEC}).

A map $\Omega(C) \to \O$ is then determined by the image of the morphism $abc \to d$, and hence by an arbitrary operation
$[\varphi] \in \O(\{a,b,c\}) \simeq \O(3)$, which determines operations 
$[\varphi]+1 \in \O(\{a+1,b+1,c+1\})$,
$[\varphi]+2 \in \O(\{a+2,b+2,c+2\})$.
\end{example}%

\begin{example}\label{D6SMALLER EX}%
Keeping $G = D_6$ and $H_1,H_2,H_3$ as in Example \ref{D6TREE EX}, removing the \textit{root} orbit (i.e bottom orbit) from the $G$-tree $T$ therein yields the $D_6$-tree
\begin{equation}\label{D6SMALLER EQ}%
	\begin{tikzpicture}[auto,grow=up, level distance = 2.2em,every node/.style={font=\scriptsize}]%
		\tikzstyle{level 2}=[sibling distance=4.75em]%
		\tikzstyle{level 3}=[sibling distance=0.75em]%
			\node at (5,0){}%
				child{node [dummy] {}%
					child{node [dummy] {}%
						child{node {}%
						edge from parent node [swap,very near end] {$s r a\phantom{1^2}$}}%
						child{node {}%
						edge from parent node [very near end] {$r^5 a$}}%
					edge from parent node [swap] {$r^5 b$}}%
					child{node [dummy] {}%
						child{node {}%
						edge from parent node [swap,very near end] {$s r^3 a$}}%
						child{node {}%
						edge from parent node [very near end] {$r^3 a$}}%
					edge from parent node [swap,near end] {$r^3 b$}}%
					child{node [dummy] {}%
						child{node {}%
						edge from parent node [swap,very near end] {$s r^5 a$}}%
						child{node {}%
						edge from parent node [very near end] {$\phantom{1^2}r a$}}%
					edge from parent node  {$r b$}}%
				edge from parent node [swap] {$r c$}};%
			\node at (0,0){}%
				child{node [dummy] {}%
					child{node [dummy] {}%
						child{node {}%
						edge from parent node [swap,very near end] {$s r^2 a$}}%
						child{node {}%
						edge from parent node [very near end] {$r^4 a$}}%
					edge from parent node [swap] {$r^4 b$}}%
					child{node [dummy] {}%
						child{node {}%
						edge from parent node [swap,very near end] {$s r^4 a$}}%
						child{node {}%
						edge from parent node [very near end] {$r^2 a$}}%
					edge from parent node [near end,swap] {$r^2 b$}}%
					child{node [dummy] {}%
						child{node {}%
						edge from parent node [swap,very near end] {$s a \phantom{1^2}$}}%
						child{node {}%
						edge from parent node [very near end] {$\phantom{1^2}a$}}%
					edge from parent node  {$\phantom{r^4}b$}}%
				edge from parent node [swap] {$c$}};%
		\begin{scope}[every node/.style={font=\footnotesize}]%
			\node at (8,0){}%
				child{node [dummy] {}%
					child{node [dummy] {}%
						child{node {}%
						edge from parent node [right] {$(G/H_3) \cdot a$}}%
					edge from parent node [right] {$(G/H_2) \cdot b$}}%
				edge from parent node [right] {$(G/H_1) \cdot c$}};%
		\end{scope}%
		\draw[decorate,decoration={brace,amplitude=2.5pt}] (5.1,0) -- (-0.1,0) node[midway]{$S$}; %
		\node at (8,-0.15) {$S$};
	\end{tikzpicture}%
\end{equation}%
\end{example}%

We end this introduction by illustrating the kind of  compositions that $G$-trees encode. Taking the $D_6$-tree $S$ from Example \ref{D6SMALLER EX}, a map $\Omega(S) \to \O$ leads to node labels 
\[
	\begin{tikzpicture}[auto,grow=up, level distance  = 2.2em,every node/.style={font=\footnotesize}]%
	\tikzstyle{level 2}=[sibling distance=4.25em]%
	\tikzstyle{level 3}=[sibling distance=0.75em]%
		\node at (4.75,0){}%
			child{node [dummy,label=190:$r \varphi$] {}%
				child{node [dummy,label=350:$r^5 \psi$] {}%
					child{node {}%
					edge from parent node {}}%
					child{node {}%
					edge from parent node {}}%
				edge from parent node {}}%
				child[right]{node [dummy,label=180:$r^3 \psi$] {}%
					child[right,sibling distance=0.1em]{node {}%
					edge from parent node {}}%
					child[left,sibling distance=0.1em]{node {}%
					edge from parent node {}}%
				edge from parent node {}}%
				child{node [dummy,label=190:$\phantom{1^2}r \psi$] {}%
					child{node {}%
					edge from parent node {}}%
					child{node {}%
					edge from parent node {}}%
				edge from parent node {}}%
			edge from parent node {}};%
		\node at (0,0){}%
			child{node [dummy,label=190:$ \varphi $] {}%
				child{node [dummy,label=350:$r^4 \psi$] {}%
					child{node {}%
					edge from parent node {}}%
					child{node {}%
					edge from parent node {}}%
				edge from parent node {}}%
				child[right]{node [dummy,label=180:$r^2 \psi$] {}%
					child[right,sibling distance=0.1em]{node {}%
					edge from parent node {}}%
					child[left,sibling distance=0.1em]{node {}%
					edge from parent node {}}%
				edge from parent node {}}%
				child{node [dummy,label=190:$\psi$] {}%
					child{node {}%
					edge from parent node {}}%
					child{node {}%
					edge from parent node [near end] {$a$}}%
				edge from parent node {$b$}}%
			edge from parent node [swap] {$c$}};%
		\node at (8.25,0){}%
			child{node [dummy,label=190:$\lbrack \varphi \rbrack$] {}%
				child{node [dummy,label=190:$\lbrack \psi \rbrack$] {}%
					child{node {}%
					edge from parent node [right] {$(G/H_3) \cdot a$}}%
				edge from parent node [right] {$(G/H_2) \cdot b$}}%
			edge from parent node [right] {$(G/H_1) \cdot c$}};%
	\end{tikzpicture}%
\]%
where $[\psi] \in \O\left((H_2/H_3)\cdot a\right)^{H_2}$ and 
$ [\varphi] \in \O\left((H_1/H_2)\cdot b\right)^{H_1}$. We note that in particular
$r[\varphi] \in \O\left(r(H_1/H_2)r^{-1}\cdot r b\right)^{r H_1 r^{-1}} = \O\left((H_1^r/H_2^r)\cdot rb\right)^{H_1^r}$ and likewise for $r \psi,r^2 \psi,\cdots, r^5 \psi$, 
so that we are adopting the convention that labels in the orbital notation are chosen according to the edge orbit generators $a,b,c$.

Further unpacking the map $\Omega(S) \to \O$, $S$ encodes the fact that the composition product %
\[\O(H_1/H_2) \times \prod_{[h] \in H_1/H_2} \O(H_2^h / H_3^h) \to \O(H_1/H_3)\]
restricts to %
\begin{equation}\label{FIXEDPOINTCOMP EQ}
\O(H_1/H_2)^{H_1} \times \left(\prod_{[h] \in H_1/H_2} \O(H_2^h / H_3^h)^{H_2^{h}}\right)^{H_1} \to \O(H_1/H_3)^{H_1}
\end{equation} %
or, using that %
$\left(\prod_{[h] \in H_1/H_2} \O(H_2^h / H_3^h)^{H_2^{h}}\right)^{H_1} \simeq \O(H_2 / H_3)^{H_2}$, simply %
\[\O(H_1/H_2)^{H_1} \times \O(H_2/H_3)^{H_2} \to \O(H_1/H_3)^{H_1}.\] %



\section{Categories of trees and forests}
\label{TREESFORESTS SEC}

In this section we introduce the several categories of trees, forests and presheaves we will be working with. We will make heavy use of the broad poset framework introduced by Weiss in \cite{We12}, which provides an algebraically flavored model for the category $\Omega$ of trees (cf. \cite{MW07}, \cite{MW08}, \cite{CM11}, \cite{HHM16}, among others). We will find this particularly convenient since, when using tree diagrams as in \S \ref{EQUIVTREESINTO SEC}, representative examples of equivariant trees are typically quite large.

\subsection{Broad posets and the category of trees}\label{BROADPOSETS SEC}

We start by recalling the key notions in \cite{We12} and  establishing some basic results.

Given a set $T$ we denote by $T^+$ the free abelian monoid generated by $T$. Elements of $T^+$ will be written in tuple notation, such as $\underline{e}=e_1 e_3 e_1 e_2 = e_1 e_1 e_2 e_3 \in T^+$ for $e_1, e_2, e_3 \in T$. 
We will also write $e_i \in \underline{e}$ whenever $e_i$ is a  ``letter'' appearing in $\underline{e}$,  
$\underline{f} \subset \underline{e}$ if $\underline{f}\underline{g} = \underline{e}$ for some $\underline{g}\in T^+$, and denote the ``empty tuple'' of $T^+$ by $\epsilon$.

\begin{definition}
A (commutative) broad poset structure \cite[Def. 3.2]{We12} on $T$ is a relation $\leq$ on $T^+,T$ (i.e. a subset of $T^+ \times T$) satisfying 
\begin{itemize}
 \item \textit{Reflexivity}: $e \leq e$ (for $e \in T$); 
 \item \textit{Antisymmetry}: if $e \leq f$ and $f \leq e$ then $e=f$ (for $e,f \in T$);
 \item \textit{Broad transitivity}: if $f_1 f_2 \cdots f_n = \underline{f} \leq e $ and $\underline{g}_i \leq f_i$, then $\underline{g}_1 \cdots \underline{g}_n \leq e$ (for $e,f_i \in T$, $\underline{f}, \underline{g}_i \in T^+$).
\end{itemize}
\end{definition}

\begin{remark}\label{PREBROAD REM}
Omitting antisymmetry yields the notion of
a \textit{pre-broad poset}.
\end{remark}

Since the main examples of broad posets are induced by constructions involving trees, we will refer to the elements of a broad poset as its \textit{edges}. 

\begin{definition}
A broad poset $P$ is called \textit{simple} if for any broad relation $e_1 \cdots e_n \leq e$ one has $e_i = e_j$ only if $i=j$.
\end{definition}

\begin{notation}
A broad poset structure $\leq$ on $T$ naturally induces the following preorder relations on $T$ and $T^+$: 
\begin{itemize} 
  \item for $f,e \in T$ we say that $f$ is a \textit{descendant} of $e$, written $f \leq_d e$, if there exists a broad relation $\underline{f} \leq e$ such that $f \in \underline{f}$;  
  \item for $\underline{f}, \underline{e} \in T^+$, we write $\underline{f} \leq \underline{e}$ if it is possible to write 
	$\underline{f} = \underline{f}_1 \cdots \underline{f}_k$, 
	$\underline{e} = e_1 \cdots e_k$ 
	such that $\underline{f}_i \leq e_i$ for $i=1,\cdots,k$.
\end{itemize}
\end{notation}

\begin{remark}
Generally, these preorders can be fairly counter-intuitive. For example, it is possible to have $a b \leq a$, or even both $a a \leq a$ and $a \leq a a$ simultaneously. 
The case of simple broad posets, however, is much simpler.
\end{remark}

\begin{proposition}\label{SIMPLEBROAD PROP}
Let $T$ be a simple broad poset. Then $\leq_d$ (resp. $\leq$) is an order relation on $T$ (resp. on $T^+$). Further, if $f_1 \cdots f_k \leq e$ then the $f_i$ are $\leq_d$-incomparable 
(in particular, $e \underline{f} \leq e$ only if $\underline{f}=\epsilon$).
\end{proposition}

\begin{proof}
	The ``further'' part is immediate: if two $f_i$ were $\leq_d$-comparable then broad transitivity would produce a non simple broad relation. 

	To see that $\leq_d$ satisfies antisymmetry, note that if $e'\underline{f} \leq e$ and $e \underline{g}\leq e'$ then 
	$e \underline{g} \underline{f}\leq e$ so that it must be 
	$\underline{g}=\underline{f}=\epsilon$ and the antisymmetry of $\leq$ on $T$ implies $e=e'$.

	Finally, we show antisymmetry of $\leq$ on $T^+$ by induction on the size of the tuple $\underline{e}$ in a pair of relations 
	$\underline{f}\leq \underline{e}$ and $\underline{e} \leq \underline{f}$. The $\underline{e}=\epsilon$ case is immediate. Otherwise let $e\in \underline{e}$ be $\leq_d$-maximal and choose $e\underline{g} \subset \underline{e}$, $f\in \underline{f}$ such that $e\underline{g} \leq f$ and choose $\underline{h}f \subset \underline{f}$ and $e'\in \underline{e}$ such that $\underline{h}f \leq e'$. Then $e \leq_d f \leq_d e'$ and by $\leq_d$-maximality of $e$ it must be $e=f=e'$ and hence,
	by the ``further'' claim, also $\underline{g}=\epsilon=
	\underline{h}$. And since this must hold regardless of how $f,e',\underline{g},\underline{h}$ are chosen, one concludes that, writing 
	$\underline{e}=e\underline{e}'$, $\underline{f}=e\underline{f}'$ it must in fact also be $\underline{e}'\leq \underline{f}'$ and 
	$\underline{f}' \leq \underline{e}'$, so that the induction hypothesis applies. 
\end{proof}

\begin{definition}
An edge $e \in T$ is called 
\begin{itemize}
 \item a \textit{leaf} if there are no $\underline{f} \in T^+$ 
  such that $\underline{f} < e$ 
	(i.e. $\underline{f} \leq e$ and $\underline{f} \neq e$);
 \item a \textit{node} if there is a non empty maximum $\underline{f}\neq \epsilon$ such that $\underline{f} < e$;
 \item a \textit{stump} if $\underline{f} = \epsilon$ is the maximum (in fact, only) $\underline{f}$ such that $\underline{f}<e$.
\end{itemize}
Further, in either the node or stump case the maximum such $\underline{f}$ is denoted $e^{\uparrow}$.
\end{definition}

\begin{remark}
While it is customary to regard stumps simply as a type of node, we will find it convenient, in lieu of Proposition \ref{CLOSEDSUBTREE PROP} and Lemma \ref{BROADREL STUMPS}, to separate the two cases.
\end{remark}

The following definition is the key purpose of \cite{We12}.

\begin{definition}\label{DENRORD DEF}
 A \textit{dendroidally ordered set} is a finite simple broad poset $T$ satisfying the following additional conditions
\begin{itemize}
 \item \textit{Nodal}: each edge $e \in T$ is either a leaf, a node or a stump;
 \item \textit{Root}: there is a maximum $r_T \in T$ for $\leq_d$, called the \textit{root} of $T$.
\end{itemize}
\end{definition}

Weiss proves in \cite{We12} that the category of dendroidally ordered sets (together with the obvious notion of monotonous function) is equivalent to the category $\Omega$ of trees (cf. \cite{MW07}, \cite{MW08}, \cite{CM11}, \cite{HHM16}, etc). As such, we will henceforth refer to dendroidally ordered sets simply as \textit{trees} and use them as our model for $\Omega$.

\begin{example}\label{BROADTREE EX}
The tree diagram
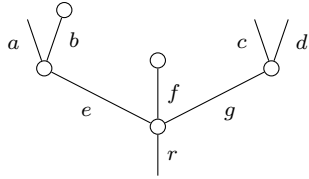
\[
	\begin{tikzpicture}[auto,grow=up, level distance  = 2.2em,every node/.style={font=\footnotesize}]%
	\tikzstyle{level 2}=[sibling distance=4.25em]%
	\tikzstyle{level 3}=[sibling distance=1.5em]%
		\node at (0,0){}%
			child{node [dummy] {}%
				child{node [dummy] {}%
					child{node {}%
					edge from parent node [swap, very near end] {$d$}}%
					child{node {}%
					edge from parent node [very near end] {$\phantom{1}c$}}%
				edge from parent node [swap] {$g$}}%
				child[level distance = 2.5em]{node [dummy] {}%
				edge from parent node [swap] {$f$}}%
				child{node [dummy] {}%
					child{node[dummy] {}%
					edge from parent node [very near end,swap] {$b$}}%
					child{node {}%
					edge from parent node [very near end] {$\phantom{b}a$}}%
				edge from parent node {$e$}}%
			edge from parent node [swap] {$r$}};%
	\end{tikzpicture}%
\]%
represents a broad poset structure on $\{a,b,c,d,e,f,g,r\}$. The nodes represent generating broad relations $\epsilon \leq b$, $ab\leq e$, $\epsilon \leq f$, $cd \leq g$, $ efg\leq r$ with the other broad relations, such as $afcd \leq r$, obtained by ``composition'' (i.e. using broad transitivity). Note that, alternatively, one can also write 
$b^{\uparrow}=\epsilon$, $e^{\uparrow}=ab$, $f^{\uparrow}=\epsilon$, $g^{\uparrow}= cd$, $r^{\uparrow} = efg$
to denote the generating broad relations.
\end{example}

We will make use of the following basic results. 

\begin{proposition}\label{MAPOUTTREE PROP}
 Let $T$ be a tree and $A$ any broad poset.
 A set map $\varphi \colon T \to A$ is a broad poset map if and only if $\varphi(e^{\uparrow}) \leq \varphi(e)$ for each node/stump $e \in T$.
 
 In particular, the broad relations of $T$ are generated by the 
 $e^{\uparrow} \leq e$ relations.
\end{proposition} 

\begin{proof}
Since for any non-identity relation $\underline{f} < e$ one has $\underline{f} \leq e^{\uparrow} < e$ one can write $e^{\uparrow} = e_1 \cdots e_k$, $\underline{f} = \underline{f}_1 \cdots \underline{f}_k$ so that $\underline{f}_i \leq e_i$ ($k=0$ is allowed, in which case $e^{\uparrow}=\underline{f}=\epsilon$), so the result follows by $\leq_d$ induction on $e$. 

The ``in particular'' claim follows from the identity map $id_T\colon T \to \tilde{T}$, where $\tilde{T}$ has the broad poset structure generated by the $e^{\uparrow} \leq e$.
\end{proof}

\begin{lemma}\label{INEQMAXIFF NOPEN LEM}
Let $T$ be a tree. For any $e \in T$ there exists a minimum 
$e^{\lambda} \in T^+$ such that $e^{\lambda} \leq e$. In fact, $e^{\lambda}=l_1 \cdots l_k$ consists of those leaves $l_i$ such that $l_i \leq_d e$.

Further, a broad relation 
$\underline{f}=f_1 \cdots f_n \leq e$ 
holds if and only if $f_i \leq_d e$, the $f_i$ are $\leq_d$-incomparable and $f_1^{\lambda} \cdots f_n^{\lambda} = e^{\lambda}$.
\end{lemma}

\begin{proof}
The proof is by $\leq_d$ induction on $e$. The leaf case is obvious. Otherwise, let 
$\underline{f} \leq e^{\uparrow} < e$
be any non identity relation and write 
$e^{\uparrow} = e_1 \cdots e_k$ and 
$\underline{f} = \underline{f}_1 \cdots \underline{f}_k$ so that $\underline{f}_i \leq e_i$ ($k=0$ is allowed). 
By induction, $e_i^{\lambda} \leq \underline{f}_i \leq e_i$ where $e_i^{\lambda}$ consists of the leaves $l$ such that $l \leq_d e_i$ and hence indeed
$e^{\lambda} = e_1^{\lambda} \cdots e_k^{\lambda} \leq \underline{f}$.

Only the ``if'' half of the ``further'' statement needs proof.
We use the same induction argument: incomparability yields 
$e \not\in \underline{f}$ provided $\underline{f} \neq e$
and, writing $e^{\uparrow}=e_1 \cdots e_k$ and
$\underline{f} = \underline{f}_1 \cdots \underline{f}_k$ 
so that $s \in \underline{f}_i$ if and only if $s \leq_d e_i$, the induction hypothesis applies.
\end{proof}

\begin{example}
In Example \ref{BROADTREE EX} $e^{\lambda}= a$, $g^{\lambda}=cd$, $f^{\lambda}=\epsilon$, $r^{\lambda}=acd$.
\end{example}

We now discuss a key operation on trees: grafting. We recall that $\eta$ denotes the tree with a single edge, also denoted $\eta$, and only the identity relation $\eta \leq \eta$.

\begin{definition}
Let $T, U \in \Omega$ be trees and let $v$ denote both a leaf $v \colon \eta \to T$  and the root $v \colon \eta \to U$. The \textit{grafting} 
$T \amalg_v U$ is the pushout (of pre-broad posets) 
\[
\begin{tikzcd}[row sep=1.2em]
\eta \ar{r}{v} \ar[swap]{d}{v} & U \ar{d} \\
T  \ar{r} & T \amalg_v U
\end{tikzcd}
\]
\end{definition}

\begin{proposition}
$T \amalg_v U$ is a tree.
\end{proposition}

\begin{proof}
	The underlying set of $T \amalg_v U$ is the underlying  coproduct and the broad relations are easily seen to come in three types:
	\begin{inparaenum}
	\item[(i)] $\underline{t} \leq t$ a relation in $T$;
	\item[(ii)] $\underline{u} \leq u$ a relation in $U$;
	\item[(iii)] $\underline{t}\underline{u} \leq t$ whenever 
	both $\underline{t} v \leq t$ in $T$ and
	$ \underline{u} \leq v$ in $U$.
	\end{inparaenum}
	The antisymmetry, simple, nodal and root conditions are straightforward.
\end{proof}

\begin{remark}\label{GENGRAFTDEF REM}
More generally, letting $\underline{v}=v_1 \cdots v_n$ denote both a tuple of leaves of $T$ and the roots 
$v_i \in U_i$
we similarly define a grafted tree 
$T \amalg_{\underline{v}}(U_1 \amalg \cdots \amalg U_n)$.
Explicitly, its broad relations have the form
\begin{inparaenum}
	\item[(i)] $\underline{t} \leq t$ a relation in $T$;
	\item[(ii)] $\underline{u} \leq u$ a relation in some $U_i$;
	\item[(iii)] $\underline{t}\underline{u}_{k_1}\cdots \underline{u}_{k_p}  \leq t$ whenever $\{k_i\} \subset \{1,\cdots,n\}$ and both $\underline{t}v_{k_1}\cdots v_{k_p} \leq t$ in $T$ and
	$\underline{u}_{k_i} \leq v_{k_i}$ in $U_i$.
\end{inparaenum}
\end{remark}

\begin{example} A tree representation of the grafting procedure follows.
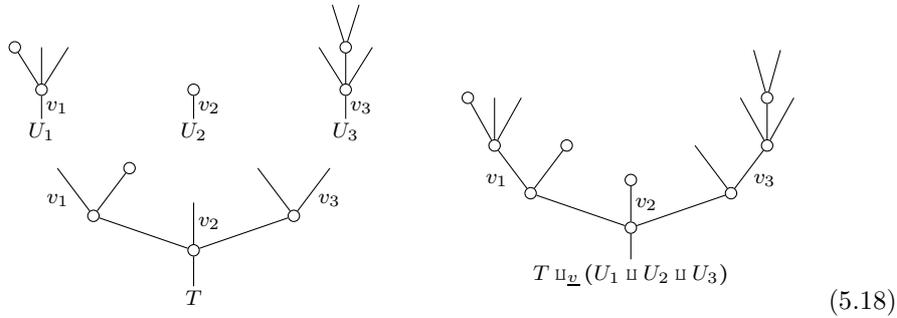
\begin{equation}\label{GRAFT EQ}
	\begin{tikzpicture}[auto,grow=up, level distance  = 1.8em,every node/.style={font=\footnotesize,inner sep=1.5pt},
	dummy/.style = {circle,draw,inner sep=0pt,minimum size=1.5mm}]%
	\tikzstyle{level 2}=[sibling distance=3.75em]%
	\tikzstyle{level 3}=[sibling distance=2.7em]%
	\tikzstyle{level 4}=[sibling distance=1em]%
		\node at (5.75,0.3){$T \amalg_{\underline{v}}(U_1 \amalg U_2 \amalg U_3)$}%
			child{node [dummy] {}%
				child[level distance = 1.3em]{node [dummy] {}%
					child[level distance = 1.8em]{node [dummy] {}%
						child
						child{node [dummy] {}%
							child
							child
						}
						child
					edge from parent node [swap] {$v_3$}}
					child[level distance = 1.8em]
				}
				child{node [dummy] {}%
				edge from parent node [swap] {$v_2$}}
				child[level distance = 1.3em]{node [dummy] {}%
					child[level distance = 1.8em]{node [dummy] {}%
					}
					child[level distance = 1.8em]{node [dummy] {}%
						child
						child
						child{node [dummy] {}%
						}			
					edge from parent node{$v_1$}}
				}
			edge from parent};%
		\node at (0,0){$T$}%
			child{node [dummy] {}%
				child[level distance = 1.3em]{node [dummy] {}%
					child[level distance = 1.8em]{
					edge from parent node [swap] {$v_3$}}
					child[level distance = 1.8em]
				}
				child{
				edge from parent node [swap] {$v_2$}}
				child[level distance = 1.3em]{node [dummy] {}%
					child[level distance = 1.8em]{node [dummy] {}%
					}
					child[level distance = 1.8em]{
					edge from parent node {$v_1$}}
				}
			edge from parent};%
	\begin{scope}[level distance  = 1.6em]
	\tikzstyle{level 2}=[sibling distance=1em]%
	\tikzstyle{level 3}=[sibling distance=1em]%
		\node at (-2,2.2) {$U_1$}
			child{node [dummy] {}%
				child
				child
				child{node [dummy] {}%
				}
			edge from parent node [swap]{$v_1$}};
		\node at (0,2.2) {$U_2$}
			child{node [dummy] {}%
			edge from parent node [swap]{$v_2$}};
		\node at (2,2.2) {$U_3$}
			child{node [dummy] {}%
				child
				child{node [dummy] {}%
					child
					child
				}
				child
			edge from parent node [swap]{$v_3$}};
	\end{scope}
	\end{tikzpicture}%
\end{equation}
\end{example}

We will also find it useful to be able to reverse the grafting procedure.

For any tree $T$ and edge $e\in T$ we will let $T^{\leq e}$ 
denote the pre-broad poset with set $\{f \in T | f\leq_d e\}$ and the generating relations $f^{\uparrow} \leq f$ for $f\leq_d e$.

Similarly, for a $\leq_d$-incomparable tuple $\underline{e}=e_1 \cdots e_n$
we will let $T_{\nless \underline{e}}$ denote the pre-broad poset with set 
$\{f \in T | \forall_i f \nless_d e_i\}$ and the generating relations $f^{\uparrow} \leq f$ such that $\forall_i f\nless_d e_i$ and $f\neq e_i$. Note that by incomparability it is $e_i\in T_{\nless \underline{e}}$.

\begin{proposition}\label{UNGRAFT PROP}
$T^{\leq e}$ and $T_{\nless \underline{e}}$ are trees. Hence, for any $\leq_d$-incomparable tuple $\underline{e}=e_1 \cdots e_n$, one has
\begin{equation}\label{REVERSEGRAFT EQ}
T \simeq T_{\nless \underline{e}} \amalg_{\underline{e}} 
(T^{\leq e_1} \amalg \cdots \amalg T^{\leq e_n}). 
\end{equation}
\end{proposition}

\begin{proof}
In both cases antisymmetry and simplicity are inherited from $T$. Further, since any string $f_1 \leq_d f_2 \leq_d \cdots \leq_d f_n$ in $T$ (note that we can assume the $\leq_d$ are induced by generating relations) where $f_1,f_n \in T^{\leq e}$ (resp. $f_1,f_n \in T_{\nless \underline{e}}$) is a string in 
$ T^{\leq e}$ (resp. $T_{\nless \underline{e}}$), the nodal and root conditions also follow.

(\ref{REVERSEGRAFT EQ}) follows from the ``in particular'' claim 
in Proposition \ref{MAPOUTTREE PROP}.
\end{proof}

\begin{example}
Denoting by $V$ the rightmost tree in (\ref{GRAFT EQ}), one has
$U_i = V^{\leq v_i}$, $T = V_{\nless \underline{v}}$.
We note here a pictorial mnemonic for our index/exponent notation: $V^{\leq v_i}$ denotes an ``upper subtree'' of $V$ while $V_{\nless \underline{v}}$ denotes a ``lower subtree''.
\end{example}

\begin{remark}\label{OUTERFACELEAVES REM}
 The leaves of $T_{\nless \underline{e}}$ consist of the edges  $e_i \in \underline{e}$ together with the leaves of $T$ not in $\underline{e}^{\lambda}$, i.e., those leaves 
$\leq_d$-incomparable with the $e_i \in \underline{e}$.
\end{remark}

\begin{corollary}\label{TALLOUTERFACT COR}
Let $U \xrightarrow{\varphi} T$ be a map in $\Omega$ and let  $r$, $\underline{l}$ be the root and leaf tuple of $U$.
Then $\varphi$ naturally factors as
$U \to T^{\leq \varphi(r)}_{\nless \varphi(\underline{l})}
\hookrightarrow T$.
\end{corollary}

\begin{proof}
By the ``in particular'' claim in Proposition \ref{MAPOUTTREE PROP} it suffices to check that any relation 
$\underline{e}\leq^T e$ in $T$
is a relation in $T^{\leq e}_{\nless \underline{e}}$. Since the root of $T^{\leq e}_{\nless \underline{e}}$ is $e$ and the leaf tuple is $\underline{e}$, this follows by Lemma \ref{INEQMAXIFF NOPEN LEM}.
\end{proof}

\begin{lemma}\label{INCOMPEDGES LEM}
Let $T$ be a tree and $\underline{e}$ a tuple of 
$\leq_d$-incomparable edges of $T$. 
Then, letting $r$ be the root of $T$, there exists a broad relation of the form 
$\underline{e} \underline{f} \leq r$.
\end{lemma}

\begin{proof}
The proof is by induction on the sum of the distances (measured in $\leq_d$ inequality chains) from the edges in $\underline{e}$ to the root. Clearly $r \in \underline{e}$ only if $r=\underline{e}$.
Otherwise one can write $r^{\uparrow} = r_1\cdots r_k$, 
$\underline{e} = \underline{e}_1 \cdots \underline{e}_k$ so that 
$s \in \underline{e}_i$ if and only if $s \leq_d r_i$.
The induction hypothesis applies to $\underline{e}_i$ and the subtrees 
$T^{\leq r_i}$.
\end{proof}

Since the converse of the previous lemma holds by Proposition \ref{SIMPLEBROAD PROP}, we obtain the following.

\begin{corollary}\label{INCOMPPLANARAX COR}
If $e,f\in T$ are $\leq_d$-incomparable and $e'\leq_d e$ then 
$e',f$ are $\leq_d$-incomparable.
\end{corollary}

\subsection{Categories of forests}\label{CATFOR SEC}

We will also need to discuss forests, i.e. formal coproducts of trees.
We start by generalizing Definition \ref{DENRORD DEF}.

\begin{definition}\label{FORESTORD DEF}
 A \textit{forestially ordered set} is a finite simple broad poset $F$ satisfying the following additional conditions:
\begin{itemize}
 \item \textit{Nodal}: each edge $e \in F$ is either a leaf, a node or a stump;
 \item \textit{Root tuple}: for each edge $e$ there is a unique $\leq_d$-maximal element $r \in T$ such that $e \leq_d r$.
 Further, we denote by $\underline{r}_F$ the tuple of $\leq_d$-maximal elements of $F$ and refer to it as the \textit{root tuple}.
\end{itemize}
\end{definition}

The last condition guarantees that any forestially ordered set decomposes as a disjoint union of dendroidally ordered sets, i.e. trees. We shall hence refer to these simply as \textit{forests} and denote by $\Phi$ the category formed by them.

\begin{definition}
	A map of forests $F \xrightarrow{\varphi} F'$ is called
	\begin{itemize}
		\item \textit{wide} if $\varphi(\underline{r}_{F}) \leq \underline{r}_{F'}$;
		\item \textit{independent} if there exists a tuple $\underline{e} \in F^+$ such that $\varphi(\underline{r}_{F})\underline{e} \leq \underline{r}_{F'}$.
	\end{itemize}
	The subcategory of forests and independent (resp. wide) maps is denoted $\Phi_w$ (resp. $\Phi_i$).
Note that there are inclusions $\Phi_w \hookrightarrow \Phi_i \hookrightarrow \Phi$.
\end{definition}

\begin{remark}
	The category $\Phi_i$ nearly coincides with the category of forests discussed in \cite[\S 3.1]{HHM16}, the only difference being that here we include the empty forest $\emptyset$.
\end{remark}

\begin{example}\label{EXAMPLEFORESTS EX}
Consider the following tree $T$ and subtrees (all labels are on edges).
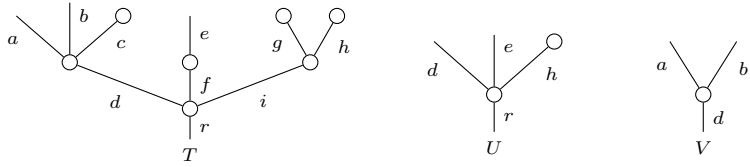
\[%
	\begin{tikzpicture}[auto,grow=up, every node/.style={font=\scriptsize},level distance = 1.75em]%
	\begin{scope}%
	\tikzstyle{level 2}=[sibling distance=4.5em]%
	\tikzstyle{level 3}=[sibling distance=2em]%
		\node at(0,0) {$T$}%
			child{node [dummy] {}
				child{node [dummy] {}
					child{node [dummy] {}
					edge from parent node[swap,near end]{$h$}}
					child{node [dummy] {}
					edge from parent node[near end]{$g$}}
				edge from parent node[swap]{$i$}}
				child{node [dummy] {}
					child{edge from parent node [swap]{$e$}}
				edge from parent node [swap]{$f$}}
				child{node [dummy] {}
					child{node [dummy] {}
					edge from parent node[near end,swap]{$c$}}
					child[level distance = 2.25em]{edge from parent node[swap,near end]{$b$}}
					child{edge from parent node [near end]{$a$}}
				edge from parent node{$d$}}
			edge from parent node[swap]{$r$}};
	\end{scope}
	\begin{scope}
	\tikzstyle{level 2}=[sibling distance=2.25em]%
			\node at(4,0.1) {$U$}%
			child[level distance = 2em]{node [dummy] {}
				child{node [dummy] {}
				edge from parent node[near end,swap]{$h$}}
				child[level distance = 2.25em]{
				edge from parent node[near end,swap]{$e$}}
				child{
				edge from parent node[near end]{$d$}}
			edge from parent node[swap]{$r$}};
	\end{scope}%
	\begin{scope}
	\tikzstyle{level 2}=[sibling distance=2.5em]%
			\node at(6.75,0.1) {$V$}%
			child[level distance = 2em]{node [dummy] {}
				child{
				edge from parent node[near end,swap]{$b$}}
				child{
				edge from parent node[near end]{$a$}}
			edge from parent node[swap]{$d$}};
	\end{scope}%
	\end{tikzpicture}
\]
Further denoting by $\eta_a,\eta_b,\cdots$ the single edge subtrees corresponding to each edge, some examples of wide morphisms are given by the inclusions
\[
 U \hookrightarrow T,
 \quad V \amalg \eta_f \hookrightarrow T,
 \quad V \amalg \eta_e \hookrightarrow T
\]
\[
 \quad V \amalg \eta_f \amalg \eta_i \hookrightarrow T,
 \quad V \amalg \eta_f \amalg \eta_g \hookrightarrow T,
 \quad \eta_a \amalg \eta_b \amalg \eta_f \hookrightarrow T,
\]
some examples of non-wide independent maps are given by
\[
	V \amalg \eta_g \hookrightarrow T,
	\quad \eta_f \amalg \eta_i \hookrightarrow T,
	\quad \eta_g \amalg \eta_h \hookrightarrow T,
	\quad \emptyset \hookrightarrow T
\]
and examples of non-independent maps are given by
\[
	V \amalg \eta_r \to T,
	\quad V \amalg \eta_f \amalg \eta_i \amalg \eta_g \to T,
	\quad \eta_i \amalg \eta_i \to T.
\]
\end{example}

The following is a useful forestial strengthening of Proposition \ref{SIMPLEBROAD PROP}, which follows by combining that result with Lemma \ref{INCOMPEDGES LEM}.

\begin{proposition}\label{INDECOMPUNIQUE PROP}
Let $F$ be a forest. Then,  
if $\underline{f} \leq \underline{e} = e_1 \cdots e_k$ with $\leq_d$-incompa\-rable $e_i$, the decomposition 
$\underline{f} = \underline{f}_1 \cdots \underline{f}_k$
with $\underline{f}_i \leq e_i$ is unique. In fact, for 
$s \in \underline{f}$ one has $s \in \underline{f}_i$ iff $s \leq_d e_i$.
\end{proposition}

\begin{proof}
If $s\leq_d e_i$, $s\leq_d e_j$ with $i\neq j$ the unique root $r$ such that $s\leq_d r$ must also satisfy $e_i\leq_d r$, $e_j\leq_d r$ and Lemma \ref{INCOMPEDGES LEM} can be used to produce a non simple relation. Thus such $e_i$ is unique and the result follows. 
\end{proof}

We now turn to describing the degeneracy-face decomposition of maps in the broad poset setting, which we obtain as 
Proposition \ref{FACT PROP} below.

\begin{definition}\label{FACEDEG DEF}
A map $F \xrightarrow{\varphi} F'$ in $\Phi_i$ is called 
	\begin{itemize}
		\item a \textit{face map} if the underlying set map is injective;
		\item a \textit{degeneracy} if the underlying set map is surjective and for each leaf $l \in F$ it is $\epsilon \nleq \varphi(l)$.
	\end{itemize}
\end{definition}

\begin{lemma}\label{LEQDCOMP LEM}
For any map $F \xrightarrow{\varphi} F'$ in $\Phi_i$, $e,\bar{e}$ are $\leq_d$-comparable iff $\varphi(e),\varphi(\bar{e})$ are $\leq_d$-comparable.

Further, if $\varphi(e)=\varphi(\bar{e})$ then $\bar{e}\underline{f}\leq e$ can hold only if $\underline{f}=\epsilon$ and thus either $e \leq \bar{e}$ or $\bar{e} \leq e$.
\end{lemma}

\begin{proof}
If $e,\bar{e}$ are not $\leq_d$-comparable, Lemma \ref{INCOMPEDGES LEM} ensures that there exists $\underline{f}$ such that $e\bar{e} \underline{f}\leq \underline{r}_F$. Thus, by definition of $\Phi_i$, there exists $\underline{g}$ such that $\varphi(e)\varphi(\bar{e})\varphi(\underline{f})\underline{g} \leq \varphi(\underline{r}_F)\underline{g} \leq \underline{r}_{F'}$, hence $\varphi(e)$, $\varphi(\bar{e})$ are $\leq_d$-incomparable.

The ``further'' claim is immediate.
\end{proof}

\begin{lemma}\label{FACEDEGISO LEM}
	If $F \xrightarrow{\varphi} F'$ is both a face map and a degeneracy then $\varphi$ is an isomorphism.
\end{lemma}

\begin{proof}
	Bijectiveness allows us to assume that the underlying sets are the same and it thus follows from Lemma \ref{LEQDCOMP LEM} that the relations $\leq_d^F$ and $\leq_d^{F'}$ coincide.
Hence, both forest structures have the same roots (i.e. $\leq_d$-maximal edges) and one needs only show that broad relations coincide in each of the constituent trees. But noting that a leaf $l$ is precisely a $\leq_d$-minimal edge such that $\epsilon \nleq l$, the definition of degeneracy implies that $F, F'$ have the same leaves and the result follows by the criterion in the ``further'' part of
Lemma \ref{INEQMAXIFF NOPEN LEM}.
\end{proof}

\begin{lemma}\label{BROADTRANS LEM}
	Let $F \xrightarrow{\varphi} F'$ be a map in $\Phi_i$ and $\underline{f} \leq e$ be a broad relation in $F$ such that $\varphi(\underline{f}) \neq \varphi(e)$. Then for any $\underline{f}',e'$ such that $\varphi(\underline{f}')=\varphi(\underline{f})$, $\varphi(e')=\varphi(e)$ it is also $\underline{f}' \leq e'$. 
\end{lemma}

\begin{proof}
Writing $\underline{f}=f_1\cdots f_n$, $\underline{f}'=f'_1\cdots f'_n$, the condition 
$\varphi(\underline{f})\neq \varphi(e)$ ensures $f_i<_d e$, and
thus Lemma \ref{LEQDCOMP LEM}
implies it is also $f_i'<_d e'$. Using the characterization in Lemma \ref{INEQMAXIFF NOPEN LEM}, the desired relation $\underline{f}' \leq e'$ will hold provided we show that $a^{\lambda}=\bar{a}^{\lambda}$ whenever $\varphi(a)=\varphi(\bar{a})$. But since leaves are $\leq_d$-minimal, this too follows from Lemma \ref{LEQDCOMP LEM}. 
\end{proof}

	

\begin{remark}
	It follows from the ``further'' part in Lemma \ref{LEQDCOMP LEM} together with antisymmetry that the pre-image of any edge by a map $\varphi$ always consists of a linearly ordered subset of edges. As such, degeneracies are necessarily maps that ``collapse linear sections of a tree'', and indeed that was their original description in \cite{We12}. A typical degeneracy (sending edges labeled $a_i$, $b_i$, $c_i$, $d_i$ to respective edges $a$, $b$, $c$, $d$) is pictured below.
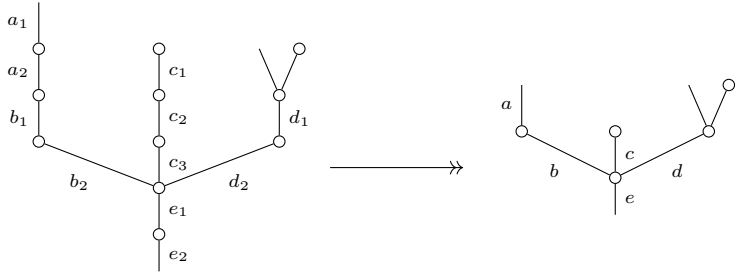
\[%
	\begin{tikzpicture}[auto,grow=up, every node/.style={font=\scriptsize},
	dummy/.style={circle,draw,inner sep=0pt,minimum size=1.5mm},level distance = 1.75em]%
	\begin{scope}%
	\tikzstyle{level 3}=[sibling distance=4.5em]%
	\tikzstyle{level 4}=[sibling distance=1.5em]%
		\node at(-0.5,0) {}%
			child{node[dummy] {}
				child{node[dummy] {}
					child{node[dummy] {}
						child{node[dummy] {}
							child{node[dummy] {}}
							child
						edge from parent node [swap] {$d_1$}}
					edge from parent node [swap] {$d_2$}}
					child{node[dummy] {}
						child{node[dummy] {}
							child{node[dummy] {}
							edge from parent node [swap] {$c_1$}}
						edge from parent node [swap]{$c_2$}}
					edge from parent node [swap]{$c_3$}}
					child{node[dummy] {}
						child{node[dummy] {}
							child{node[dummy] {}
								child{edge from parent node {$a_1$}}
							edge from parent node {$a_2$}}
						edge from parent node {$b_1$}}
					edge from parent node {$b_2$}}
				edge from parent node [swap]{$e_1$}}
			edge from parent node [swap] {$e_2$}};
	\end{scope}
	\begin{scope}%
	\tikzstyle{level 2}=[sibling distance=3.5em]%
	\tikzstyle{level 3}=[sibling distance=1.5em]%
		\node at(5.5,0.75) {}%
				child{node[dummy] {}
						child{node[dummy] {}
							child{node[dummy] {}}
							child
						edge from parent node [swap] {$d$}}
							child{node[dummy] {}
							edge from parent node [swap] {$c$}}
						child{node[dummy] {}
								child{edge from parent node {$a$}}
						edge from parent node {$b$}}
				edge from parent node [swap] {$e$}};
		\draw[->>] (1.75,1.5) -- (3.5,1.5);
	\end{scope}
	\end{tikzpicture}
\]
\end{remark}

\begin{lemma}\label{SECTIONISTREE LEM}
	Let $\varphi \colon F \to F'$ be a map in $\Phi_i$
	and let $U \subset F$ be any sub-broad poset consisting of exactly one edge in each pre-image of $\varphi$ and the broad relations of $F$ between them. Then $U$ is a forest. 
\end{lemma}

\begin{proof}
Simplicity of $U$ is inherited from $F$. Given $u\in U$, let $\bar{u} \in F$ be the $\leq_d$-minimal edge such that 
$\varphi(\bar{u})=\varphi(u)$. Lemma \ref{BROADTRANS LEM} implies that $u$ is a leaf in $U$ if $\bar{u}$ is a leaf in $F$. Otherwise, $\bar{u}^{\uparrow}< \bar{u}$ (in $F$) and again by Lemma \ref{BROADTRANS LEM} the unique tuple $\underline{v}$ of $U$ such that 
$\varphi(\underline{v})=\varphi(\bar{u}^{\uparrow})$
provides the desired node tuple $u^{\uparrow}=\underline{v}$ for $U$.
Lastly, by Lemma \ref{BROADTRANS LEM} $s$ will be a root of $U$ iff $\varphi(s)=\varphi(r)$ for $r$ a root of $F$.
\end{proof}

We now prove the following factorization result, which in the $\Omega$ case first appeared as \cite[Lemma 3.1]{MW07} and in the $\Phi_i$ case was proven in \cite[Lemma 3.1.3]{HHM16}.

\begin{proposition}\label{FACT PROP}
	Each map $\varphi$ of $\Phi_i$ has a factorization 
	$\varphi = \varphi^+ \circ \varphi^-$ as a degeneracy followed by a face. Further, this decomposition is unique up to unique isomorphism.

	Finally, the decomposition restricts to the subcategories $\Omega$ and $\Phi_w$.
\end{proposition}

\begin{proof}
	Given a map $F \xrightarrow{\varphi} F'$ and picking any $U$
as in Lemma \ref{SECTIONISTREE LEM}, the isomorphism $U\simeq \varphi(F)$ allows us to equip $\varphi(F)$ with a forest structure. Moreover, by Lemma \ref{BROADTRANS LEM} the broad relations of $\varphi(F)$ are exactly the image of those in $F$, and thus independent of the chosen $U$. The existence of a factorization follows.

	The uniqueness of the $F \overset{\varphi^-}{\twoheadrightarrow} G \overset{\varphi^+}{\rightarrowtail} F'$ factorization follows since by the description of the broad structure on $\varphi(F)$ there is clearly a broad poset map $\varphi(F) \to G$, which by Lemma \ref{FACEDEGISO LEM} is an isomorphism.

	Further, the decomposition restricts to $\Omega$ since the image by a degeneracy map of a tree is necessarily a tree and restricts to $\Phi_{w}$ since any edge surjective map $F \overset{\varphi^-}{\twoheadrightarrow} G$ must map roots to roots and hence it must be $\varphi^-(\underline{r}_F)=\underline{r}_G$.
\end{proof}

\begin{corollary}\label{SPLITDEG COR}
	If $F \overset{\rho}{\twoheadrightarrow} F'$ is a degeneracy in any of $\Omega,\Phi_i,\Phi_w$ then the broad relations in $F'$ are precisely the image of the broad relations of $F$. Further, any section $F'\xrightarrow{s} F$ of the underlying set map is a section in $\Omega,\Phi_i,\Phi_w$.
\end{corollary}

The following will be needed in \S \ref{PRESHEAFCAT SEC} when discussing dendroidal boundaries.

\begin{corollary}\label{ABSPUSH COR}
	In any of $\Omega,\Phi_i,\Phi_w$, pairs of degeneracies with common domain have absolute pushouts\footnote{Recall that an absolute colimit can be described as either a colimit that is preserved by the Yoneda embedding or (equivalently) a colimit that is preserved by any functor.}.
\end{corollary}

\begin{proof}
We will throughout write $\Xi$ for any of $\Omega,\Phi_i,\Phi_w$ and $\Xi[F]\in \mathsf{Set}^{\Xi^{op}}$ for the presheaf represented by $F\in \Xi$.
Given a diagram of degeneracies 
$E \overset{\bar{\rho}}{\twoheadleftarrow}
 F \overset{\rho}{\twoheadrightarrow} E_0$
we first inductively extend it to a diagram of degeneracies as on the left below
\begin{equation}\label{ABSPUSHCOR EQ}
\begin{tikzcd}
	F \ar[swap,twoheadrightarrow]{d}{\bar{\rho}} 
	\ar[twoheadrightarrow]{r}{\rho} &
	E_0 \ar[twoheadrightarrow]{r}{\rho} &
	E_1 \ar[twoheadrightarrow]{r}{\rho} &
	\cdots \ar[twoheadrightarrow]{r}{\rho} &
	E_n &
	E_i \ar{r}{s_i} 
	\ar[swap,twoheadrightarrow]{d}{\rho} &
	F \ar[twoheadrightarrow]{d}{\bar{\rho}} 
 \\
	\bar{E} \ar[twoheadrightarrow]{rrrru} &&&&&
	E_{i+1} \ar{r}[swap]{\iota} &
	\bar{E}.
\end{tikzcd}
\end{equation}
as follows: 
assuming $E_0 \to \cdots \to E_i$ have been built,
Corollary \ref{SPLITDEG COR}  implies that a (necessarily unique) compatible degeneracy
$\bar{E} \to E_i$
exists iff it exists at the level of the underlying sets;
otherwise, there must exist edges 
$e_1,e_2 \in E_i$
and lifts 
$f_1,f_2 \in F$ 
such that
$\rho^{\circ i+1}(f_k) = e_k$
and
$\bar{\rho}(f_1) = \bar{\rho}(f_2)$,
and choosing $s_i$ to be a section of 
$\rho^{\circ i+1}$
such that $s_i(e_k) = f_k$,
one defines $\rho \colon E_i \to E_{i+1}$
via the degeneracy-face factorization of $\bar{\rho} s_i$,
as in the right square in (\ref{ABSPUSHCOR EQ}). The condition on $s_i$ then guarantees that $\rho \colon E_i \twoheadrightarrow E_{i+1}$ is never an isomorphism, so that this procedure always terminates.

It remains to show that 
$\Xi[\bar{E}] \coprod_{\Xi[F]} \Xi[E_0] \to \Xi[E_n]$
is an isomorphism. Surjectivity is immediate from the existence of a section $s\colon E_n \to E_0$ of $\rho^{\circ n}$. For injectivity, the existence of a section 
$\bar{s}\colon \bar{E} \to F$ of $\bar{\rho}$ implies that one can represent any element 
of $\Xi[\bar{E}] \coprod_{\Xi[F]} \Xi[E_0]$ by an element
in $\Xi[E_0]$, 
i.e. as an equivalence class $[\varphi]$ 
for a map $\varphi \colon T \to E_0$.
It now suffices to show by induction on $i$ that
if $\rho^{\circ i} \varphi = \rho^{\circ i} \psi$
then it is $[\varphi] = [\psi]$ or, equivalently, that it is $[\varphi] = [\rho s_i \rho^{\circ i} \varphi]$. The case $i=0$ is trivial.
Otherwise, we claim that
\[
	[\varphi]
=
	[\rho s_{i-1}\rho^{\circ i-1}\varphi]
=
	[\rho s_{i-1} \rho^{\circ i} s_i \rho^{\circ i} \varphi]
=
	[\rho s_i \rho^{\circ i} \varphi].
\]
The leftmost and rightmost identities follow from the induction hypothesis for $i-1$ and the functions
$\varphi$, $\rho s_i \rho^{\circ i} \varphi$.
For the inner identity, one considers the lifts 
$ s_{i-1}\rho^{\circ i-1}\varphi \colon T \to F$,
$
 s_{i-1} \rho^{\circ i} s_i \rho^{\circ i} \varphi
\colon T \to F
$
and notes that postcomposition with $\bar{\rho}$ yields
$
\bar{\rho}s_{i-1}\rho^{\circ i-1}\varphi =
\iota \rho^{\circ i}\varphi
$
and
$
\bar{\rho} s_{i-1} \rho^{\circ i} s_i \rho^{\circ i} \varphi =
\iota \rho^{\circ i+1} s_i \rho^{\circ i} \varphi =
\iota \rho^{\circ i} \varphi
$. Injectivity now follows, finishing the proof.
\end{proof}

We now recall the usual (\cite{We12},\cite{MW08},\cite{CM11},\cite{HHM16}) description of faces as composites of maximal ``codimension 1'' faces. We first discuss some terminology.

Firstly, we will regard a face $F'$ of $F$ as a subset of $F$ together with a subset of the broad relations of $F$, and write $F' \hookrightarrow F$ to indicate this. Further, if the broad relations between edges of $F'$ in the broad posets $F'$ and $F$ coincide, then we will call $F'$ a \textit{full face} of $F$ and write $F' \subset F$ instead.

Secondly, an edge $e \in F$ is called
\textit{external} if $e$ is either a leaf or a root and \textit{internal} otherwise.

Finally, we denote a generating broad relation of $F$ by $v_e=(e^{\uparrow} \leq e)$ and refer to it as the \textit{vertex at $e$}.

\begin{notation}\label{FACETYPES NOT}
The maximal faces of $F$ in $\Omega,\Phi_i,\Phi_w$ have the following types.
\begin{itemize}
\item
The \textit{inner face} (valid for any of $\Omega,\Phi_i,\Phi_w$) associated to an inner edge $e$ is the full face $F-e \subset F$ obtained by removing $e$;
\item
The \textit{leaf vertex outer face} (valid for any of $\Omega,\Phi_i,\Phi_w$) associated to a
vertex $v_e$ such that $e^{\uparrow}$ consists of leaves is the full face $F_{\nless e} \subset F$ obtained by removing the leafs in $e^{\uparrow}$;
\item
The \textit{stump outer face} (valid for any of $\Omega,\Phi_i,\Phi_w$) associated to a stump vertex $v_e=(\epsilon = e^{\uparrow} \leq e)$ is the face $F_{\nless e} \hookrightarrow F$ with the same edges as $F$ but removing $\epsilon \leq e$ as a generating broad relation (note that this also removes some composite relations);
\item
The \textit{root vertex outer face} (valid only for $\Omega$) for an edge $e \in r^{\uparrow}$ such that the edges of $r^{\uparrow}$ other than $e$ are leaves is the full face 
$T^{\leq e} \subset T$ consisting of those edges $\bar{e}$ such that $\bar{e} \leq_d e$;
\item
The \textit{root face} (valid only for $\Phi_i,\Phi_w$) associated to a root $r_i \in \underline{r}_F$ which is\textit{ not} also a leaf is the full face $F-r_i \subset F$ obtained by removing $r_i$;
\item
The \textit{stick component face} (valid only for $\Phi_i$) associated to a stick $\eta \in F$ 
(i.e. an edge that is simultaneously a root and a leaf) is the full face $F-\eta \subset F$ obtained by removing $\eta$.
\end{itemize}
\end{notation}

\begin{remark}
	The implicit claim that an inner face $F-e$ is itself a forest can easily be checked using Lemma \ref{INEQMAXIFF NOPEN LEM}, which shows that $f^{\uparrow,F-e}$ can be defined to consist of the $\leq_d$-maximal edges $f_i \neq e$ such that $f_i \leq_d f$.
	
	Similarly, Lemma \ref{INEQMAXIFF NOPEN LEM} shows that a broad relation $f_1 \cdots f_n\leq f$ in $F$ holds in a stump outer face $F_{\nless e}$ if the condition $e \leq_d f$ implies that $e \leq_d f_i$ for some $i$.
\end{remark}

\begin{example}
	Consider the trees in Example \ref{EXAMPLEFORESTS EX}. One can write
	\[U = \left(\left(\left(\left(\left(T_{\nless c}\right)_{\nless d}\right)-f\right)-i\right)-g\right),\]
where the intermediate steps (from the inside out) are a stump face, a leaf vertex face and three inner faces. Further, both $\eta_a=V^{\leq a}$ and $\eta_b=V^{\leq b}$ are root vertex outer faces of $V$ when viewing $V$ as a tree and $\eta_a \amalg \eta_b=V-d$ is a root face of $V$ when viewing $V$ as a forest.
\end{example}

\subsection{The category of equivariant trees}\label{EQUIVTREECAT SEC}

Let $G$ be a finite group. We will denote by $\Phi^G$ the category of $G$-forests, i.e. forests equipped with a $G$-action.

\begin{definition}\label{GTREES DEF}
	The \textit{category of $G$-trees}, denoted $\Omega_G$, is the full subcategory $\Omega_G \subset \Phi^G$ of $G$-forests $F$ such that
the root tuple $\underline{r}_F$ consists of a single $G$-orbit.
\end{definition}

\begin{remark}\label{FORVSTREE REM}
	The relationship between $\Phi^G$ and $\Omega_G$
	is similar to the relationship between the category 
	$\mathsf{Fin}^G$ of finite $G$-sets and the \textit{orbital category} $\mathsf{O}_G$ consisting of the orbital $G$-sets $G/H$.
\end{remark}

Examples of equivariant trees can be found throughout 
\S \ref{EQUIVTREESINTO SEC}. The author is aware that the fact that $G$-trees often ``look life forests'' is likely counter-intuitive at first (indeed, that was a major hurdle in the development of the theory presented in this paper). However, the following two facts may assuage such concerns: 
\begin{inparaenum}
	\item[(i)] similarly to how a non-equivariant tree is a forest that can not be decomposed as a coproduct of forests, so too a $G$-tree can not be equivariantly decomposed as a coproduct of $G$-forests;
	\item[(ii)] the orbital representation of a $G$-tree (cf. \S \ref{EQUIVTREESINTO SEC}) always \textit{does} ``look like a tree''.
\end{inparaenum}

\begin{remark}
Note that $\Omega_G$, \textit{the category of $G$-trees}, is rather different from $\Omega^G$,\textit{ the category of trees with a $G$-action}. In fact, each $G$-tree is (non-canonically) isomorphic to a forest of the form $G \cdot_H T$ for some $H \leq G$ and $T \in \Omega^H$. More precisely, one has the following elementary proposition.
\end{remark}

\begin{proposition}\label{GTREESGROTH PROP}
	$\Omega_G$ is equivalent to the Grothendieck construction for the functor (where $\mathsf{O}_G$ is the orbit category; see Remark \ref{FORVSTREE REM})
\[
\begin{tikzcd}[row sep=0em]
	\mathsf{O}_G^{op} \ar{r} & \mathsf{Cat} \\
	G/H \ar[mapsto]{r} & \Omega^{\underline{G/H}}
\end{tikzcd}
\]
where $\underline{G/H}$ denotes the groupoid with objects the cosets $g H$ and arrows $gH \xrightarrow{\bar{g}} \bar{g} g H$.
\end{proposition}

\begin{remark}\label{OMEGAGINC RMK}
	There is a natural inclusion $G \times \Omega \hookrightarrow \Omega_G$ given by regarding each object $(\**,T)\in G \times \Omega$ as the $G$-tree given by the $G$-free forest $G \cdot T$.
\end{remark}

\begin{remark}
While maps in $\Omega$ can be built out of two types of maps, faces and degeneracies (Proposition \ref{FACT PROP}), in $\Omega_G$ we need a third type of map: quotients. 

To see this, note that by Proposition \ref{GTREESGROTH PROP}, each $G$-tree $T$ sits (up to equivalence) inside one of the subcategories $\Omega^{\underline{G/H}}$, and that $\Omega^H$ is equivalent to  the latter. Since it is immediate by (the proof of) Proposition \ref{FACT PROP} that the degeneracy-face decomposition extends to $\Omega^H$, Proposition \ref{GTREESGROTH PROP} implies that any map in $\Omega_G$ factors as a degeneracy followed by a face (both inside one of the fibers $\Omega^{\underline{G/H}}$) followed by a cartesian map. We will prefer to refer to cartesian maps as quotients.

For a representative example, let $G=\mathbb{Z}_{/8}$ and consider the map below (represented in orbital notation, cf. \S \ref{EQUIVTREESINTO SEC}), where we follow the following conventions:
\begin{inparaenum}
	\item[(i)] edges in different trees with the same label are mapped to each other;
	\item[(ii)] if an edge is denoted $e^{+i}$, we assume that its orbit is disjoint from that of $e$ and that, if no edge labeled $e^{+i}$ appears in the target tree, then $\varphi(e^{+i})=\varphi(e)+i$.
\end{inparaenum}

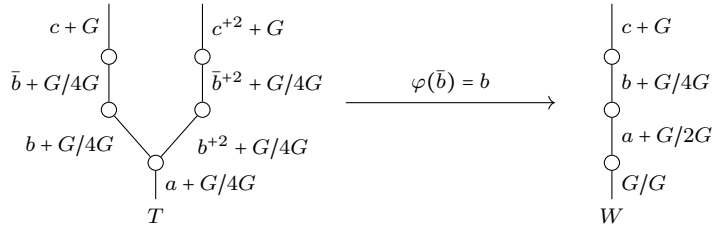
\[%
	\begin{tikzpicture}[auto, grow=up, every node/.style={font=\footnotesize},level distance = 2em]%
	\tikzstyle{level 2}=[sibling distance=3.5em]%
	\tikzstyle{level 3}=[sibling distance=2em]%
		\node at (0,0) {$T$}%
			child{node [dummy] {}
				child{node [dummy] {}
					child{node [dummy] {}
						child{
						edge from parent node [swap] {$c^{+2}+G$}}
					edge from parent node [swap] {$\bar{b}^{+2}+G/4G$}}
				edge from parent node [near end,swap] {$b^{+2}+G/4G$}}
				child{node [dummy] {}
					child{node [dummy] {}
						child{
						edge from parent node {$c+G$}}
					edge from parent node {$\bar{b}+G/4G$}}
				edge from parent node [near end] {$b+G/4G$}}
			edge from parent node [swap] {$a+G/4G$}};%
		\node at (6,0){$W$}%
			child{node [dummy] {}
				child{node [dummy] {}
					child{node [dummy] {}
						child{
						edge from parent node [swap] {$c+G$}}
					edge from parent node [swap] {$b+G/4G$}}
				edge from parent node [swap] {$a+G/2G$}}
			edge from parent node [swap] {$G/G$}};%
		\draw[->] (2.5,1.5) -- node [above] 
		{$\varphi(\bar{b}) = b$}
		(5.25,1.5);
	\end{tikzpicture}%
\]%
This map can be factored as (where, for the sake of brevity, we write $a+G/4G$ as $a/4$, etc.)

\begin{equation}\label{OMEGAG FACT}%
	\begin{tikzpicture}[auto, grow=up, every node/.style={font=\footnotesize},level distance = 2em]%
	\tikzstyle{level 2}=[sibling distance=2.5em]%
	\tikzstyle{level 3}=[sibling distance=2em]%
		\node at (-0.2,0) {$T$}%
			child{node [dummy] {}
				child{node [dummy] {}
					child{node [dummy] {}
						child{
						edge from parent node [swap] {$c^{+2}/8$}}
					edge from parent node [swap] {$\bar{b}^{+2}/4$}}
				edge from parent node [near end,swap] {$b^{+2}/4$}}
				child{node [dummy] {}
					child{node [dummy] {}
						child{
						edge from parent node {$c/8$}}
					edge from parent node {$\bar{b}/4$}}
				edge from parent node [near end] {$b/4$}}
			edge from parent node [swap] {$a/4$}};%
		\node at (3.15,0) {$U$};%
		\node at (3.15,0.3) {}%
			child{node [dummy] {}
				child{node [dummy] {}
						child{
						edge from parent node [swap] {$c^{+2}/8$}}
				edge from parent node [near end,swap] {$b^{+2}/4$}}
				child{node [dummy] {}
						child{
						edge from parent node {$c/8$}}
				edge from parent node [near end] {$b/4$}}
			edge from parent node [swap] {$a/4$}};%
		\begin{scope}
		\tikzstyle{level 2}=[sibling distance=6em]%
		\tikzstyle{level 3}=[sibling distance=1em]%
		\node at (6.9,0) {$V$}%
			child{node [dummy] {}
				child{node [dummy] {}
					child{node [dummy] {}
							child{
							edge from parent node [swap] {$c^{+3}/8$}}
					edge from parent node [near end,swap] {$b^{+3}/4$}}
					child{node [dummy] {}
							child{
							edge from parent node {$c^{+1}/8$}}
					edge from parent node [near end] {$b^{+1}/4$}}
				edge from parent node [swap,near end] {$a^{+1}/4$}}%
				child{node [dummy] {}
					child{node [dummy] {}
							child{
							edge from parent node [swap] {$c^{+2}/8$}}
					edge from parent node [near end,swap] {$b^{+2}/4$}}
					child{node [dummy] {}
							child{
							edge from parent node {$c/8$}}
					edge from parent node [near end] {$b/4$}}
				edge from parent node[near end] {$\phantom{1^1}a/4$}}%
			edge from parent node [swap] {$G/4G$}};
		\end{scope}
		\node at (10.2,0){$W$}%
			child{node [dummy] {}
				child{node [dummy] {}
					child{node [dummy] {}
						child{
						edge from parent node [swap] {$c/8$}}
					edge from parent node [swap] {$b/4$}}
				edge from parent node [swap] {$a/2$}}
			edge from parent node [swap] {$G/G$}};%
		\draw[->] (1.25,1.5) -- node [above] 
		{$\bar{b} \mapsto b$}
		(2,1.5);
		\draw[->] (4.45,1.5) -- (4.95,1.5);
		\draw[->] (9.15,1.5) -- (9.65,1.5);
	\end{tikzpicture}.%
\end{equation}%
It is perhaps worthwhile to unpack the last map in (\ref{OMEGAG FACT}), which is an example of a quotient. The $G$-tree labeled $V$ can be written as 
$V \simeq G \cdot_{4G} \bar{V}$ where 
$\bar{V} \in \Omega^{4G}$ is the tree with a $4G$-action pictured below.
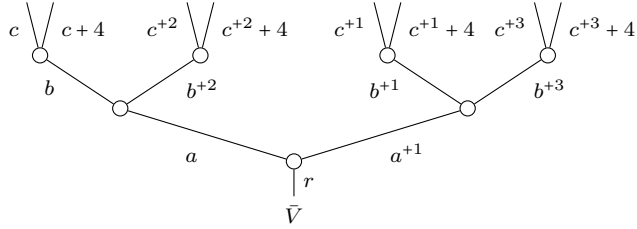
\[
	\begin{tikzpicture}[auto, grow=up, every node/.style={font=\footnotesize},level distance = 2em]%
		\begin{scope}
		\tikzstyle{level 2}=[sibling distance=13em]%
		\tikzstyle{level 3}=[sibling distance=6em]%
		\tikzstyle{level 4}=[sibling distance=1em]%
		\node at (0,0) {$\bar{V}$}%
			child{node [dummy] {}
				child{node [dummy] {}
					child{node [dummy] {}
							child{
							edge from parent node [swap, very near end] {$c^{+3}+4$}}
							child{
							edge from parent node [very near end] {$c^{+3}$}}
					edge from parent node [near end,swap] {$b^{+3}$}}
					child{node [dummy] {}
							child{
							edge from parent node [swap, very near end] {$c^{+1}+4$}}
							child{
							edge from parent node [very near end] {$c^{+1}$}}
					edge from parent node [near end] {$b^{+1}$}}
				edge from parent node [swap] {$a^{+1}$}}%
				child{node [dummy] {}
					child{node [dummy] {}
							child{
							edge from parent node [swap,very near end] {$c^{+2}+4$}}
							child{
							edge from parent node [very near end] {$c^{+2}$}}
					edge from parent node [near end,swap] {$b^{+2}$}}
					child{node [dummy] {}
							child{
							edge from parent node [swap,very near end] {$c+4\phantom{1^1}$}}
							child{
							edge from parent node [very near end] {$\phantom{1^2}c$}}
					edge from parent node [near end] {$b$}}
				edge from parent node {$\phantom{1^1}a$}}%
			edge from parent node [swap] {$r$}};
		\end{scope}
	\end{tikzpicture}
\]
In words, $V$ consists (non-equivariantly) of four trees identical to $\bar{V}$ which are interchanged by the action of elements of $G$ other than $0,4$. $W$, on the other hand, consists of a single (non-equivariant) tree, also shaped like $\bar{V}$, and can be thought of as the quotient of $V$ obtained by gluing the four trees so that the edges $e^{+i}+j$ and $e+i+j$ are identified.
\end{remark}

\begin{remark}\label{NOTREEDY RMK}
	One particularly convenient property of $\Omega$ is that $\Omega^{op}$ is a generalized Reedy category, in the sense of \cite{BM08}. In fact, $\Omega$ is a \textit{dualizable} generalized Reedy category, so that both $\Omega$ and $\Omega^{op}$ are generalized Reedy.

Unfortunately, this is not the case for $\Omega_G$: while indeed $\Omega_G$ itself can be shown to be generalized Reedy, the opposite category $\Omega^{op}_G$ is not. The problem is readily apparent in the factorization in (\ref{OMEGAG FACT}). Indeed, for the Reedy factorizations to hold (cf. \cite[Defn. 1.1(iii)]{BM08}), quotient maps would need to be considered the same type of maps as face maps, i.e. degree raising maps of $\Omega_G$. However, the quotient map in (\ref{OMEGAG FACT}) fails \cite[Defn. 1.1(iv')]{BM08}, since there is an automorphism of $V$ 
(given by $e^{+i} \mapsto e^{+i+1}-1$, where $e^{+0}$ is interpreted as $e$, and yet undefined $e^{+i}$ labels are interpreted by regarding $i \in \mathbb{Z}_{/k}$ as needed) compatible with the quotient map to $W$.
\end{remark}

\subsection{Presheaf categories}\label{PRESHEAFCAT SEC}

We now establish some key terminology and notation concerning the presheaf categories we will use. Recall that the category of \textit{dendroidal sets} is the presheaf category 
$\mathsf{dSet}=\mathsf{Set}^{\Omega^{op}}$.

\begin{definition}
	The category of $G$-\textit{equivariant dendroidal sets} is the category
	\[\mathsf{dSet}^G = \mathsf{Set}^{\Omega^{op} \times G}.\]
	The category of \textit{genuine $G$-equivariant dendroidal sets} is the category
	\[\mathsf{dSet}_G = \mathsf{Set}^{\Omega_G^{op}}.\]
\end{definition}

Twisting the inclusion in Remark \ref{OMEGAGINC RMK} by the inverse map $G^{op}\xrightarrow{(-)^{-1}}G$ yields an inclusion 
$u \colon \Omega^{op} \times G \hookrightarrow \Omega_G$.
\begin{proposition}\label{REFLEXCAT PROP}
The adjunction
	\begin{equation}\label{REFLEXADJ EQ}
		u^{\**} \colon \mathsf{dSet}_G \rightleftarrows \mathsf{dSet}^G \colon u_{\**}
	\end{equation}
identifies $\mathsf{dSet}^G$ as a reflexive subcategory of $\mathsf{dSet}_G$.
\end{proposition}

\begin{remark}\label{TWODSETG REM}
	Since Theorem \ref{GINFTYOP THM} concerns $\mathsf{dSet}^G$, that category will be our main focus throughout the present paper, although $\mathsf{dSet}_G$ also plays a role in its proof (cf. \S \ref{FIBOBJCHAR SEC}).
	
	Nonetheless, $\mathsf{dSet}_G$ is arguably the most interesting category. Indeed, the adjunction (\ref{REFLEXADJ EQ}) bears many similarities to the adjunction $\mathsf{sSet}^{\mathsf{O}_G^{op}} \rightleftarrows \mathsf{sSet}^{G^{op}}$ and, as will be shown in upcoming work, the full structure of the ``homotopy operad'' of a $G$-$\infty$-operad is described as an object in $\mathsf{dSet}_G$ rather than in $\mathsf{dSet}^G$
	(more precisely, the claim is that the homotopy operad of a 
	$G$-$\infty$-operad forms a ``colored genuine equivariant operad''; (single colored) genuine equivariant operads have been recently formalized by the author and Peter Bonventre in \cite{BP17}). 
	We note that this is similar to how $\pi_n$ of a $G$-space forms a $G$-coefficient system rather than just a $G$-set.
	We conjecture that a model structure on $\mathsf{dSet}_G$ making (\ref{REFLEXADJ EQ}) into a Quillen equivalence exists, and that too is the subject of current work. 
	The presence of extra technical difficulties when dealing with $\Omega_G$ (cf. Remark \ref{NOTREEDY RMK}), however, make it preferable to address the $\mathsf{dSet}^G$ case first. 
\end{remark}

\begin{notation}
	Recall the usual notation
		\begin{equation}\label{YONEDA EQ}
		\Omega \xrightarrow{T \mapsto \Omega[T]} \mathsf{dSet}
		\end{equation}
	for the Yoneda embedding.
	
	One can naturally extend this notation to the category $\Phi$ of forests: given $F = \amalg_i T_i$, set $\Omega[F]=\amalg_i \Omega[T_i]$. Passing to the $G$-equivariant object categories and using the inclusion $\Omega_G \hookrightarrow \Phi^G$ we will slightly abuse notation and write
	\begin{equation}\label{YONEDAG EQ}
		\Omega_G \xrightarrow{T \mapsto \Omega[T]} \mathsf{dSet}^G.
	\end{equation}
More explicitly, if $T \simeq G \cdot_H T_e$ for some $T_e \in \Omega^H$, then $\Omega[T] \simeq G \cdot_H \Omega[T_e]$, where $\Omega[T_e]$ is just the Yoneda embedding of (\ref{YONEDA EQ}) together with the resulting $H$-action.
\end{notation}

\begin{remark}\label{REPEV RMK}
	Note that while (\ref{YONEDAG EQ}) defines ``representable functors'' for each $T \in \Omega_G$, given a presheaf $X \in \mathsf{dSet}^G$ the evaluations $X(U)$ are defined only for $U \in \Omega$, i.e., for $U$ a non-equivariant tree.
	
	This is in contrast with $\mathsf{dSet}_G$, where both representables and evaluations are defined in terms of $\Omega_G$. We note that to reconcile this observation with the inclusion $u_{\**}$ of (\ref{REFLEXADJ EQ}) the non-equivariant tree $U \in \Omega$ should be reinterpreted as the free $G$-tree $G \cdot U \in \Omega_G$ (cf. Remark \ref{OMEGAGINC RMK}).
\end{remark}

We end this section by discussing a category of ``forestial sets'' which, while secondary for our purposes, will greatly streamline our discussion of the dendroidal join in \S \ref{DENDJOIN SEC}.

\begin{definition}
	The category of \textit{wide forestial sets} is the category
	\[\mathsf{fSet}_w = \mathsf{Set}^{\Phi^{op}_w}.\]
\end{definition}

\begin{remark}
	The category $\mathsf{fSet}_i = \mathsf{Set}^{\Phi^{op}_i}$ of what we might call ``independent forestial sets'' was one of the main objects of study in \cite{HHM16}, where they are called simply ``forest sets''.
\end{remark}

Mimicking (\ref{YONEDA EQ}) by writing $\Phi_i[F] \in \mathsf{fSet}_i$ for the representable functor of $F \in \Phi_i$, it is shown in \cite{HHM16} that one can define a formal boundary $\partial \Phi_i[F]$ possessing the usual properties one might expect. 

We will find it desirable to be able to use the analogous construction for the representable $\Phi_w[F] \in \mathsf{fSet}_w$, but this does not quite follow from the result in \cite{HHM16}, since while $\Phi_i[F] \in \mathsf{fSet}_i$ can be forgotten to a presheaf $u^{\**}\Phi_i[F] \in \mathsf{fSet}_w$, one typically has a proper inclusion $\Phi_w[F] \hookrightarrow u^{\**}\Phi_i[F]$.

We thus instead mimic the discussion in \cite{BM08}, making use of the key technical results established in \S \ref{CATFOR SEC}.

Letting $\Xi$ denote any of $\Omega,\Phi_i,\Phi_w$ and setting 
$$|F|=\#\{\text{edges of $F$}\} + \# \{\text{stumps of $F$}\},$$ then Lemma  \ref{FACEDEGISO LEM} and Proposition \ref{FACT PROP} say that $\Xi$ is a dualizable generalized Reedy category (cf. \cite[Defn. 1.1]{BM08}).
As in \cite[\S 6]{BM08}, call an element $x \colon \Xi[F] \to X$ of a presheaf $X \in \mathsf{Set}^{\Xi^{op}}$ \textit{degenerate} if it factors through a non invertible degeneracy operator and \textit{non-degenerate} otherwise. Corollary \ref{ABSPUSH COR} then allows us to adapt the proof of 
\cite[Prop. 6.9]{BM08} to obtain the following.

\begin{proposition}\label{UNIQUEFACT LEMMA}
Let $X \in \mathsf{Set}^{\Xi^{op}}$ for $\Xi$ any of $\Omega,\Phi_i,\Phi_w$. Then any element $x \colon \Xi[F] \to X$ has a factorization, unique up to unique isomorphism, 
\[
 \Xi[F] \xrightarrow{\rho_x} \Xi[G] \xrightarrow{\bar{x}}  X
\]
as a degeneracy operator $\rho_x$ followed by a non degenerate element $\bar{x}$. 
\end{proposition}


Defining skeleta as in \cite[\S 6]{BM08} the proof of \cite[Cor. 6.8]{BM08}	yields the following.

\begin{corollary}\label{SKELL COR}
Let $\Xi$ be any of $\Omega,\Phi_i,\Phi_w$. 
The counit $sk_n X \to X$ for $X \in \mathsf{Set}^{\Xi^{op}}$ is a monomorphism whose image consists of those elements of $X$ that factor through some $\Xi[F] \to X$ for $|F| \leq n$. 
\end{corollary}


\begin{definition}
	Let $\Xi$ be any of $\Omega,\Phi_i,\Phi_w$. 
	The \textit{formal boundary} 
\[\partial \Xi[F] \hookrightarrow \Xi[F]\]
 is the subobject formed by those maps that factor through a non invertible map in $\Xi^+$, i.e. through a non invertible face map.
\end{definition}

Note that by combining the Reedy axioms \cite[Defn. 1.1]{BM08} with Corollary \ref{SKELL COR} one has
\[\partial \Xi[F] \simeq sk_{|F|-1} \Xi[F].\]

\section{Normal monomorphisms and anodyne extensions}\label{NORMONANODYNEEXT SEC}

\subsection{Equivariant normal monomorphisms}

Recalling that the cofibrations in $\mathsf{dSet}$ are not the full class of monomorphisms \cite[Prop. 1.5]{CM11}, but rather the subclass of so called \textit{normal monomorphisms}, one should expect a similar phenomenon to take place in $\mathsf{dSet}^G$.

We start by noting that for $X \in \mathsf{dSet}^G$ and $U \in \Omega$, the set $X_U=X(U)$ is acted on by the group $G \times \Sigma_U$, where $\Sigma_U$ denotes the automorphism group of $U$.

\begin{definition}
	A subgroup $N \leq G \times \Sigma_U$ is called 
	a \textit{$G$-graph subgroup} if $N \cap \Sigma_U = \**$.
\end{definition}

It is straightforward to check that a $G$-graph subgroup $N$ can equivalently be described by a partial homomorphism $G \geq H \xrightarrow{\rho} \Sigma_U$. Further, since such a $\rho$ allows us to write 
$U \in \Omega^H$, one has that each $G$-graph subgroup $N$ has an associated $G$-tree $G \cdot_H U$. More precisely, one has the following result.

\begin{proposition}
The functor $\Omega[-] \colon \Omega_G \to \mathsf{dSet}^G$ induces an equivalence between $\Omega_G$ and the full subcategory of quotients of the form $(G \cdot \Omega[U])/N$ for $U \in \Omega$ and $N\leq G \times \Sigma_U$ a $G$-graph subgroup.
\end{proposition}

Recalling the discussion following (\ref{YONEDA EQ}) one can, for a forest 
$F \simeq \amalg_i T_i$ in $\Phi$, define 
$\partial \Omega[F] = \amalg_i \partial \Omega[T_i]$. Carrying this discussion through to $G$-objects leads to the following definition.

\begin{definition}\label{BOUNDINCG DEF}
	The \textit{boundary inclusions} of $\mathsf{dSet}^G$ are the maps of the form
	\begin{equation}\label{BOUNDINCG EQ}
		\partial \Omega[T] \hookrightarrow \Omega [T]
	\end{equation}
	for $T \in \Omega_G$.
	
	More explicitly, if $T \simeq G \cdot_H T_e$ for some $T_e \in \Omega^H$, then the (non-equivariant) presheaf $\Omega[T_e]$ inherits a $H$-action and (\ref{BOUNDINCG EQ}) is isomorphic to the map 
	\[G \cdot_H \left(\partial\Omega[T_e] \hookrightarrow \Omega[T_e]\right)\]
or, letting $N \leq G \times \Sigma_{T_e}$ denote the $G$-graph subgroup associated to $T_e$,
\[\left(G\cdot \left( \partial \Omega[T_e] \hookrightarrow \Omega[T_e]\right)\right)/N.\]
\end{definition}

The following is an immediate generalization of \cite[Prop. 7.2]{BM08}.

\begin{proposition}\label{NORMCHAR PROP}
Let $\phi \colon X \to Y$ be a map in $\mathsf{dSet}^{G}$. Then the following are equivalent.
\begin{itemize}
\item[(i)] for each tree $U \in \Omega$, the relative latching map 
$l_U(\phi) \colon X_{U} \underset{L_{U} X}{\coprod} L_{U} Y \to Y_{U}$ 
is a $\Sigma_{U}$-free extension;
\item[(ii)] $\phi$ is a monomorphism and, for each $U \in \Omega$ and non degenerate $y \in Y_{U} - \phi(X)_{U}$, the isotropy group $\{g \in G \times \Sigma_{U} | gy=y\}$ is a $G$-graph subgroup;
\item[(iii)] for each $n \geq 0$, the relative $n$-skeleton 
$sk_n(\phi) = X \underset{sk_n X}{\coprod} sk_n Y$ 
is obtained from the relative $(n-1)$-skeleton by attaching boundary inclusions.
\end{itemize}
\end{proposition}

\begin{definition}
	A monomorphism satisfying any of the equivalent conditions in Proposition \ref{NORMCHAR PROP} will be called a \textit{$G$-normal monomorphism}.
\end{definition}

\begin{remark}\label{GNORMALUNDER REM}
	Note that by Proposition \ref{NORMCHAR PROP}(i) a monomorphism 
in $\mathsf{dSet}^G$ is $G$-normal  iff it is normal as a map in $\mathsf{dSet}$ \cite[Prop. 7.2]{BM08}.
As such, we often drop $G$ from the terminology.
Moreover, all monomorphisms over a $G$-normal dendroidal set are hence $G$-normal monomorphisms \cite[Corollaries 1.7 and 1.8]{CM11}.
\end{remark}

\subsection{Equivariant anodyne extensions}

The key to the preceding section was the observation that if $T \in \Omega^H$ then the usual boundary $\partial \Omega[T]$ inherits a $H$-action. However, such is not the case for inner horns: if $e \in T$ is an inner edge, then $\Lambda^e[T]$ (cf. \cite[\S 5]{MW08}) will inherit a $H$-action iff $e$ is a $H$-fixed edge.

Therefore, to define $G$-inner horns, one must treat all inner edges in an \textit{inner edge orbit} in an uniform way. To do so, we 
first recall the notion of generalized inner horns (cf. \cite[Lemma 5.1]{MW08}).

\begin{definition}
	Let $E \subset Inn(T)$ be a subset of the inner edges of $T \in \Omega$. We define
		\[\Lambda^{E}[T] \hookrightarrow \partial \Omega[T] \hookrightarrow \Omega[T]\]
	to be the subpresheaf formed by the union of those faces
	\textit{other than} the inner faces of the form
$T-E'$ for $E' \subset E$.
	
	More generally, given a forest $F = \amalg_{i} T_i$ and $E=\amalg_{i} E_i$ with $E_i \subset Inn(T_i)$ we set
		\begin{equation}\label{FORESTIALGEN HORN}
			\Lambda^E[F] = \amalg_{i} \Lambda^{E_i}[T_i].
		\end{equation}
\end{definition}

\begin{remark}
	The reader of \cite{HHM16} may note that (\ref{FORESTIALGEN HORN}) clashes with \cite[\S 3.6]{HHM16}. This is because in \cite{HHM16} the presheaf being defined lives in $\mathsf{fSet}_i$ rather than in $\mathsf{dSet}$.
\end{remark}

\begin{definition}\label{INNERHORNG DEF}
	The \textit{generating $G$-inner horn inclusions} are the maps in $\mathsf{dSet}^G$ of the form
	\[\Lambda^{Ge}[T] \to \Omega[T]\]
	where $T \in \Omega_G$ is a $G$-tree and $Ge$ is the $G$-orbit of an inner edge $e$. 
\end{definition}

\begin{definition}\label{GINFTYOP DEF}
	A $G$-dendroidal set $X$ is called a \textit{$G$-$\infty$-operad} if $X$ has the right lifting property with respect to all generating $G$-inner horn inclusions.
\[
\begin{tikzcd}
	\Lambda^{Ge}[T] \ar{r} \ar{d} & X \\
	\Omega[T] \ar[dashed]{ur}
\end{tikzcd}
\]
	Further, $A \to B$ is called a \textit{$G$-inner anodyne extension} if it is in the saturation of the generating $G$-inner horn inclusions under pushouts, transfinite compositions and retracts.
\end{definition}

\begin{example} If one considers the $G=\mathbb{Z}_{/4}$-tree $T$ in Example \ref{Z4 EX}, one possible inner orbit edge is $Gb=\{b,b+1\}$. 
The following are the (inner) faces of $T$ \textit{not} included in $\Lambda^{Ge}[T]$.
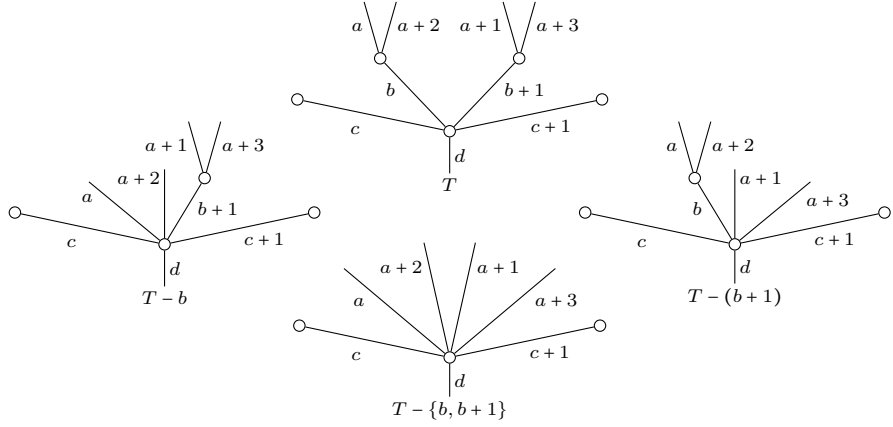
\[%
	\begin{tikzpicture}[auto,grow=up, every node/.style={font=\scriptsize,inner sep=1.5pt},
	dummy/.style = {circle,draw,inner sep=0pt,minimum size=1.5mm},
	level distance = 2em]%
	\tikzstyle{level 2}=[sibling distance=3.7em]%
	\tikzstyle{level 3}=[sibling distance=1.3em]%
		\node at (0,0){$T$}%
			child{node [dummy] {}%
				child[level distance = 1.2em,sibling distance=3.8em]{node [dummy]  {}%
				edge from parent node [swap] {$c+1$}}%
				child[sibling distance=5.2em,level distance = 2.75em]{node [dummy] {}%
					child[level distance = 2.3em]{node {} edge from parent node [swap,near end] {$a+3$}}%
					child[level distance = 2.3em]{node {} edge from parent node[near end] {$a+1$}}%
				edge from parent node [swap,near end] {$b+1$}}%
				child[sibling distance=5.2em,level distance = 2.75em]{node [dummy] {}%
					child[level distance = 2.3em]{node {} edge from parent node [swap,near end] {$a+2$}}%
					child[level distance = 2.3em]{node {} edge from parent node [near end] {$\phantom{1+}a$}}%
				edge from parent node [near end] {$b$}}%
				child[level distance = 1.2em,sibling distance=3.8em]{node [dummy] {}%
				edge from parent node {$c\phantom{||}$} }%
			edge from parent node [swap] {$d$}};%
		\node at (-3.75,-1.5){$T-b$}%
			child{node [dummy] {}%
				child[level distance = 1.2em,sibling distance=2.8em]{node [dummy]  {}%
				edge from parent node [swap] {$c+1$}}%
				child[sibling distance=1.5em,level distance = 2.5em]{node [dummy] {}%
					child[level distance = 2.3em]{node {} edge from parent node [swap,near end] {$a+3$}}%
					child[level distance = 2.3em]{node {} edge from parent node[near end] {$a+1$}}%
				edge from parent node [swap,near end] {$b+1$}}%
				child[level distance = 3em]{node {} edge from parent node [very near end] {$a+2$}}%
				child[level distance = 2.5em,sibling distance=3em]{node {} edge from parent node [very near end] {$a$}}%
				child[level distance = 1.2em,sibling distance=2.8em]{node [dummy] {}%
				edge from parent node {$c\phantom{||}$} }%
			edge from parent node [swap] {$d$}};%
		\node at (3.75,-1.5){$T-(b+1)$}%
			child{node [dummy] {}%
				child[level distance = 1.2em,sibling distance=2.8em]{node [dummy]  {}%
				edge from parent node [swap] {$c+1$}}%
				child[level distance = 2.5em,sibling distance=3em]{node {} edge from parent node [swap,very near end] {$a+3$}}%
				child[level distance = 3em]{node {} edge from parent node[swap,very near end] {$a+1$}}%
				child[sibling distance=1.5em,level distance = 2.5em]{node [dummy] {}%
					child[level distance = 2.3em]{node {} edge from parent node [swap,near end] {$a+2$}}%
					child[level distance = 2.3em]{node {} edge from parent node [near end] {$\phantom{1+}a$}}%
				edge from parent node [near end] {$b$}}%
				child[level distance = 1.2em,sibling distance=2.8em]{node [dummy] {}%
				edge from parent node {$c\phantom{||}$} }%
			edge from parent node [swap] {$d$}};%
		\node at (0,-3){$T - \{b,b+1\}$}%
			child{node [dummy] {}%
				child[level distance = 1.2em,sibling distance=2.25em]{node [dummy]  {}%
				edge from parent node [swap] {$c+1$}}%
				child[level distance = 3.5em,sibling distance=2.75em]{node {} edge from parent node [swap,near end] {$a+3$}}%
				child[level distance = 4.5em,sibling distance=2em]{node {} edge from parent node[swap,very near end] {$a+1$}}%
				child[level distance = 4.5em,sibling distance=2em]{node {} edge from parent node[very near end] {$a+2$}}%
				child[level distance = 3.5em,sibling distance=2.75em]{node {} edge from parent node [near end] {$\phantom{1+}a$}}%
				child[level distance = 1.2em,sibling distance=2.25em]{node [dummy] {}%
				edge from parent node {$c\phantom{||}$} }%
			edge from parent node [swap] {$d$}};%
	\end{tikzpicture}%
\]%
We recall (cf. Remark \ref{REPEV RMK}) that since presheaves $X \in \mathsf{dSet}^G$ are only evaluated on non-equivariant trees $U \in \Omega$, the faces above 
are merely non-equivariant faces of the equivariant tree $T$: indeed, $T - b$ and $T - (b+1)$ do not admit a full compatible $G$-action. 
Rather, $G$ acts instead on the set of such faces and since 
$T - b$ admits a compatible $H=2G$-action and one can think of the disjoint union 
$\left(T-b\right) \amalg \left(T-(b+1)\right)$ as the $G$-tree 
$G \cdot_H (T-b)$.
\end{example}

\begin{remark}
	An eventual goal of the project this work belongs to is to show that there is a Quillen equivalence
	\[
	W_!
		\colon
	\mathsf{dSet}^G
		\rightleftarrows
	\mathsf{sOp}^G
		\colon
	hcN_d,
\]
generalizing \cite[Thm. 8.15]{CM13b} (where $\mathsf{sOp}^G$ is the category of $G$-equivariant colored simplicial operads and $hcN_d$ is the dendroidal homotopy coherent nerve).

While the proof of such a result is work in progress, and some of the best evidence in that direction is the subject of a current parallel write-up making Remark \ref{TWODSETG REM} precise, we provide here a first piece of evidence by generalizing \cite[Thm. 7.1]{MW08}.
\end{remark}

\begin{proposition}
Suppose that $\O \in \mathsf{sOp}^G$ is locally $G$-graph fibrant, i.e., that
$\O(\underline{a};b)^{\Gamma}$ is fibrant whenever 
$\Gamma \leq G \times \Sigma_{\underline{a}}$ stabilizes $\underline{a}$, $b$ and satisfies $\Gamma \cap \Sigma_{\underline{a}}=\**$.

Then $hcN_d(\O)$ is a $G$-$\infty$-operad.
\end{proposition}

\begin{proof} Since any $G$-tree has the form $G\cdot_H T$ for some $T \in \Omega^H$, it suffices after unpacking adjunctions to solve all the lifting problems as on the left below.
\begin{equation}\label{HCNDPROP EQ}
\begin{tikzcd}
W_{!}\Lambda^{H e}[T] \ar{r}{f} \ar{d} & \O &
W_{!}\Lambda^{H e}[T](\underline{l};r) \ar{r} \ar{d} & \O(f(\underline{l});f(r))
	\\
W_{!} \Omega[T] \ar[dashed]{ru} & &
W_{!} \Omega[T] \ar[dashed]{ru}(\underline{l};r) &
\end{tikzcd}
\end{equation}
Repeating the argument in the proof of \cite[Thm. 7.1]{MW08}, it suffices to build this lift for the mapping spaces between the leaves $\underline{l}$ and root $r$ of $T$. I.e., one needs only solve the rightmost lifting problem in (\ref{HCNDPROP EQ}), which needs only be equivariant with respect to the subgroup $\Gamma \leq H \times \Sigma_{\underline{l}}$ encoding the $H$-set $\underline{l}$ (crucially, note that operadic compatibility is automatic). Since the condition on $\O$ guarantees that $\O(f(\underline{l});f(r))$ is genuinely 
$\Gamma$-fibrant and the vertical map on the right side of 
(\ref{HCNDPROP EQ}) is (generalizing the formula in the proof of \cite[Prop. 4.5]{CM13b})
\[
\left(\{1\} \to \Delta[1]\right)^{\square He}
	\square
\left(\partial \Delta[1] \to \Delta[1]\right)^{\square Inn(T)- He},
\]
which is a genuine $\Gamma$-trivial cofibration, the result follows.
\end{proof}

We will develop for $G$-inner horns most of the key results of 
\cite{MW08} and \cite{CM11}, the proofs of which turn out to need only moderate modifications in order to generalize to the equivariant context. 

The hardest of those results, concerning the tensor product, will be the subject of \S \ref{TENSORPROD SEC}. To finish this section, we collect a couple of easier results, starting with the analogue of \cite[Lemma 5.1]{MW08}.

\begin{proposition}\label{GENERALIZED_HORN_ANODYNE_PROP}
Let $T \in \Omega_G$ be a $G$-equivariant tree and $E$ a $G$-equivariant subset of the inner edges of $T$. Then the generalized $G$-horn inclusion
\[\Lambda^{E}[T] \to \Omega[T]\]
is $G$-inner anodyne.
\end{proposition}

\begin{proof}
Since $E$ consists of a union of edge orbits, one immediately reduces to proving that maps of the form
\[\Lambda^{E}[T] \to \Lambda^{E-G e}[T]\]
are $G$-inner anodyne. 
In the non-equivariant case \cite[Lemma 5.1]{MW08} such maps can be described as single pushouts, but 
here we require multiple pushouts, naturally indexed by an equivariant poset which we now describe.

Firstly, let $T_e$ denote the (non-equivariant) tree component containing the edge $e$ and set $H \leq G$ to be its isotropy, resulting in a canonical identification $G \cdot_{H} T_e \simeq T$. 
Writing $\mathsf{Inn}_{He}(T_e)$ for the $H$-poset (under inclusion) of the inner faces of $T_e$ collapsing only edges in $He$, it suffices to check that, for any $H$-equivariant convex\footnote{
Recall that a subset $B \subset \mathcal{P}$ of a poset $\mathcal{P}$ is called \textit{convex}
if $\bar{b} \leq b$ and $b \in B$ implies $\bar{b} \in B$.}
subsets $B \subset B' \subset \mathsf{Inn}_{He}(T_e)$ it is 
\begin{equation}\label{FILT EQ}
	\Lambda^{E}[T] \cup  \bigcup_{g\in G,U \in B} \Omega[g U] \to
	\Lambda^{E}[T] \cup  \bigcup_{g \in G, U \in B'} \Omega[gU]
\end{equation}
$G$-inner anodyne.
Without loss of generality, we may assume $B'$ is obtained from $B$ by adding a single orbit $H V$ and, setting $\bar{H} \leq H$ to be the isotropy of $V$ in $\mathsf{Inn}_{H e}(T_e)$, we claim that (\ref{FILT EQ}) is a pushout of
\[G \cdot_{\bar{H}} \left( \Lambda^{E_e-He} [V] \to \Omega[V] \right),\]
where $E_e = E \cap T_e$ denotes the subset of inner edges of $T_e$ that are in $E$.
This claim is straightforward except for the following: one needs to note that the $G$-isotropy of any faces in $\mathsf{Inn}_{E_e - He}(V)$ (i.e. those faces missing from $\Lambda^{E_e - He}[V]$) is indeed contained in $\bar{H}$, and this follows since 
$\bar{H}$ can also be described as the subgroup of $G$ sending the edge subset $He \cap V = Ge \cap V$ to itself.

This concludes the proof by nested induction on the order of $G$ and the number of $G$-orbits of $E$.
\end{proof}

The following is the equivariant analogue of \cite[Lemma 5.2]{MW08}. Note that edge orbits of a $G$-tree $T$ are encoded by maps 
$G/H\cdot \eta \xrightarrow{G/H \cdot e} T$
for some $H \leq G$.

\begin{proposition}\label{GRAFTANODYNE PROP}
	Suppose that $T$ has a leaf orbit and $U$ root orbit both isomorphic to $G/H$. Write $V = T\amalg_{G/H \cdot \eta} U$ for the grafted $G$-tree.
	
	Then
\begin{equation}\label{GRAFTANODYNE EQ}
	\Omega[T] \amalg_{\Omega[G/H \cdot \eta]} \Omega[U] \to \Omega[V]
\end{equation}
	is inner $G$-anodyne.
\end{proposition}

\begin{proof}
Let $\mathsf{Out}(V)$ denote the $G$-poset of outer faces (which  have the form $V_{\nless \underline{e}}^{\leq e}$, and are hence not inner faces of any other face) of the grafted tree $V$, and $\mathsf{Out}_{\not\subset T,U}(V)$ the $G$-subposet of those outer faces contained in neither $T$ nor $U$.

It suffices to show that for all $G$-equivariant convex subsets $B \subset B'$ of $\mathsf{Out}_{\not\subset T,U}(V)$
it is
\begin{equation} \label{WHATEVER EQ}
	\Omega[T] \cup  \Omega[U] \cup \bigcup_{R \in B} R
		\to
	\Omega[T] \cup \Omega[U]  \cup \bigcup_{R \in B'} R
\end{equation}
$G$-inner anodyne 
(indeed, (\ref{WHATEVER EQ}) recovers (\ref{GRAFTANODYNE EQ}) when $B=\emptyset$, $B'=\mathsf{Out}_{\not\subset T,U}(V)$). 

As before, we can assume $B'$ is obtained from $B$ by adding the orbit $G S$ of a single outer face $S$. Letting $H \leq G$ denote the isotropy of $S$, 
one has that (\ref{WHATEVER EQ}) is then the pushout (note that the $G$-isotropy of an \textit{inner} face of the \textit{outer} face $S$ is at most $H$) of
\[G \cdot_H \left( \Lambda^{Inn(S)}[S] \to \Omega[S]\right),\]
finishing the proof.
\end{proof}

\begin{remark}
A key difference between the proofs of
Propositions \ref{GENERALIZED_HORN_ANODYNE_PROP} and
\ref{GRAFTANODYNE PROP}
versus their non-equivariant analogues
is the need to check that the isotropies are correct when attaching equivariant horns.
\end{remark}

\section{Tensor products}
\label{TENSORPROD SEC}

Our goal in this section is to prove equivariant analogues of \cite[Prop. 9.2]{MW08}, \cite[Thm. 5.2]{CM11} and \cite[Thm. 4.2]{CM11}, which are the key technical results in their respective papers. These results concern the interaction of anodyne extensions with the tensor product and the join constructions, which we recall in \S \ref{TENSORDEF SEC} and \S \ref{DENDJOIN SEC}.
We now present our versions of the results, starting with the analogue of \cite[Prop. 9.2]{MW08}.

\begin{theorem}\label{EXPNPROP THM}
	Let $S,T \in \Omega_G$ be $G$-trees and let $G\xi$ be an inner orbit edge of $T$. Then the map
	\[\partial \Omega[S] \otimes \Omega[T] 
			\coprod_{\partial \Omega[S] \otimes \Lambda^{G \xi}[T]}
		\Omega[S] \otimes \Lambda^{G \xi}[T] 
			\to
		\Omega[S] \otimes \Omega[T]
	\]
is a $G$-inner anodyne extension if either
\begin{itemize}
	\item[(i)] both $S$ and $T$ are open $G$-trees (i.e. have no stumps);
	\item[(ii)] at least one of $S,T$ is a linear $G$-tree (i.e. isomorphic to $G\cdot_H [n]$ for $[n] \in \Delta$).
\end{itemize}
\end{theorem}
The proof of Theorem \ref{EXPNPROP THM} will be the subject of \S \ref{GANODYNEPROOF SEC}. More specifically, the result will follow from Proposition \ref{EXPG PROP} when $B=\emptyset$ 
and $B'=\mathsf{IE}_{G \xi}(S \otimes T)$.

The following is the equivariant analogue of \cite[Thm. 5.2]{CM11}.

\begin{theorem}\label{OUTERIN THM}
	Let $S \in \Omega^G$ be a tree with a $G$-action such that $S \neq \eta$ and denote by 
	$A \to \Omega[S] \otimes \Omega[1]$ the pushout product map
	\[\partial \Omega[S] \otimes \Omega[1] 
			\coprod_{\partial \Omega[S] \otimes \{1\}}
		\Omega[S] \otimes \{1\}
			\to
		\Omega[S] \otimes \Omega[1].
	\]
	Then there is a factorization $A \to B \to \Omega[S] \otimes \Omega[1]$ such that
	\begin{itemize}
		\item[(a)] $A \to B$ is a $G$-inner anodyne extension;
		\item[(b)] there is a pushout (the join $S \star \eta$ is introduced in Definition \ref{JOIN DEF})
			\begin{equation}\label{OUTPUSH EQ}
				\begin{tikzcd}
				\Lambda^{\eta}[S \star \eta] \ar[r] \ar[d] 	& B \ar[d] \\
				\Omega[S \star \eta]                \ar[r] 	& \Omega[S] \otimes \Omega[1]; 
				\end{tikzcd}
			\end{equation}
		\item[(c)] letting $\eta \xrightarrow{r} S$ denote the root edge, the composite
			\[\Omega[1] \simeq \Omega[\eta \star \eta] 
					\xrightarrow{r \star id} 
				\Omega[S \star \eta] \to \Omega[S] \otimes \Omega[1]\]
			coincides with the composite 
			\[
				\Omega[1] \simeq \eta \otimes \Omega[1] 
			  \xrightarrow{r \otimes id}
				\Omega[S] \otimes \Omega[1].
				\]
	\end{itemize}
\end{theorem}
Theorem \ref{OUTERIN THM} will be proven at the end of \S \ref{GANODYNEPROOF SEC} as a direct consequence of the arguments used in the proof of Proposition \ref{EXPG PROP}.
We note that $\Lambda^{\eta}[S \star \eta]$ is an \textit{outer} horn, lacking only the outer face $(S \star \eta) - \eta = (S \star \eta)^{\leq r} \simeq S$.

The following is the equivariant analogue of \cite[Thm. 4.2]{CM11}. Note that we use the notation 
$i \colon \Delta \to \Omega$
for the inclusion and 
$i^{\**}\colon \mathsf{dSet} \to \mathsf{sSet}$
for the restriction.

\begin{theorem}\label{JOINLIFT THM}
	Let $S \in \Omega^G$ be a tree with a $G$-action. Assume further that $S$ has at least two vertices and unary root vertex 
	$G/G \cdot \Omega[1] \xrightarrow{v_r} S$. Then a lift exists in any commutative diagram
	\begin{equation}\label{JOINLIFT EQ}
		\begin{tikzcd}
		\Lambda^r[S] \ar{r}{f} \ar[d] & X \ar[d]\\
		\Omega[S] \ar[r] \ar[dashed]{ur}   & Y
		\end{tikzcd}
	\end{equation}
such that $X \to Y$ is a $G$-inner fibration between $G$-$\infty$-operads and $f(v_r)$ is an equivalence in the $\infty$-category $i^{\**}(X^G)$.
\end{theorem}
We will prove Theorem \ref{JOINLIFT THM} at the end of \S \ref{DENDJOIN SEC}.

\subsection{Tensor product} \label{TENSORDEF SEC}

To keep the proofs of Theorems \ref{EXPNPROP THM} and \ref{OUTERIN THM} compact we will prefer to use broad poset language throughout. We start by defining
tensor products in this framework.

Given $\underline{s} \in S^+$, $\underline{t} \in T^+$, we will let 
$\underline{s} \times \underline{t} \in (S \times T)^{+}$ denote the obvious tuple whose elements 
$(s,t) \in \underline{s} \times \underline{t}$
are those pairs with $s \in \underline{s}$, $t \in \underline{t}$.

\begin{definition}
  Given pre-broad posets $S$, $T$ (cf. Remark \ref{PREBROAD REM}), their \textit{tensor product} $S \otimes T$ is the pre-broad poset whose underlying set is $S \times T$ and whose relations are generated by relations of the form
	$\underline{s} \times t \leq (s,t)$ 
(resp. $s \times \underline{t} \leq (s,t)$)
	for $s \in S$, $t \in T$ and 
$\underline{s} \leq s$ (resp. $\underline{t} \leq t$) 
a broad relation in $S$ (resp. T).
\end{definition}

\begin{proposition}\label{TENSORCHILDREN UNOPEN PROP1}
If $S,T$ are simple broad posets then so is $S \otimes T$.
Further, for any non identity broad relation
$(s_1,t_1)\cdots (s_n,t_n) \leq (s,t)$ in $S \otimes T$
one has:
\begin{itemize}
	\item[(i)] for all $0 \leq i,j \leq n$ it is $s_i\leq_d s$, $t_j\leq_d t$;
	\item[(ii)] for $i \neq j$ either there exists
	$\underline{s}$ such that $s_i s_j \underline{s} \leq s$ 
	or there exists 
	$\underline{t}$ such that $t_i t_j \underline{t} \leq t$
	(or both);
	\item[(iii)] if both of the pairs $s_i, s_j$ and $t_i, t_j$ are $\leq_d$-comparable then $i=j$.
\end{itemize}
\end{proposition}

\begin{proof}
Note first that the ``further'' conditions suffice to check that $S\otimes T$ is a simple broad poset, i.e. that they imply antisymmetry and simplicity. Indeed, antisymmetry follows from combining (i) with the fact that $\leq_d^S$, $\leq_d^T$
are order relations
while simplicity is a particular case of (iii).

(i) follows since the condition holds for generating relations and is preserved by transitivity.
Similarly, (ii) holds for generating relations and is readily  seen to be preserved by transitivity when applied to relations satisfying (i). Lastly, (iii) follows from (ii)
and the $\leq_d$-incomparability result in Proposition \ref{SIMPLEBROAD PROP}. 
\end{proof}

\begin{remark}
	The main claim in Proposition \ref{TENSORCHILDREN UNOPEN PROP1} fails for non simple broad posets.
	As an example, let $S$ be the broad poset $\{a,b\}$ with generating relations $ab\leq a$, $ab \leq b$ 
	(antisymmetry holds since $\leq$ decreases the size of tuples)
	and $T$ be $\{c\}$ with generating relation $\epsilon \leq c$. 
	Then $(a,c) \leq (b,c) \leq (a,c)$ hold in $S \otimes T$.
\end{remark}

\begin{proposition}\label{TENSORCHILDREN UNOPEN PROP2}
Let $S,T$ be trees. An edge $(s,t) \in S \otimes T$ has one of five types:
  \begin{itemize}
	 \item[(leaf)] it is a leaf if both $s \in S$, $t \in T$ are leaves;
	 \item[(stump)] it is a stump if $s \in S$ is a leaf and $t \in T$ is a stump or vice versa, or if both $s \in S$, $t \in T$ are stumps;
	 \item[(leaf node)] it is a node if $s \in S$ is a node and $t \in T$ is a leaf or vice versa. In fact 
	$(s,t)^{\uparrow}=s^{\uparrow} \times t$ or 
	$(s,t)^{\uparrow}=s \times t^{\uparrow}$, accordingly;
	 \item[(null node)] it is a node such that $\epsilon \leq (s,t)$ if $s \in S$ is a node and $t \in T$ a stump or vice versa. In fact 
	$(s,t)^{\uparrow}=s^{\uparrow} \times t$ or 
	$(s,t)^{\uparrow}=s \times t^{\uparrow}$, accordingly;
 	 \item[(fork)] if $s\in S$, $t \in T$ are both nodes then there are exactly two maximal $\underline{f}$ such that $\underline{f} < (s,t)$, namely $s \times t^{\uparrow}$ and $s^{\uparrow} \times t$. We call such $(s,t)$ a \textit{fork}.
	\end{itemize}
\end{proposition}

\begin{proof}
%

Only the fork case requires proof. In fact, it is tautological that only $s \times t^{\uparrow}$ and $s^{\uparrow} \times t$ can possibly be maximal, hence one needs only verify that neither 
$s \times t^{\uparrow} \leq s^{\uparrow} \times t$ 
nor 
$s^{\uparrow} \times t \leq s \times t^{\uparrow}$. This follows since the $S$ coordinates of the pairs in the tuple $s^{\uparrow} \times t$ are $<_d$ than those in $s \times t^{\uparrow}$ and vice versa.
 \end{proof}

In order to simplify notation, we will henceforth write  
$e^{\uparrow S} = s^{\uparrow} \times t$, 
$e^{\uparrow T} = s \times t^{\uparrow}$ and 
$e^{\uparrow S,T} = s^{\uparrow} \times t^{\uparrow}$.

\begin{proposition} \label{NODEFACTCHAR PROP}
Let $S$, $T$ be trees and consider the broad relation 
\[\underline{e} = (s_1,t_1)(s_2,t_2) \cdots (s_n,t_n) \leq (s,t) = e.\]
in $S \otimes T$. Then:
\begin{itemize}
\item[(i)]   
$\underline{e} \leq e^{\uparrow S}$ (resp. $\underline{e} \leq e^{\uparrow T})$ if and only if $s_i \neq s, \forall i$ (resp. $t_i \neq t, \forall i$);
\item[(ii)] $\underline{e} \leq e^{\uparrow S,T}$ if and only if both $s_i \neq s$ and $t_i \neq t, \forall i$.
\end{itemize}
\end{proposition}

\begin{proof}
Only the ``if'' directions need proof, and the proof follows by upward $\leq_d$ induction on $s,t$. The base cases of either $s$ or $t$ a leaf are obvious.

Otherwise, let $\underline{e}$ satisfy the ``if'' condition in (i). Since it must be either 
$\underline{e} \leq e^{\uparrow S}$ or 
$\underline{e} \leq e^{\uparrow T}$
we can assume it is the latter case. 
Writing
$t^{\uparrow} = u_1 \cdots u_k$, 
$\underline{e} = \underline{e}_1 \cdots \underline{e}_k$ 
so that $\underline{e}_i \leq (s,u_i)$ (note that possibly $k=0$), the induction hypothesis now yields 
$\underline{e}_i \leq (s,u_i)^{\uparrow S} = s^{\uparrow} \times u_i$, 
and hence
\begin{equation}\label{INEQUALS EQ}
 \underline{e} = 
\underline{e}_1 \cdots \underline{e}_k \leq (s^{\uparrow} \times u_1) \cdots (s^{\uparrow} \times u_k) = 
 s^{\uparrow} \times t^{\uparrow} \leq s^{\uparrow} \times t = e^{\uparrow S}.
\end{equation}
The proof of (ii) simply disregards the last inequality in (\ref{INEQUALS EQ}).
\end{proof}

\begin{corollary}
$\underline{e} \leq e^{\uparrow S,T}$ 
 if and only if both 
$\underline{e} \leq e^{\uparrow S}$  and 
$\underline{e} \leq e^{\uparrow T}$.
\end{corollary}

\begin{lemma}\label{INEQMAXIFF NOPENTENSOR LEM}
Let $S$, $T$ be trees. For any $e = (s,t) \in S \otimes T$ there exists a minimum $e^{\lambda} \in (S \times T)^+$ such that $e^{\lambda} \leq e$. In fact, $e^{\lambda} = s^{\lambda} \times t^{\lambda}$.
\end{lemma}

\begin{proof}
The proof is by $\leq_d$ induction on $e$. The case of $s$, $t$ both leaves is obvious. Otherwise, for any non-identity relation either
$\underline{f} \leq e^{\uparrow S} < e$
or 
$\underline{f} \leq e^{\uparrow T} < e$,
and the analysis in the proof of Lemma \ref{INEQMAXIFF NOPEN LEM} applies in either case to show that indeed $s^{\lambda} \times t^{\lambda} \leq \underline{f}$.
\end{proof}

\subsection{Subtrees}

\begin{definition}
 Let $S$, $T$ be trees.
A \textit{subtree} of $S \otimes T$ is a tree $U$ together with a broad poset map $U \hookrightarrow S \otimes T$ that is an underlying monomorphism. Further, a subtree is called \textit{full} if the relations in $U$ coincide with those in its image and in that case we instead write $U \subset S \otimes T$.
\end{definition}

 We have the following characterization.

\begin{proposition}\label{CLOSEDSUBTREE PROP}
$U \overset{\varphi}{\hookrightarrow} S \otimes T$ is full iff for each leaf $l \in U$ it is not $\epsilon \leq \varphi(l)$.
\end{proposition}

\begin{proof}
To simplify notation we will simply write $u$ for both an edge $u \in U$ and its image 
$\varphi(u) \in S \otimes T$ and instead decorate the broad relations as $\leq^U$, $\leq^{S \otimes T}$, and similarly write $u^{\lambda,U}$, $u^{\lambda,S \otimes T}$
following Lemmas \ref{INEQMAXIFF NOPEN LEM} and \ref{INEQMAXIFF NOPENTENSOR LEM}.

We need to show that a broad relation 
$\underline{u} = u_1 \cdots u_n \leq^{S \otimes T} u$ 
can not fail the conditions in Lemma \ref{INEQMAXIFF NOPEN LEM} with respect to $U$. Further, we note that, since $S \otimes T$ has a $\leq_d$-maximal element $(r_S,r_T)$, the proof of Lemma \ref{LEQDCOMP LEM} applies to show that the two $\leq_d$-relations on $\varphi(U)$ coincide. Thus, only condition (iii) of Lemma \ref{INEQMAXIFF NOPEN LEM} could possibly fail, and this would happen only if $u_1^{\lambda,U} \cdots u_n^{\lambda,U}$ lacked some of the leaves in $u^{\lambda,U}$. But our hypothesis is that $l^{\lambda,S \otimes T} \neq \epsilon$ for $l$ any leaf of $U$, hence this is impossible.
\end{proof}



\begin{definition}\label{ELEMTENS DEF}
Let $S$, $T$ be trees. A subtree $U \hookrightarrow S \otimes T$ is called
\begin{itemize}
 \item \textit{elementary} if all of its generating broad relations are 
of the form $e^{\uparrow S} \leq e$ or  $e^{\uparrow T} \leq e$;
 \item \textit{initial} if $U$ contains the ``double root'' 
$(r_S,r_T) \in S \otimes T$.
\end{itemize}
Further, a \textit{maximal elementary subtree} is a subtree that is not contained in any other. 
Note that, since one can graft new root vertices to $U$,  Proposition \ref{MAPOUTTREE PROP} implies that maximal elementary  subtrees are necessarily initial.
\end{definition}

\begin{remark} Maximal elementary subtrees are called \textit{percolation schemes} in \cite[Example 9.4]{MW08}.
\end{remark}

\begin{remark}\label{ELEMINNERFACE REM}
As noted in Proposition \ref{TENSORCHILDREN UNOPEN PROP2}, the relation $\epsilon \leq e$ is a decomposable relation of $S \otimes T$ whenever $e$ is a null node. As a consequence, any elementary tree containing such a generating relation is in fact an inner face of a larger elementary tree.  
\end{remark}


\begin{lemma}\label{CLOSEDSUBTREE LEM}
If $U$ is a full face of $S$, then $U \otimes T \subset S \otimes T$, i.e., $U \otimes T$ contains all broad relations in its image.
\end{lemma}

\begin{proof}
Given a relation $\underline{e}=(s_1,t_1)\cdots(s_n,t_n)\leq (s,t) =e$
with all $s_i \in U$ (but not necessarily $s \in U$), 
we first claim that there is a factorization 
$\underline{e} \leq \underline{u} \times t \leq (s,t)$
with $\underline{u}$ in $U$ and $\underline{u} \leq s$ in $S$.
If $s \in U$ one simply takes $\underline{u}=s$. Otherwise
Proposition \ref{NODEFACTCHAR PROP} yields
$\underline{e} \leq e^{\uparrow S} \leq e$ and the claim follows by 
$\leq_d$ induction on $e$.

To check the desired claim that $\underline{e}\leq e$ will be in $U \otimes T$ if $s_i,s\in U$, we again argue by $\leq_d$ induction on $e$, with the case $\underline{e}\leq e^{\uparrow T}\leq e$ being immediate and the case $\underline{e}\leq e^{\uparrow S}\leq e$ following by the result in the  previous paragraph.
%
%
\end{proof}

\begin{definition}
Let $S$, $T$ be trees and 
$A=\{a_i\}$, $B=\{b_j\}$ subsets of the sets of stumps of $S$, $T$, respectively, and let 
$v_A=\{\epsilon \leq a_i\}$, $v_B=\{\epsilon \leq b_j\}$
denote the corresponding vertices.

We say that a subtree $U \hookrightarrow S \otimes T$ \textit{misses $v_A$ and $v_B$} 
if one has a factorization 
$U \hookrightarrow (S - v_A) \otimes (T - v_B) \hookrightarrow S \otimes T$.

Further, if $B = \emptyset$ (resp. $A = \emptyset$) we say simply that ``$U$ misses $v_A$'' (resp. ``$U$ misses $v_B$'').
\end{definition}

\begin{remark}
In \cite{MW08} similar notions of ``$U$ missing an inner edge/leaf vertex are also defined, but we note that 
(due to Lemma \ref{CLOSEDSUBTREE LEM})
 those notions are far more straightforward. In fact, as explained in the erratum to the follow up paper \cite{CM11}, the earlier treatment overlooked some subtle properties of stumps.

For instance, note that Corollary \ref{ROOTCOND COR} below
implies that the notion 
``$U$ misses $v_A$ and $v_B$'' does not coincide with the notion 
``$U$ misses $v_A$ and $U$ misses $v_B$'' whenever $A$ and $B$ are both non-empty. 
\end{remark}

\begin{lemma}\label{BROADREL STUMPS}
Let $S$, $T$ be trees and 
$A=\{a_i\}$, $B=\{b_j\}$
subsets of the stumps of $S$, $T$, respectively.
Then a broad relation 
\[
 \underline{f} = f_1 f_2 \cdots f_k \leq e
\]
in $S \otimes T$ is a broad relation in $(S-v_A) \otimes (T-v_B)$
if and only if 
\[e^{\lambda , (S - v_A) \otimes (T - v_B)} \leq \underline{f}.\]
\end{lemma}

\begin{proof}
Only the ``if'' direction needs proof. 
We argue by $\leq_d$ induction on $e = (s,t)$.
The base case, that of $s$ a leaf of $S-v_A$ and $T$ a leaf of $T-v_B$, is obvious (we note that the proof will follow even when this case is vacuous).

Otherwise, either $\underline{f} \leq e^{\uparrow S} \leq e$ or 
$\underline{f} \leq e^{\uparrow T} \leq e$ and our assumption ensures, respectively, that $s \not\in A$ or $t \not\in B$. 
Writing $e^{\uparrow \**}$ to denote either $e^{\uparrow S}$ or $e^{\uparrow T}$ as appropriate, this last observation guarantees that the relation $e^{\uparrow \**} \leq e$ is in $(S - v_A) \otimes (T - v_B)$. 
Further, writing 
$\underline{f}=\underline{f}_1 \cdots \underline{f}_k$
and 
$e^{\uparrow \**} = e_1 \cdots e_k$
so that $\underline{f}_i \leq e_i$,
the induction hypothesis shows that these last relations are also in  
$(S - v_A) \otimes (T - v_B)$.
\end{proof}

Recalling Proposition \ref{INDECOMPUNIQUE PROP} hence yields the following.

\begin{corollary}\label{FACTORIZATION COR}
A collection of broad relations of the form $\underline{g}_i \leq f_i$, $f_1 \cdots f_k \leq e$ are all in $(S-v_A) \otimes (T-v_B)$ if and only if the composite relation $\underline{g}_1 \cdots \underline{g}_k \leq e$ is.
\end{corollary}

\begin{corollary}\label{ROOTCOND COR}
 A subtree $U \hookrightarrow S \otimes T$ misses $v_A$ and $v_B$  
if and only if, for $r$ the root of $U$, one has
\[r^{\lambda, (S-v_A) \otimes (T-v_B)} \leq r^{\lambda,U}.\]
\end{corollary}

\begin{proof}
This follows from Corollary \ref{FACTORIZATION COR} 
since any generating relation in $U$ is a factor of $r^{\lambda,U} \leq r$.
\end{proof}

The following was first stated in \cite[Prop 1.9 in erratum]{CM11} and proven via a careful combinatorial analysis in   \cite{CM14}. We include here a short broad poset proof.
Recall that $\mathsf{dSet} \times \mathsf{dSet} \xrightarrow{\otimes} \mathsf{dSet}$
is defined by setting $\Omega[S] \otimes \Omega[T] =
Hom(-,S\otimes T)$ together with the requirement that
$\otimes$ commutes with colimits in each variable.

\begin{proposition}\label{MONOPUSHOUTPROD PROP}
Let $S, T \in \Omega$ be trees which are either
\begin{inparaenum}
\item[(i)] both open;
\item[(ii)] $S=[n]$ is linear.
\end{inparaenum}
Then the square
\begin{equation}\label{MONOMORSQUARE EQ}
\begin{tikzcd}[row sep=1.3em,column sep=1.3em]
	\partial \Omega[S] \otimes \partial \Omega[T] \ar[hook]{r} \ar[hook]{d} &
	\Omega[S] \otimes \partial \Omega[T] \ar[hook]{d}
\\
	\partial \Omega[S] \otimes \Omega[T] \ar[hook]{r} &
	 \Omega[S] \otimes \Omega[T]
\end{tikzcd}
\end{equation}
consists of normal monomorphisms. Further, 
\begin{equation}\label{MONOPUSHOUT EQ}
\partial \Omega[S] \otimes \Omega[T] 
   \underset{\partial \Omega[S] \otimes \partial \Omega[T]}{\coprod}
	\Omega[S] \otimes \partial \Omega[T]
	 \hookrightarrow
	\Omega[S] \otimes \Omega[T]
\end{equation}
is also a normal monomorphism.

In particular, (\ref{MONOMORSQUARE EQ}) is a pullback square.
\end{proposition}

\begin{proof}
Note first that since $\otimes$ commutes with colimits in each variable, 
\begin{equation}\label{BOUNDARYTENSOR EQ}
\partial \Omega[S] \otimes \Omega[T] = 
\colim_{F \in \mathsf{Faces}(S)-\{S\}} \Omega[F] \otimes \Omega[T].
\end{equation}
In the open case, since all faces are full, 
Lemma \ref{CLOSEDSUBTREE LEM} implies that if
$U \subset F_S \otimes T$ and $U \subset S \otimes F_T$ then it will be $U \subset F_S \otimes F_T$, showing that 
(\ref{MONOPUSHOUT EQ}) is a monomorphism whenever 
(\ref{MONOMORSQUARE EQ}) consists of monomorphisms. By skeletal induction, it thus suffices to check that the right and bottom maps in (\ref{MONOMORSQUARE EQ}) are monomorphisms, and this will follow if for any $U \subset S \otimes T$ there exists a minimal face $F \subset S$ such that $U \subset F \otimes T$. Clearly $F = \{s|\exists_t (s,t) \in U\}$ will work once we show that this is indeed a face. This is clear when $U$ is elementary
(in which case each vertex of $U$ either adds a vertex to $F$ or nothing at all) and holds in general since any $U$ is an inner face of an elementary subtree.

In the $S = [n]$ linear case the fact that $S$ is open still suffices to conclude that (\ref{MONOPUSHOUT EQ}) being a monomorphism will follow once we show that
(\ref{MONOMORSQUARE EQ}) consists of monomorphisms.
And, similarly, that 
$\partial \Omega[n] \otimes \Omega[T] \to 
\Omega[n] \otimes \Omega[T]$ is a monomorphism follows by the same argument, building $F$ in the same way.
It remains to show that $ \Omega[n] \otimes \partial \Omega[T] \to 
\Omega[n] \otimes \Omega[T]$
is a monomorphism. We note that the projection 
$\pi \colon [n]\otimes T \to T$ given by $\pi(k,e)=e$ is a map of broad posets. Given $U \hookrightarrow [n] \otimes T$ we claim that $\pi(U)$ is the minimal face such that 
$U \hookrightarrow [n] \otimes \pi(U)$, noting that this is implied by the more general claim that $U \hookrightarrow [n] \otimes F$ iff $\pi(U) \hookrightarrow F$. It now suffices to check this when $F$ is a maximal face, with the case of $F$ full being obvious from Lemma \ref{CLOSEDSUBTREE LEM} and the stump outer face case following from
Corollary \ref{ROOTCOND COR}.

The claim that (\ref{MONOMORSQUARE EQ}) is a pullback square is elementary
(compare with the proof of \cite[Prop. 1.9]{CM11}).
\end{proof}

\subsection{Pushout product filtrations}\label{GANODYNEPROOF SEC}

This section features our main technical proofs, namely the proof of Theorem \ref{EXPNPROP THM} and the related but simpler proof of Theorem \ref{OUTERIN THM}.

The majority of the ideas in this section are adapted from the (rather long) proof of 
\cite[Prop. 9.2]{MW08}, but here we will need to significantly repackage those ideas. To explain why, we note that the filtrations in the proof of \cite[Prop. 9.2]{MW08} are actually divided into three nested tiers: an outermost tier described immediately following 
\cite[Cor. 9.3]{MW08}, an intermediate tier described in the proof of \cite[Lemma 9.9]{MW08} and an innermost tier described in the proof of \cite[Lemma 9.7]{MW08}. However, in the equivariant case $G$ acts transversely to these tiers, i.e. one can not attach dendrices at an inner tier stage without also attaching dendrices in a different outer tier stage.

Our solution will be to encode the top two filtration tiers as a poset 
$\mathsf{IE}_{G \xi}(S \otimes T)$ on which $G$ acts (to handle the lower tier). To improve readability, however, we first describe our repackaged proof in the non-equivariant case, then indicate the (by then minor) necessary equivariant modifications.

\vskip 5pt

We will make use of an order relation on elementary subtrees (cf. Definition \ref{ELEMTENS DEF})
 of $S \otimes T$. 
\begin{definition} \label{LEX DEF}
 Write $V \leq_{lex} U$
whenever $U$ is obtained from $V$ by replacing the intermediate edges in a string of broad relations 
$e^{\uparrow S,T} \leq e^{\uparrow S} \leq e$
occurring in $V$ with the intermediate edges in 
$e^{\uparrow S,T} \leq e^{\uparrow T} \leq e$
occurring in $U$. An illustrative diagram follows.
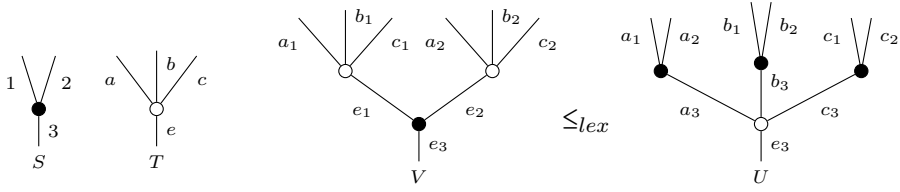
\[
	\begin{tikzpicture}[auto,grow=up, level distance = 2em, every node/.style={font=\scriptsize},dummy/.style = {circle,draw,inner sep=0pt,minimum size=1.75mm}]
		\node at (0,0.2) {$S$}
			child{node [dummy,fill=black]{}
				child[sibling distance = 1.25em]{
				edge from parent node [swap,near end] {$2$}}
				child[sibling distance = 1.25em]{
				edge from parent node [near end] {$1$}}
			edge from parent node [swap] {$3$}};
		\node at (1.55,0.2) {$T$}
			child{node [dummy]{}
				child[sibling distance = 1.5em]{
				edge from parent node [swap,near end] {$c$}}
				child[sibling distance = 2.25em,level distance = 2.2em]{
				edge from parent node [swap,near end] {$b$}}
				child[sibling distance = 1.5em]{
				edge from parent node [near end] {$a$}}
			edge from parent node [swap] {$e$}};
		\node at (5,0) {$V$}
			child{node [dummy,fill=black] {}
				child[sibling distance = 5.5em]{node [dummy] {}
					child[sibling distance = 1.75em]{
					edge from parent node [swap,near end] {$c_2$}}
					child[level distance = 2.3em]{
					edge from parent node [swap,very near end] {$b_2$}}
					child[sibling distance = 1.75em]{
					edge from parent node [near end] {$a_2$}}
				edge from parent node [swap] {$e_2$}}
				child[sibling distance = 5.5em]{node [dummy] {}
					child[sibling distance = 1.75em]{
					edge from parent node [swap,near end] {$c_1$}}
					child[level distance = 2.3em]{
					edge from parent node [swap,very near end] {$b_1$}}
					child[sibling distance = 1.75em]{
					edge from parent node [near end] {$a_1$}}				
				edge from parent node {$e_1$}}
			edge from parent node [swap] {$e_3$}};
		\node at (9.5,0) {$U$}
			child{node [dummy] {}
				child[sibling distance = 3.75em]{node [dummy,fill=black] {}
					child[sibling distance = 0.75em]{
					edge from parent node [swap,very near end] {$c_2$}}
					child[sibling distance = 0.75em]{
					edge from parent node [very near end] {$c_1$}}
				edge from parent node [swap] {$c_3$}}
				child[level distance = 2.3em]{node [dummy,fill=black] {}
					child[sibling distance = 0.75em]{
					edge from parent node [swap,very near end] {$b_2$}}
					child[sibling distance = 0.75em]{
					edge from parent node [very near end] {$b_1$}}
				edge from parent node [near end,swap] {$b_3$}}
				child[sibling distance = 3.75em]{node [dummy,fill=black] {}
					child[sibling distance = 0.75em]{
					edge from parent node [very near end,swap] {$a_2$}}
					child[sibling distance = 0.75em]{
					edge from parent node [very near end] {$a_1$}}				
				edge from parent node {$a_3$}}
			edge from parent node [swap] {$e_3$}};
		\node at (7.2,0.75) [font=\large] {$\leq_{lex}$};
	\end{tikzpicture}
\]
\end{definition}

\begin{remark}
	The $\leq_{lex}$ relation is compatible with the grafting procedure in Remark \ref{GENGRAFTDEF REM}. In particular, we will throughout assume that a relation $V \leq_{lex} U$ can be built by first ungrafting $U$ (Proposition \ref{UNGRAFT PROP}), then applying $\leq_{lex}$ relations to each piece, and finally regrafting the pieces to obtain $V$.
\end{remark}

In what follows we refer to  generating relations of the form 
$e^{\uparrow S} \leq e$ (resp. $e^{\uparrow T} \leq e$)
in an elementary subtree as $S$-vertices (resp. $T$-vertices). Also, given vertices $v=(e^{\uparrow *} \leq e)$, 
$w=(f^{\uparrow *} \leq f)$ we write $v\leq_d w$ if $e \leq_d f$.

\begin{proposition}\label{PARTIALORDER PROP}
Suppose $T$ is open (i.e. has no stumps). Then
$\leq_{lex}$ induces a partial order on the set of elementary subtrees of $S \otimes T$.

Further, $\leq_{lex}$ together with the inclusion $\hookrightarrow$ assemble into a partial order as well, and we denote this latter order simply by $\leq$.
\end{proposition}

\begin{proof}
One needs only check antisymmetry. Let $g(U)$ count pairs $(v_S,v_T)$ of a $S$-vertex and $T$-vertex in $U$ such that
$v_S \leq_d v_T$. 
Since the generating relations of $\leq_{lex}$ strictly increase $g$,
$\leq_{lex}$ is a partial order. 
Similarly, letting 
$h(U) = \# \{ \text{stumps of }U\} + 
 \sum_{l \in \{\text{leaves and stumps of }U\}}{d(l,r)}$ (where $d(-,r)$ denotes ``distance to the root'', measured in generating $\leq_d$ relations), $\hookrightarrow$ increases $h$ and 
$\leq_{lex}$ either 
\begin{enumerate*}
 \item[(i)] preserves $h$ if $e$ (as in Definition \ref{LEX DEF}) is a fork;
 \item[(ii)] is an instance of $\hookrightarrow$ if $e$ is a null node 
(since $T$ is assumed open).
\end{enumerate*}
Thus  $\leq$ is a partial order.
\end{proof}

\begin{remark}
Note that if $T$ is not open, then by Remark \ref{ELEMINNERFACE REM} it is possible for the combination of $\leq_{lex}$ and $\hookrightarrow$ to fail antisymmetry.
\end{remark}

Henceforth we will let $\xi$ denote a \textit{fixed} inner edge of $T$.

\begin{definition}\label{IEETA DEF}
An initial elementary subtree $U \hookrightarrow S \otimes T$
(cf. Definition \ref{ELEMTENS DEF})
 is called \textit{$\xi$-internal}
if it contains an edge of the form $(s,\xi)$, abbreviated as $\xi_s$, and the $T$-vertex $\xi_s^{\uparrow T} \leq \xi_s$.

For $T$ open, we will denote the subposet of such trees by $(\mathsf{IE}_{\xi}(S \otimes T),\leq)$.

Further, when $S$ is open one can modify the order in $\mathsf{IE}_{\xi}(S \otimes T)$ by reversing the $\leq_{lex}$ order (but not the 
$\hookrightarrow$ order). The resulting poset will be denoted $\mathsf{IE}_{\xi}^{oplex}(S \otimes T)$.
\end{definition}

\begin{lemma}\label{INTERCHANGE LEM} Suppose $T$ is open.
Let $U \hookrightarrow S \otimes T$ be an elementary subtree with root vertex a $T$-vertex $e^{\uparrow T} \leq e$ and suppose that 
$e^{\lambda, U} \leq e^{\uparrow S}$ (or, by Proposition \ref{NODEFACTCHAR PROP}, that none of the leaves in $e^{\lambda, U}$ have the same $S$ coordinate as $e$).

Then there exists an elementary subtree $V$ such that $V \leq_{lex} U$ and $V$ contains the relations 
$e^{\uparrow S,T} \leq e^{\uparrow S} \leq e$.
\end{lemma}

\begin{proof}
We argue by induction on the sum of the distances (cf. proof of Proposition \ref{PARTIALORDER PROP}) between the leaves and stumps of $U$ and its root $e$. 
The base case, that of $U$ the elementary tree generated by $e^{\uparrow S,T} \leq e^{\uparrow T} \leq e$, is obvious.

 Otherwise, writing $e=(s,t)$,
for each $t_i \in t^{\uparrow}$ either 
$(s,t_i)^{\uparrow S} \leq^U (s,t_i)$ or
$(s,t_i)^{\uparrow T} \leq^U (s,t_i)$. Applying the induction hypothesis 
to each of the subtrees $U^{\leq (s,t_i)}$ 
(cf. Proposition \ref{UNGRAFT PROP})
in the latter case yields trees $W_{i}\leq_{lex} U^{\leq (s,t_i)}$, which after grafted yield a tree
$W \leq_{lex} U$
such that $W$ contains all relations 
$(s,t_i)^{\uparrow S} \leq (s,t_i)$. 
But now $W$ contains the relations 
$e^{\uparrow S,T} \leq e^{\uparrow T} \leq e$ 
and hence a final generating $\leq_{lex}$ relation yields 
the desired
$V \leq_{lex} W \leq_{lex} U$.
\end{proof}

\begin{example}A typical illustration of the previous result follows.
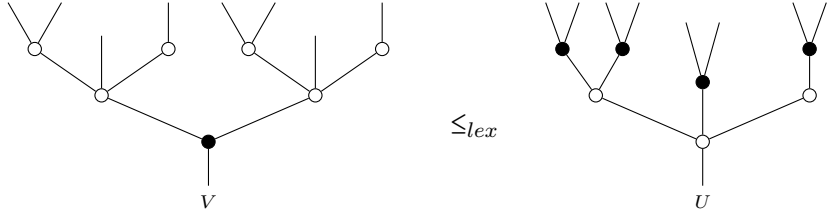
\[
	\begin{tikzpicture}[auto,grow=up, level distance = 2.25em, every node/.style={font=\scriptsize},dummy/.style = {circle,draw,inner sep=0pt,minimum size=1.75mm}]
		\node at (6.5,0) {$U$}
				child[level distance = 2.25em]{node [dummy,fill=white] {}
					child[sibling distance = 4em,level distance = 1.75em]{node [dummy] {}
						child[sibling distance = 1.em]{node [dummy,fill=black] {}
							child[sibling distance = 1.25em]
							child[sibling distance = 1.25em]
						edge from parent node [swap] {}}
					edge from parent node [swap] {}}
					child[sibling distance = 1.5em,level distance = 2.25em]{node [dummy,fill=black] {}
						child[sibling distance = 1.25em]
						child[sibling distance = 1.25em]
					edge from parent node [swap] {}}
					child[sibling distance = 4em,level distance = 1.75em]{node [dummy] {}
						child[sibling distance = 2em]{node [dummy,fill=black] {}
							child[sibling distance = 1.25em]
							child[sibling distance = 1.25em]
						edge from parent node [swap,near end] {}}
						child[sibling distance = 2.5em]{node [dummy,fill=black] {}
							child[sibling distance = 1.25em]
							child[sibling distance = 1.25em]
						edge from parent node [near end] {}}
					edge from parent node {}}
				edge from parent node [swap] {}};
		\node at (0,0) {$V$}
			child{node [dummy,fill=black] {}
				child[level distance = 1.75em,sibling distance = 8em]{node [dummy,fill=white] {}
					child[sibling distance = 2.5em,level distance = 1.75em]{node [dummy] {}
						child[sibling distance = 1.em]
					edge from parent node [swap] {}}
					child[sibling distance = 1.5em,level distance = 2.25em]
					child[sibling distance = 2.5em,level distance = 1.75em]{node [dummy] {}
						child[sibling distance = 2em]
						child[sibling distance = 2em]
					edge from parent node {}}
				edge from parent node [swap] {}}
				child[level distance = 1.75em,sibling distance = 8em]{node [dummy,fill=white] {}
					child[sibling distance = 2.5em,level distance = 1.75em]{node [dummy] {}
						child[sibling distance = 1.em]
					edge from parent node [swap] {}}
					child[sibling distance = 1.5em,level distance = 2.25em]
					child[sibling distance = 2.5em,level distance = 1.75em]{node [dummy] {}
						child[sibling distance = 2em]
						child[sibling distance = 2em]
					edge from parent node {}}
				edge from parent node [swap] {}}
			};
		\node at (3.5,1)[font=\large] {$\leq_{lex}$};
	\end{tikzpicture}
\]

\end{example}

\begin{lemma}\label{INTERTREE LEM}
Suppose $U \hookrightarrow S \otimes T$, 
$V \hookrightarrow S \otimes T$ 
are subtrees with common leaves and root. 
Then $F = U \cap V$ defines a full face of both $U$ and $V$.
\end{lemma}

\begin{proof}
As a set, $F$ could alternatively be defined as the underlying set of the composite inner face of $U$ that removes all inner edges of $U$ not in $V$, or vice versa. Thus, the real claim is that both constructions yield the same broad relations. Noting that the $\leq_d$ order relations on $U$, $V$ are induced from $S \otimes T$ (as argued in the proof of Proposition \ref{CLOSEDSUBTREE PROP}), this follows from Lemma \ref{INEQMAXIFF NOPEN LEM}.
\end{proof}

\begin{lemma}\label{COMMONFACE LEM}
Suppose $T$ is open.
 If $F$ is a common face (resp. inner face) of two elementary subtrees $U,V$, then $F$ is also a face (resp. inner face) of an elementary subtree $W$ such that $W \leq U$, $W \leq V$ (resp. $W \leq_{lex} U$, $W \leq_{lex} V$). 
In fact, in the inner face case the $\leq_{lex}$ inequalities factor through generating $\leq_{lex}$ inequalities involving only trees having $F$ as an inner face. 
\end{lemma}

\begin{proof}
Letting $r$, $\underline{l}$ denote the root and leaves of $F$,
by Corollary \ref{TALLOUTERFACT COR} one can replace $U$, $V$ with $U^{\leq r}_{\nless \underline{l}}$, $V^{\leq r}_{\nless \underline{l}}$, reducing to the case where $F,U,V$ have exactly the same leaves and root. Thus, by Lemma 
\ref{INTERTREE LEM} we are free to assume $F = U \cap V$.

If the root vertices $r^{\uparrow U} \leq r$, $r^{\uparrow V} \leq r$ coincide, the result follows by induction on $\leq$. Otherwise, we can assume that the root vertex of $U$ is $r^{\uparrow S} \leq r$ and that of $V$ is $r^{\uparrow T} \leq r$. 
Lemma \ref{INTERCHANGE LEM} now applies to 
$V_{\nless r^{\uparrow F}}$, and one can hence build
 $W \leq_{lex} V$ with a strictly larger intersection with $U$, finishing the proof. 
\end{proof}

Recall that a subset $B$ of a poset $\mathcal{P}$ is called \textit{convex} if $\bar{b} \leq b$ and $b \in B$ implies $\bar{b} \in B$.

\begin{proposition}\label{ANODYNEEXT PROP}
Let $S,T$ be trees and $\xi \in T$ an inner edge. Further, assume that either both $S$ and $T$ are open or that one of them is linear.
Set 
\[A =  \Omega[S] \otimes \Lambda^{\xi}[T] 
\underset{ \partial \Omega[S] \otimes  \Lambda^{\xi}[T]}{\coprod}
\partial \Omega[S] \otimes  \Omega[T]\]
and regard $A$ and the $\Omega[V]$ below as subpresheaves of $\Omega[S] \otimes \Omega[T]$. 
 
Then, for any convex subsets $B \subset B'$ of the poset $\mathsf{IE}_{\xi}(S \otimes T)$ (or, in the special case of $S$ a linear tree and $T$ not open, of the poset $\mathsf{IE}^{oplex}_{\xi}(S \otimes T)$), one has that 
\[A \cup \bigcup_{V \in B} \Omega[V] \to
  A \cup \bigcup_{V \in B'} \Omega[V]\]
is an inner anodyne extension.
\end{proposition}

To simplify notation we will throughout the proof suppress $\Omega$ from the notation, e.g., 
$A \cup \bigcup_{V \in B} \Omega[V]$
will be denoted simply as $A \cup \bigcup_{V \in B} V$.

The key to the proof is given by the following couple of lemmas.

\begin{lemma}\label{CHAREDGE1 LEM}
Suppose that either $S$, $T$ are both open or that $T$ is linear.

Then, for $U \in \mathsf{IE}_{\xi}(S \otimes T)$ and $B$ convex such that $\{V|V<U\}\subset B$, any edge $\xi_s$ of $U$ with vertex $\xi_s^{\uparrow T} \leq \xi_s$
is a characteristic edge (in the sense of \cite[Lemma 9.7]{MW08}), i.e., for each inner face $F$ of $U$ containing the edge $\xi_s$, then 
$F$ is in $A \cup \bigcup_{V \in B}V$ if and only if $F - \xi_s$ is. 
\end{lemma}

\begin{proof}
Suppose first that $F-\xi_s$ is in $A$ but that $F$ is not.

Either $F-\xi_s \hookrightarrow S' \otimes T$ or 
$F -\xi_s \hookrightarrow S \otimes T'$ for $S'\hookrightarrow S$, $T' \hookrightarrow T$ some maximal subface where in the latter case $T'\neq T-\xi$. Considering the cases in Notation \ref{FACETYPES NOT}, the stump cases are excluded by Corollary \ref{ROOTCOND COR} and in the full cases Lemma
\ref{CLOSEDSUBTREE LEM} implies that the only possibility is for $F-\xi_s$ to have no edge with $S$-coordinate $s$ while $F$ does. Further, if   $s$ were
to be a root or leaf of $S$, $F - \xi_s$ would still contain a root or leaf with $S$-coordinate $s$ (this latter case uses the fact that $T$ is open). Thus, the only possibility is $F - \xi_s \hookrightarrow (S - s) \otimes T$ for $s$ an inner edge of $S$.

Now let $\xi^{\uparrow <s}_s =e_1 \cdots e_k$ consist of the $\leq_d$-maximal $e_i=(s_i,t_i)$ such that both $e_i<_d\xi_s$ and $s_i<_d s$ and consider the subtree 
$U^{\leq \xi_s}_{\nless \xi^{\uparrow <s}_s}$. Then:
\begin{inparaenum}
\item[(i)] this tree has no leaf with $S$ coordinate $s$, or else that would be a leaf of $U$ (Remark \ref{OUTERFACELEAVES REM}), and thus also of $F$, so that it could not be $F - \xi_s \hookrightarrow (S - s) \otimes T$;
\item[(ii)] the leaf tuple of this tree is hence $\xi^{\uparrow <s}_s$;
\item[(iii)] by definition of $\xi^{\uparrow <s}_s$, all inner edges of this tree have $S$ coordinate $s$.
\end{inparaenum}
But the condition $F - \xi_s \hookrightarrow (S - s) \otimes T$ now implies that 
$F$ contains none of the inner edges of $U^{\leq \xi_s}_{\nless \xi^{\uparrow <s}_s}$, so that 
Lemma \ref{INTERCHANGE LEM} implies that 
$F$ is a subface of some 
$V <_{lex} U$, hence contained in $A \cup \bigcup_{V \in B} V$.

Suppose now that $F-\xi_s$ is a subface of some $V\in B$. By Lemma \ref{COMMONFACE LEM} (and its proof) we can assume that in fact $F-\xi_s$ is an inner subface of $V$ and $V <_{lex} U$. Further,  by the ``in fact'' part of Lemma \ref{COMMONFACE LEM} one can also assume that this is a generating $\leq_{lex}$ relation. But then $V$ necessarily contains $\xi_s$, since generating $\leq_{lex}$ relations do not add edges whose vertex is a $T$-vertex. Thus $F-\xi_s \hookrightarrow V$ implies 
$F \hookrightarrow V$.
\end{proof}

\begin{example}
The following is a typical tree illustration of $U^{\leq \xi_s}_{\nless \xi^{\uparrow <s}_s}$. In words, this subtree always:
\begin{inparaenum} 
\item[(i)] has $S$-vertices (in black), all of the same arity, below its leaves;
\item[(ii)] has a $T$-vertex above its root $\xi_s$ (in white);
\item[(iii)] all its remaining vertices are $T$-vertices (so that the edges marked $s$ have $S$-coordinate $s$).
\end{inparaenum}
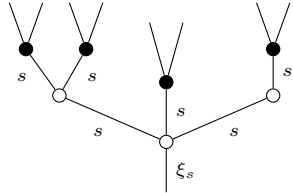
\[
	\begin{tikzpicture}[auto,grow=up, level distance = 2.25em, every node/.style={font=\scriptsize},dummy/.style = {circle,draw,inner sep=0pt,minimum size=1.75mm}]
		\node at (0,0) {}
			child[level distance = 2.25em]{node [dummy,fill=white] {}
				child[sibling distance = 4em,level distance = 1.75em]{node [dummy] {}
					child[sibling distance = 1.em]{node [dummy,fill=black] {}
						child[sibling distance = 1.25em]
						child[sibling distance = 1.25em]
					edge from parent node [swap] {$s$}}
				edge from parent node [swap] {$s$}}
				child[sibling distance = 1.5em,level distance = 2.25em]{node [dummy,fill=black] {}
					child[sibling distance = 1.25em]
					child[sibling distance = 1.25em]
				edge from parent node [swap] {$s$}}
				child[sibling distance = 4em,level distance = 1.75em]{node [dummy] {}
					child[sibling distance = 2em]{node [dummy,fill=black] {}
						child[sibling distance = 1.25em]
						child[sibling distance = 1.25em]
					edge from parent node [swap,near end] {$s$}}
					child[sibling distance = 2.5em]{node [dummy,fill=black] {}
						child[sibling distance = 1.25em]
						child[sibling distance = 1.25em]
					edge from parent node [near end] {$s$}}
				edge from parent node {$s$}}
			edge from parent node [swap] {$\xi_s$}};
	\end{tikzpicture}
\]
\end{example}

\begin{remark}
For the previous proof to work it is crucial for the tree $U^{\leq \xi_s}_{\nless \xi^{\uparrow <s}_s}$ to in fact have inner edges, as is ensured by the fact that $\xi$ is not a stump of $T$.
In this latter case we will instead need to use 
the following alternative lemma.
\end{remark}

\begin{lemma} \label{CHAREDGE2 LEM}
Suppose that $S$ is a linear tree (i.e. $S \simeq [n]$ for $[n] \in \Delta$).

Then for $U \in \mathsf{IE}_{\xi}^{oplex}(S \otimes T)$ the $\leq_d$-maximal edge of $U$ of the form $\xi_s$ is a characteristic edge (in the sense of Lemma \ref{CHAREDGE1 LEM}). 
\end{lemma}

\begin{proof}
Since $S$ is linear we will for simplicity label its edges as $0 \leq 1 \leq \cdots \leq n$.

Suppose first that $F-\xi_s$ is in $A$, so that repeating the argument in the proof of Lemma \ref{CHAREDGE1 LEM} we conclude it must be 
$F-\xi_s \hookrightarrow (S-s) \otimes T$ for $s<n$ (note that this makes sense even if $s=0$ is the leaf of $S$, which must be considered if $T$ is not open).

Since $s \neq n$, one can choose a $\leq_d$-minimal edge $a_{s+1}$ of $U$ such that $\xi_s \leq_d a_{s+1}$.
Then:
\begin{inparaenum}
\item[(i)] the characterization of $\xi_s$ implies $a \neq \xi$;
\item[(ii)] the characterization of $a_{s+1}$ implies that $U$ contains the $S$-vertex $a_s = a_{s+1}^{\uparrow S} \leq a_{s+1}$;
\item[(iii)] since $U$ contains both $\xi_s$ and $a_s$ it contains the $T$-vertex  $a_s^{\uparrow T} \leq a_s$, which can be rewritten as $a_{s+1}^{\uparrow S,T} \leq a_{s+1}^S$.
\end{inparaenum}
$U$ therefore contains the relations $a_{s+1}^{\uparrow S,T} \leq a_{s+1}^{\uparrow S} \leq a_{s+1}$, which since $F$ must collapse 
$a_s=a_{s+1}^{\uparrow S}$
yields that $F$ is a subface of the tree $V \leq_{oplex} U$ obtained by replacing $a_s = a_{s+1}^{\uparrow S}$ with $a_{s+1}^{\uparrow T}$.

The case of $F-\xi_s$ a subface of some $V \in B$ follows by an argument identical to that in the proof of Lemma \ref{CHAREDGE1 LEM}, except now noting that generating $\leq_{oplex}$ relations do not add edges whose vertex is a $S$-vertex.
\end{proof}

\begin{proof}[Proof of Proposition \ref{ANODYNEEXT PROP}]

Without loss of generality we can assume that $B'$ is obtained from $B$ by adding a single $\xi$-internal initial elementary tree $U$ with $\xi_s$ its corresponding edge.

We first note that the outer faces of $U$ are in $A \cup \bigcup_{V \in B} V$. 
Since a maximal outer face $\bar{U} \hookrightarrow U$ is always still elementary, $\bar{U}$ will be $\xi$-internal initial elementary unless
\begin{enumerate*}
  \item[(i)] $\bar{U}=U^{\leq(e,r_T)}$ (resp. $\bar{U}=U^{\leq(r_S,e)}$) is a root vertex face, in which case 
	$\bar{U} \hookrightarrow (S^{\leq e}) \otimes T$
	(resp. $\bar{U} \hookrightarrow S \otimes (T^{\leq e})$)
	and is hence in $A$; 
 \item[(ii)] $\bar{U}=U_{\nless \xi_s}$ \textit{and} $\bar{U}$ is no longer $\xi$-internal since it contains no $T$-vertices of the form $\xi_{\tilde{s}}^{\uparrow T} \leq \xi_{\tilde{s}}$.
But then it would be 
$\bar{U} \hookrightarrow S \otimes T_{\nless \xi}$
(by either Lemma \ref{CLOSEDSUBTREE LEM} or Corollary \ref{ROOTCOND COR})
and thus $\bar{U}$ is in $A$.
\end{enumerate*}

Finally, we let $\mathsf{Inn}_{\hat{\xi}_s}(U)$ denote the poset of inner faces of $U$ removing only edges \textit{other than} $\xi_s$. We claim that for any convex subsets 
$ C \subset C' \subset \mathsf{Inn}_{\hat{\xi}_s}(U)$ the map
\begin{equation}\label{ATTACHTREE EQ}
  A \cup \bigcup_{V \in B} V \cup \bigcup_{W \in C} W
   \to 
 A \cup \bigcup_{V \in B} V \cup \bigcup_{W \in C'} W
\end{equation}
is inner anodyne. 
We argue by induction on $C$ and again we can assume that $C'$ is obtained from $C$ by adding a single 
$X \in \mathsf{Inn}_{\hat{\xi}_s}(U)$ not yet in the domain of (\ref{ATTACHTREE EQ}). 
The concavity of $C,C'$ and the characteristic edge condition in Lemmas \ref{CHAREDGE1 LEM}, \ref{CHAREDGE2 LEM} then
imply that the only faces of $X$ not in the source of (\ref{ATTACHTREE EQ}) are precisely $X$ and $X - \xi_s$, showing that (\ref{ATTACHTREE EQ}) is a pushout of 
$\Lambda^{\xi_s}[X] \to \Omega[X]$, finishing the proof.
\end{proof}

In the $G$-equivariant case, given an inner edge orbit $G \xi$, we write $\mathsf{IE}_{G\xi}(S \otimes T)$ for the poset of initial elementary trees containing at least one $T$-vertex of the form
$(g\xi)^{\uparrow T}_s \leq(g\xi)_s$ (alternatively, one has $\mathsf{IE}_{G\xi}(S \otimes T)
=\bigcup_{g \in G}\mathsf{IE}_{g\xi}(S \otimes T)$). Note that in this case the group $G$ acts on the poset $\mathsf{IE}_{G\xi}(S\otimes T)$ as well. The following is the equivariant version of Proposition \ref{ANODYNEEXT PROP}.

\begin{proposition}\label{EXPG PROP}
Let $S,T \in \Omega_G$ be $G$-trees and $\xi \in T$ an inner edge. Further, assume that either both $S$ and $T$ are open or that one of them is linear (i.e. of the form $G/H \cdot [n]$).
Set 
\[
	A = \Omega[S]\otimes \Lambda^{G \xi}[T] 
	\coprod_{\partial \Omega[S]\otimes \Lambda^{G \xi}[T]}
	\partial \Omega[S] \otimes  \Omega[T].
\]
and regard $A$ and the $\Omega[V]$ below as subpresheaves of $\Omega[S] \otimes \Omega[T]$.

Then for any $G$-equivariant convex subsets $B \subset B'$ of $\mathsf{IE}_{G\xi}(S \otimes T)$ (or, in the special case of $S$ a linear tree and $T$ not open, of $\mathsf{IE}^{oplex}_{G\xi}(S \otimes T)$) one has that
\[
	A \cup \bigcup_{V \in B} \Omega[V] \to A \cup \bigcup_{V \in B'} \Omega[V]
\]
is a $G$-inner anodyne extension.
\end{proposition}

Again we will suppress $\Omega$ from the notation of the proof.

\begin{proof}
Note that we are free to assume $S,T \in \Omega^G \subset \Omega_G$, i.e. that $S$, $T$ are actual trees with a $G$-action rather than $G$-indecomposable forests. Indeed, otherwise writing $S \simeq G \cdot_{H} S_e$, $T \simeq G \cdot_{K} T_e$ for $S_e \in \Omega^H, T_e \in \Omega^K$ yields a decomposition
\[S \otimes T \simeq \coprod_{[g]\in H \backslash G /K} 
G \cdot_{H \cap gKg^{-1}} S_e \otimes gT_e,\]
where when regarding $S_e, g T_e \in \Omega^{H \cap gKg^{-1}}$ we omit the forgetful functors.

In analogy with the non-equivariant case we can assume $B'$ is obtained from $B$ by adding the \textit{$G$-orbit} of a single $\xi$-internal initial elementary tree $U$ with $\xi_s$ the corresponding edge. Let $H\leq G$ denote the $G$-isotropy of $U$
in $\mathsf{IE}_{G \xi}(S\otimes T)$.

That the outer faces of any of the conjugates $gU$ are in 
$A \cup \bigcup_{V \in B} V$ follows by the corresponding non-equivariant argument in the proof of Proposition \ref{ANODYNEEXT PROP}.

The key is now to prove the equivariant analogues of Lemmata \ref{CHAREDGE1 LEM} and \ref{CHAREDGE2 LEM}, stating that $H \xi_s$ is a \textit{characteristic edge orbit} of $U$, i.e., that for each inner face $F$ of $U$ 
with isotropy 
$\bar{H} \leq H$ and
containing an edge $(h \xi)_{(h s)} \in H \xi_s$, then 
$F$ is in $A \cup \bigcup_{V \in B} V$ iff 
$F - \bar{H} (h\xi)_{(h s)}$ is (note the condition on isotropy).
By equivariance, we may without loss of generality assume that $F$ contains $\xi_s$ itself. 

When proving the equivariant analogue of Lemma \ref{CHAREDGE1 LEM}, in the case of $F - \bar{H} \xi_s$ in $A$ the argument in that proof yields that $F$ itself must already lack all the inner edges of at least one of the 
$\bar{H}$-conjugates of the $U^{\leq \xi_s}_{\nless \xi^{\uparrow <s}_s}$ subtree (and hence, by $\bar{H}$-equivariance, all of them) and therefore applying Lemma \ref{INTERCHANGE LEM} again shows that $F$ is a subface of some $V <_{lex} U$. 
In the case $F - \bar{H} \xi_s$ in some $V < U$, repeating the argument in the proof we can again assume $V <_{lex} U$ via a generating $\leq_{lex}$ relation and such $V$ must likewise contain all the edges in $H \xi_s$. 

Proving the equivariant analogue of Lemma \ref{CHAREDGE2 LEM} requires no changes to the proof, since defining $a_{s+1}$ in the same way one still concludes that $F$ must lack $a_s$ (in fact, $F$ must lack $\bar{H} a_s$, though that fact in not needed).

Lastly, we equivariantly modify the last two paragraphs in the proof of Proposition \ref{ANODYNEEXT PROP}: 
setting $\mathsf{Inn}_{\widehat{H \xi_s}}(U)$ to be the $H$-poset of inner faces of $U$ lacking only edges \textit{not in} $H \xi_s$, we show by induction on 
$H$-equivariant concave subsets 
$C \subset C' \subset \mathsf{Inn}_{\widehat{H \xi_s}}(U)$
that the map
\begin{equation}\label{ATTACHTREE EQ2}
	A \cup \bigcup_{V \in B} V \cup \left( \bigcup_{g \in G,W \in C} gW \right)
		\to
	A \cup \bigcup_{V \in B} V \cup \left( \bigcup_{g \in G,W \in C'} gW \right)
\end{equation}
is inner $G$-anodyne.
Again we can assume that $C'$ is obtained from $C$ by adding the $H$-orbit of a single $X$ with $H$-isotropy $\bar{H} \leq H$  and not in the domain of (\ref{ATTACHTREE EQ2}) (further, $X$ can be chosen to contain $\xi_s$).
At this point an extra equivariant argument is needed: one needs to know that the $G$-isotropy of $X$ coincides with its $H$-isotropy $\bar{H}$. To see this, note that if it was otherwise then $X$ would be contained in both $U$ and a distinct conjugate $gU$, and Lemma \ref{COMMONFACE LEM} would imply that $X$ is already in the domain of (\ref{ATTACHTREE EQ2}).
Finally, 
 repeating the ``characteristic edge (orbit)'' argument we see that the map (\ref{ATTACHTREE EQ2}) is a pushout of
\[
	G \cdot_{\bar{H}} \left( \Lambda^{\bar{H}\xi_{s}}[X] \to \Omega[X] \right),
\]
finishing the proof.
\end{proof}

We now adapt the previous proof to deduce the easier Theorem \ref{OUTERIN THM}.

\begin{proof}[Proof of Theorem \ref{OUTERIN THM}]
	The argument follows by attempting to follow the proof of Proposition \ref{EXPG PROP} when $T=[1]$ and $\xi=1$ (note that Definition \ref{IEETA DEF} still makes sense, although the edge $1_s$ of $U$ may now possibly be the root, hence not internal). The only case where (the equivariant analogue) of Lemma \ref{CHAREDGE1 LEM} does not provide a characteristic edge orbit is when the root vertex of $U$ is $1_r^{\uparrow T} \leq 1_r$,
	in which case $U$ will be in $A$ unless (using the notation in the proof of Proposition \ref{MONOPUSHOUTPROD PROP}) it is $\pi(U)=S$, so that in fact it is $U = S \otimes \{0\} \cup \{1_r\}$.
Denoting this latest $U$ as $S \star \eta$ and noting that it is the maximum of the poset 
$\mathsf{IE}_{\xi}(S \otimes [1])$ one concludes that letting 
$A \to B$ denote
\[A \to A \cup 
 \bigcup_{V \in \mathsf{IE}_{\xi}(S \otimes [1]),V \neq S \star \eta} V
\]
this is indeed inner $G$-anodyne. The pushout (\ref{OUTPUSH EQ}) follows by noting that the only face of $S \star \eta$ not in $B$ is $(S \star \eta)-\eta$
(in fact, the only other face $F$ such that $\pi(F)=S$ is $(S \star \eta)-0_r$, which is a common face of the elementary tree $V$ obtained by applying a $\leq_{lex}$ relation to the root of $S \star \eta$).
\end{proof}

\subsection{Dendroidal join}\label{DENDJOIN SEC}

We now turn to the equivariant version of the dendroidal join $\star$ discussed in \cite[\S 4]{CM11}, which will be needed to understand the last piece of the filtration in Theorem \ref{OUTERIN THM}. We recall that several categories of forests were discussed in \S \ref{CATFOR SEC}.

\begin{definition}\label{JOIN DEF}
Given an object $F \in \Phi$ and $[n] \in \Delta$ we define $F \star [n] \in \Omega$ as the broad poset having underlying set $F \amalg [n]$ and relations
\begin{itemize}
	\item $e_1 \cdots e_n \leq e$ if $e_i,e \in F$ and $e_1 \cdots e_n \leq^F e$;
	\item $i \leq j$ if $i,j \in [n]$ and $i \leq^{[n]} j$;
	\item $e_1 \cdots e_n \leq i$ if $e_j \in F$, $i \in [n]$ and $e_1 \cdots e_n \leq^F \underline{r}_F$.
\end{itemize}
\end{definition}

\begin{example} %
	As explained in \cite[\S 4.3]{CM11}, one can readily visualize $\star$ when using tree diagrams, such as in the following example. %
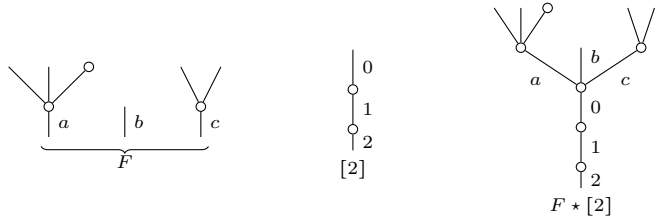
\[%
	\begin{tikzpicture}[auto,grow=up, every node/.style={font=\scriptsize},level distance = 1.5em,dummy/.style    = {circle,draw,inner sep=0pt,minimum size=1.25mm}]%
	\begin{scope}%
	\tikzstyle{level 2}=[sibling distance=1.5em]%
		\node at (0,0) {} %
			child{node [dummy] {} %
				child{node [dummy] {}} %
				child %
				child %
			edge from parent node[swap]{$a$}}; %
		\node at (1,0) {} %
			child{edge from parent node[swap]{$b$}}; %
		\node at (2,0) {} %
			child{node [dummy] {} %
				child %
				child %
			edge from parent node[swap]{$c$}}; %
		\draw[decorate,decoration={brace,amplitude=2.5pt}] (2.1,0) -- (-0.1,0) node[midway]{$F$}; %
	\end{scope} %
	\begin{scope} %
		\node at(4,-0.3) {$[2]$}%
			child{node [dummy] {}%
				child{node [dummy] {}%
					child{%
					edge from parent node[swap]{$0$}}%
				edge from parent node[swap]{$1$}}	%
			edge from parent node[swap]{$2$}};%
	\end{scope}
	\begin{scope} %
		\tikzstyle{level 4}=[sibling distance=2.25em]%
		\tikzstyle{level 5}=[sibling distance=1em]%
		\node at(7,-0.8) {$F \star [2]$}%
			child{node [dummy] {}%
				child{node [dummy] {}%
					child{node [dummy] {}%
						child{node [dummy] {}%
							child
							child
						edge from parent node[swap]{$c$}}
						child{edge from parent node[swap,near end]{$b$}}
						child{node [dummy] {}%
							child{node [dummy] {}}%
							child
							child
						edge from parent node{$a$}}
					edge from parent node[swap]{$0$}}%
				edge from parent node[swap]{$1$}}	%
			edge from parent node[swap]{$2$}};%
	\end{scope}
	\end{tikzpicture}
\]
Further, note that when $F=\emptyset$ is the empty forest, it is 
$\underline{r}_F=\epsilon$, and since $\epsilon \leq \epsilon$, $\emptyset \star [n]$ adds a stump at the top of $[n]$.
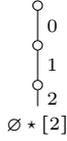
\[%
	\begin{tikzpicture}[auto,grow=up, every node/.style={font=\scriptsize},level distance = 1.5em,dummy/.style    = {circle,draw,inner sep=0pt,minimum size=1.25mm}]%
	\begin{scope}%
	\tikzstyle{level 2}=[sibling distance=1.5em]%
		\node at(4,-0.3) {$\emptyset \star [2]$}%
			child{node [dummy] {}%
				child{node [dummy] {}%
					child{node [dummy] {}%
					edge from parent node [swap] {$0$}}%
				edge from parent node [swap] {$1$}}	%
			edge from parent node [swap] {$2$}};%
	\end{scope}%
	\end{tikzpicture}
\]
\end{example}
We now discuss the functoriality of $\star$. As implicit in the discussion in \cite[\S 4.5]{CM11}, $\star$ is only functorial with respect to some maps of forests. Indeed, it is clear from the third condition in Definition \ref{JOIN DEF} that $\star$ will be functorial in $F$ precisely with respect to the maps in $\Phi_w$. 

Moreover, the canonical inclusions $[n] \to F \star [n]$, $F \to F \star [n]$ can be encoded thusly: letting $\Delta_+$, $\Phi_{w+}$, $\Phi_{i+}$ denote the categories $\Delta$, $\Phi_w$, $\Phi_i$ together with an additional initial object $+$, one has the following.

\begin{proposition}
	$\star$ defines a bifunctor
	\[\Phi_{w+} \times \Delta_+ \xrightarrow{-\star-} \Phi_{i+}\]
	such that $+ \star [n] = [n]$, $F \star + = F$.
\end{proposition}
Note that $\star$ usually lands in $\Omega_+$, the only exceptions occurring when the second input is the additional initial object.

We now extend the join operation to presheaves by defining 
\[\mathsf{fSet}_w \times \mathsf{sSet} \xrightarrow{-\star-} \mathsf{dSet}\]
to be the composite (writing $u \colon \Xi \hookrightarrow \Xi_+$ for the inclusion and $u^{\**} \colon \mathsf{Set}^{\Xi_+^{op}}\rightleftarrows \mathsf{Set}^{\Xi^{op}} \colon u_{\**}$ for the standard adjunction)
\begin{equation} \label{STAR DEF}
	\mathsf{fSet}_w \times \mathsf{sSet}
		\xrightarrow{u_{*}\times u_{*}} 
	\mathsf{Set}^{\Phi_{w+}^{op}} \times \mathsf{Set}^{\Delta_+^{op}}
		\xrightarrow{\times}
	\mathsf{Set}^{\Phi_{w+}^{op} \times \Delta_+^{op}}
		\xrightarrow{Lan_{\star}}
	\mathsf{Set}^{\Phi_{i+}^{op}}
		\xrightarrow{u^{*}}
	\mathsf{dSet}.
\end{equation}

\begin{remark}
	Unpacking (\ref{STAR DEF}) one can write (cf. \cite[Defn. 1.2.8.1]{Lu09})
	\begin{equation}\label{JOINFORMULA EQ}
		(X\star Y)(T) = X(T) \amalg Y(T) \amalg \coprod_{F\to T,[n]\to T,T\simeq F\star [n]} X(F) \times Y([n]),
	\end{equation}
	where $Y(T)=\emptyset$ when $T$ is not linear.
\end{remark}

\begin{remark}\label{CONNCOLIM REM}
Due to the passage through the $(-)_+$ categories in (\ref{STAR DEF}), $\mathsf{fSet}_w \times \mathsf{sSet} \xrightarrow{-\star-} \mathsf{dSet}$ does not preserve colimits in each variable. 
Rather, the functors $F \star (-)$, $(-) \star [n]$ preserve colimits when mapping into the under categories 
$\mathsf{dSet}_{F/}$, $\mathsf{dSet}_{[n]/}$. Therefore, $\star$ does nonetheless preserve \textit{connected colimits} in each variable.
\end{remark}

The following is an equivariant generalization of the key technical lemma 
\cite[Lemma 4.10]{CM11} combined with key arguments in the proof of \cite[Thm. 4.2]{CM11}.

\begin{proposition}\label{JOINPUSHOUT PROPOSITION}
	Let $A \xrightarrow{f} B$ be a normal monomorphism in $\mathsf{fSet}_w^G$ (defined as in Proposition
\ref{NORMCHAR PROP}(iii))
 and $C \xrightarrow{g} D$ be a left anodyne map in $\mathsf{sSet}$.
	Then
	\[
		f\square^{\star}g \colon A \star D \coprod_{A \star C} B \star C \to C\star D
	\]
	is $G$-inner anodyne.
\end{proposition}

\begin{proof}
	That $f \square^{\star} g$ is indeed a monomorphism whenever $f$, $g$ are monomorphisms follows directly from (\ref{JOINFORMULA EQ}).

	In lieu of Remark \ref{CONNCOLIM REM} concerning connected colimits, it suffices to consider the case where $f$, $g$ have the form $\partial \Phi_w[F] \to \Phi_w[F]$ and $\Lambda^i[n] \to \Delta[n]$, $0 \leq i <n$.
But in that case $f\square^{\star}g$ is simply the inner horn inclusion 
	\[\Lambda^i(F \star [n]) \to \Omega(F \star [n]).\]
\end{proof}

\begin{proof}[Proof of Theorem \ref{JOINLIFT THM}]
	First note that the conditions on $S$ are equivalent to saying that $S \simeq F \star [1]$ for some $F \in \Phi_w^G$. One can thus rewrite the left vertical map in Theorem \ref{JOINLIFT THM} as
	\[(\Lambda^1[1] \to \Delta[1]) \square^{\star} (\emptyset \to \Phi_w[F])\]
and denoting by $A\backslash(-)$ the right adjoint to 
$\mathsf{sSet} \xrightarrow{A \star (-)} \mathsf{dSet}^G_{A\star +/}$ standard adjunction arguments allows us to convert (\ref{JOINLIFT EQ}) into the equivalent lifting problem
\begin{equation}\label{JOINLIFT2 EQ}
\begin{tikzcd}
	\Lambda^1[1] \ar[r] \ar[d]     & \Phi_w[F] \backslash X \ar[d] &
	\{0\} \ar[r] \ar[d]     & \Phi_w[F] \backslash X \ar[d]
\\
	\Delta[1] \ar[r] \ar[dashed]{ru} & 
	\Phi_w[F] \backslash Y 
	\underset{\emptyset \backslash Y}{\times}
	\emptyset \backslash X &
	\Delta[1] \ar[r] \ar[dashed]{ru} & 
	\Phi_w[F] \backslash Y
	\underset{i^{\**}(Y^G)}{\times}	i^{\**}(X^G)
\end{tikzcd}
\end{equation}
where the right hand diagram merely simplifies the notation on the left: 
 $\emptyset \backslash Z \simeq i^{\**}(Z^G)$, where we caution that $\emptyset \not\simeq \Phi_w[\emptyset]$, the former being the \textit{empty presheaf} and the latter the representable presheaf on the \textit{empty forest}.

Standard repeated applications of Proposition \ref{JOINPUSHOUT PROPOSITION} (setting $A = \emptyset$ or $Y = *$ as needed) yield that:
\begin{inparaenum}
	\item[(i)] $\Phi_w[F]\backslash X \to i^{\**}(X^G)$, $\Phi_w[F]\backslash Y \to i^{\**}(Y^G)$ are left fibrations;
	\item[(ii)] $\Phi_w[F]\backslash X$, $\Phi_w[F]\backslash Y$ are left fibrant and thus $\infty$-categories;
	\item[(iii)] the rightmost map 
	$\Phi_w[F]\backslash X \to \Phi_w[F]\backslash Y \times_{i^{\**}(Y^G)}	i^{\**}(X^G)$ in (\ref{JOINLIFT2 EQ}) is a left fibration.
\end{inparaenum}	
Therefore, the map 
	$\Phi_w[F]\backslash Y \times_{i^{\**}(Y^G)}	i^{\**}(X^G) \to i^{\**}(X^G)$ is itself a left fibration (it is a pullback of $\Phi_w[F]\backslash Y \to i^{\**}(Y^G)$) and since left fibrations conserve equivalences 
\cite[Prop. 2.7]{Joy02} the
image of the lower map in (\ref{JOINLIFT2 EQ}) is a equivalence.
The result now follows since equivalences can be lifted over left fibrations between $\infty$-categories
\cite[Props. 2.4 and 2.7]{Joy02}.
\end{proof}

\begin{remark}
	The interested reader may note that Proposition \ref{JOINPUSHOUT PROPOSITION} and Theorem \ref{JOINLIFT THM} have notably shorter proofs than the analogue results \cite[Lemma 4.10]{CM11} and \cite[Thm. 4.2]{CM11}. In fact, some of our arguments closely resemble the proof for the simplicial case as found in \cite[Prop. 1.2.4.3]{Lu09}.
	The brevity of these arguments is the reason for our introduction of the forest category $\Phi_w$: from our  perspective, many of the arguments in \cite{CM11} are replaced with the task of showing that the boundaries $\partial \Phi_w[F]$ satisfy the usual formal properties.
\end{remark}

\section{Model structure on equivariant dendroidal sets}
\label{MODELSTR SEC}

We now adapt the treatment in \cite{CM11} to equip the category $\mathsf{dSet}^G$ 
with a model structure where the cofibrations are the normal monomorphisms (cf. Prop. \ref{NORMCHAR PROP}) and the fibrant objects are the $G$-$\infty$-operads (cf. Defn. \ref{GINFTYOP DEF}). Since we have already established the equivariant analogues of the main technical results needed (Theorems \ref{EXPNPROP THM}, \ref{OUTERIN THM} and 
 \ref{JOINLIFT THM}), the proofs in \cite{CM11} will now carry over to our context with only minor changes needed.

\subsection{Existence of the model structure}\label{EXIST SEC}

Our goal in this first section is to establish Proposition \ref{DSETGMODEL PROP},
which abstractly builds the desired model structure on 
$\mathsf{dSet}^G$ by generalizing \cite[Prop. 3.12]{CM11}.

As it turns out, adapting the treatment in \cite[\S 3]{CM11} requires only minimal modifications 
(essentially ``adding $G$ to the statements therein''). As such, we will be brief in our discussion, and refer the reader to \cite[\S 3]{CM11} for extra details.

\vskip 5pt

Recall the notation $J = N (0 \rightleftarrows 1)$ for the nerve of the groupoid generated by a single isomorphism between two distinct objects. As in \cite[\S 3.2]{CM11} we write $J_d=i_{!}(J)$ when regarding $J$ as a dendroidal set, and we will further regard $J_d$ as a $G$-dendroidal set by equipping it with the \textit{trivial} $G$-action.

Following \cite[\S 3.2]{CM11}, we define $\mathsf{An}$, the class of $J$-anodyne extensions, to be the saturation of the $G$-inner horn inclusions together with the maps
\begin{equation}\label{JANODEF EQ}
	\{i\}\otimes \Omega[T]
	\underset{\{i\}\otimes \partial \Omega[T]}{\coprod}
	J_d \otimes \partial \Omega[T]
		\to
	J_d \otimes \Omega[T],
\qquad i=0,1,\quad T \in \Omega_G.
\end{equation}
A $G$-dendroidal set $X$ (resp. map $X \to Y$) is then called $J$-fibrant (resp. $J$-fibration) if it has the right lifting property with respect to the maps in $\mathsf{An}$. 

The following generalizes the 
necessary parts of \cite[Prop. 3.3]{CM11}
(note that that result is slightly corrected in the
erratum to \cite{CM11}).

\begin{proposition} \label{SLIGHTCOR PROP}
Let $A \to B$ be a normal monomorphism in 
$\mathsf{dSet}^G$. Then:
\begin{itemize}
	\item[(i)] for $i=0,1$ the map 
$\{i\} \otimes B \amalg_{\{i\} \otimes A} J_d \otimes A	\to J_d \otimes B$ is in $\mathsf{An}$;
	\item[(ii)] 
$\{0,1\} \otimes B \amalg_{\{0,1\} \otimes A} J_d \otimes A	\to J_d \otimes B$ is a normal monomorphism, 
which is $\mathsf{An}$ if  $A \to B$ is.
\end{itemize}
\end{proposition}

\begin{proof}
	(i) follows since normal monomorphisms are built by attaching boundary inclusions $\partial \Omega[T] \to \Omega[T]$ for $T \in \Omega_G$. Similarly, (ii) follows from the ``$S$ linear'' cases of
	Proposition \ref{MONOPUSHOUTPROD PROP} and
	Theorem \ref{EXPNPROP THM} (recall that $J_d$ is a simplicial set).
\end{proof}

As in \cite[\S 3.4]{CM11},
for a $G$-dendroidal set $B$ we now write 
$\mathsf{dSet}^G/B$ for the category of $G$-dendroidal sets over $B$ and let $\mathsf{An}_B$
denote the class of maps in 
$\mathsf{dSet}^G/B$ whose image in $\mathsf{dSet}^G$ is in $\mathsf{An}$.

\begin{proposition} Write 
$\partial^0 \colon \{0\} \to J_d$, 
$\partial^1 \colon \{1\} \to J_d$,
$\sigma \colon J_d \to \{\**\}$
for the standard maps\footnote{Note that $\{\**\} = i_{!}(\**)$
 denotes the terminal simplicial set, not the terminal dendroidal set.}
and abbreviate
$\mathcal{J} = (J_d \otimes -, \partial^0 \otimes -, 
	\partial^1 \otimes -, \sigma \otimes -)$.

Then whenever $B$ is normal the pair $(\mathcal{J}, \mathsf{An}_B)$ is a homotopical structure on $\mathsf{dSet}^G/B$ as defined in \cite[Def. 1.3.14]{Ci06}.
\end{proposition}

\begin{proof}
There are two parts: showing that $\mathcal{J}$
is an \textit{elementary homotopical datum} as in
\cite[Def. 1.3.6]{Ci06} and that 
$\mathsf{An}_B$ is a \textit{class of anodyne extensions with respect to $\mathcal{J}$} as in \cite[Def. 1.3.10]{Ci06}. We will refer to the axiom names used therein.

For the first claim, both axioms DH1 and DH2 follow from Proposition \ref{MONOPUSHOUTPROD PROP} since when $B$ is normal all monomorphisms over $B$ are normal monomorphisms.

For the second claim, axiom An0 follows by \cite[Lemma 1.3.52]{Ci06} while An1,An2 follow from Proposition \ref{SLIGHTCOR PROP}.
\end{proof}

Combining \cite[Thm. 1.3.22, Prop. 1.3.31, Prop. 1.3.36, Lemma 1.3.52]{Ci06} now yields the following (this generalizes \cite[Prop. 3.5, Remark 3.6]{CM11}).

\begin{proposition}\label{FIRSTMOD PROP}
	For any normal $G$-dendroidal set $B$, 
	the category $\mathsf{dSet}^G/B$  has a left proper cofibrantly generated model structure such that
	\begin{itemize}
		\item the cofibrations are the monomorphisms;
		\item $J$-anodyne extensions over $B$ are trivial cofibrations;
		\item the fibrant objects are the $X \xrightarrow{p} B$
		such that $p$ is a $J$-fibration in $\mathsf{dSet}^G$;
		\item a map $X \xrightarrow{f} X'$ between fibrant
		$X \to B$, $X'\to B$ is a fibration in $\mathsf{dSet}^G/B$
		iff $f$ is a $J$-fibration in $\mathsf{dSet}^G$.
	\end{itemize}	
\end{proposition}

The following generalizes \cite[Lemma 3.7]{CM11}
again following from \cite[Cor. 1.3.35]{Ci06}.

\begin{lemma}\label{AN LEMMA}
Let $X \to Y$ be a trivial fibration between normal $G$-dendroidal sets. Then any section $s \colon Y \to X$ is in $\mathsf{An}$.
\end{lemma}

Fix once and for all a normalization $E_{\infty}$ of the terminal $G$-dendroidal set $\**$, i.e. a trivial fibration 
$E_{\infty} \to \**$ with $E_{\infty}$ a normal $G$-dendroidal set.

The following generalizes \cite[Lemma 3.9]{CM11}.

\begin{lemma}\label{ANYNOR LEM}
For any normal $G$-dendroidal set $X$ and map $a\colon X \to E_{\infty}$, the map $(a,id) \colon X \to E_{\infty} \times X$
is in $\mathsf{An}$.
\end{lemma}

\begin{proof}
Since $(a,id)$ is a section of the projection 
$X \times E_{\infty} \to X$, this follows from Lemma \ref{AN LEMMA}.
\end{proof}

The following generalizes \cite[Lemma 3.10]{CM11}.

\begin{lemma}\label{IMPLYFIB LEM}
Let $i \colon A \to B$ be a map of normal $G$-dendroidal sets and 
$p \colon X \to Y$ any map of $G$-dendroidal sets.
Then $p$ has the right lifting property with respect to $i$ iff, for any map $B \to E_{\infty}$, 
the map 
$E_{\infty} \times X
 \xrightarrow{id \times p}
 E_{\infty} \times Y$
 has the right lifting property with respect to $i$ in 
 $\mathsf{dSet}^G/E_{\infty}$.
\end{lemma}

\begin{proof}
Given a lifting problem as on the left below, one obtains a lifting problem as on the right by arbitrarily choosing a map 
$B\to E_{\infty}$ (such a map always exists since $B$ is assumed normal). It is clear that the lifting problems are equivalent.
	\[
		\begin{tikzcd}
			A \ar{r} \ar{d}[swap]{i} & X \ar{d}{p}
		& &
			A \ar{r} \ar{d}[swap]{i} & E_{\infty} \times  X \ar{d}{id \times p}
		\\
			B \ar{r} \ar[dashed]{ru} & Y
		& &
			B \ar{r} \ar[dashed]{ru} & E_{\infty} \times Y
		\end{tikzcd}
	\]
\end{proof}

We finally obtain the model structure on $\mathsf{dSet}^G$,  generalizing \cite[Prop. 3.12]{CM11}.

\begin{proposition}\label{DSETGMODEL PROP}
	$\mathsf{dSet}^G$ is equipped with a left proper cofibrantly generated model structure such that
	\begin{itemize}
		\item the cofibrations are the normal monomorphisms;
		\item inner $G$-anodyne extensions are trivial cofibrations;
		\item the fibrant objects are the $J$-fibrant objects;
		\item the fibrations between fibrant objects are the $J$-fibrations.
	\end{itemize}
\end{proposition}

\begin{proof}
	As in \cite[Prop. 3.12]{CM11}, the model structure is built via the adjunction
	\[
	p_{!}\colon \mathsf{dSet}^G/E_{\infty}
		\rightleftarrows
	\mathsf{dSet}^G \colon p^{\**}	
	\]
where $p_{!}(X \to B) = X$, $p^{\**}(X) = (E_{\infty} \times X \to E_{\infty})$. Noting that both $p_!$ and $p^{\**}$ preserve all colimits it follows that condition (iii) in
\cite[Prop. 1.4.23]{Ci06}
reduces to verifying that $p^{\**}p_{!}(j)$
is a trivial cofibration in $\mathsf{dSet}^G/E_{\infty}$ whenever
$j$ is. Since this follows from 
Lemma \ref{ANYNOR LEM}, the transfer model structure exists.

	The claim that cofibrations are normal monomorphisms follows since all monomorphisms over $E_{\infty}$ are normal. The converse follows since all boundary inclusions 
	$\partial \Omega[T] \to \Omega[T]$ are in the image of $p_!$,
and the claim concerning anodyne extensions follows by the same argument. The fibrancy claims follow from Lemma \ref{IMPLYFIB LEM}.
Lastly, left properness follows from that in 
Proposition \ref{FIRSTMOD PROP} together with the observation that
$p^{\**}$ preserves both cofibrations and colimits and detects weak equivalences.
\end{proof}

\begin{remark}
	As in the case of the model structure on $\mathsf{dSet}$ or of the Joyal model structure on $\mathsf{sSet}$, the trivial cofibrations (resp. fibrations) in Proposition \ref{DSETGMODEL PROP} do not coincide with the inner $G$-anodyne extensions (resp. $J$-fibrations), but merely contain (resp. are contained in) them.
\end{remark}

\subsection{Characterization of fibrant objects}\label{FIBOBJCHAR SEC}

Much as in \cite{CM11}, the bulk of the work is now that of characterizing the fibrant objects as indeed being the $G$-$\infty$-operads.

We will need to make use of the adjunction
\[
	u^{\**} \colon \mathsf{dSet}_G 
		\rightleftarrows 
	\mathsf{dSet}^{G}\colon u_{\**}
\]
discussed in Proposition \ref{REFLEXCAT PROP}.

\begin{definition}\label{NORMALMONDSETLOWG DEF}
	The class of \textit{normal monomorphims} of $\mathsf{dSet}_G$ is the saturation of the maps of the form $u_{\**}(A\to B)$ for $A \to B$ a normal monomorphism in $\mathsf{dSet}^G$.
	Further, a map $X \to Y$ in $\mathsf{dSet}_G$ is called a \textit{trivial fibration} if it has the right lifting property with respect to normal monomorphisms.
\end{definition}

We now extend some notation from \cite[\S 6.1]{CM11}. 

As usual, for $X$ an $\infty$-category, $k(X)$ will denote the maximal Kan complex inside $X$ and $\tau(X)\in \mathsf{Cat}$
 will denote its homotopy category.

\begin{notation}\label{XTOTHEPARX NOT}
For a $G$-$\infty$-operad $X$ and simplicial set $K$ (thought of as having a trivial $G$-action), we define $X^{(K)}\in \mathsf{dSet}_G$ to have $T$-dendrices (recall $T \in \Omega_G$) the maps 
	\[i_{!}(K)\otimes \Omega[T] \xrightarrow{a} X\]
such that for each edge orbit $G/H \cdot \eta \xrightarrow{G/H \cdot e} T$ the induced map 
	\[K \xrightarrow{a_e} i^{\**} \left( X^H \right)\]
factors through $k\left(i^{\**}\left(X^H\right)\right)=k\left(\left(i^{\**}(X)\right)^H\right)$.
\end{notation}

\begin{remark}
Note that for $X$ a $G$-$\infty$-category one always has $k(X^G) \subset k(X)^G$, but that this inclusion is rarely an equality.
\end{remark}

\begin{notation}\label{KAX NOT}
For a normal $G$-dendroidal set $A \in \mathsf{dSet}^G$ and a $G$-$\infty$-operad $X$ 
we define $k(A,X) \in \mathsf{sSet}$ to have $n$-simplices the maps
	\[i_!(\Delta[n]) \otimes A \xrightarrow{b} X\]
such that, for all element orbits $G/H \cdot \eta \xrightarrow{G/H \cdot a} A$ the induced map
	\[ \Delta[n] \xrightarrow{b_a} i^{\**} \left( X^H \right) \]
factors through $k\left(i^{\**}\left(X^H\right)\right)=k\left(\left(i^{\**}(X)\right)^H\right)$.
\end{notation}

Note that there are canonical isomorphisms
\begin{equation}\label{CANISODSETG EQ}
Hom_{\mathsf{sSet}}(K,k(A,X)) \simeq 
Hom_{\mathsf{dSet}_G} \left( u_{\**}(A),X^{(K)} \right).
\end{equation}

The following is the analogue of \cite[Thm. 6.4]{CM11}. Recall that a map $\mathcal{C} \to \mathcal{D}$ in $\mathsf{Cat}$
is called a \textit{categorical fibration} if it has the right lifting property against the inclusion $1 \to [1]$, where $[1]=(0\to 1)$.

\begin{theorem}\label{XDELTA1 THM}
	Let $p \colon X \to Y$ be a $G$-inner fibration between $G$-$\infty$-operads. Then 
	\begin{equation}\label{EV1MAP EQ}
	ev_1 \colon X^{(\Delta[1])} \to Y^{(\Delta[1])} \times_{u_{\**}(Y)} u_{\**}(X)
	\end{equation}
 has the right lifting property with respect to inclusions 
	$u_{\**}(\partial \Omega[S] \to \Omega[S])$ for any $G$-tree $S \in \Omega_G$ with at least one vertex.
	
	Consequently, (\ref{EV1MAP EQ})	is a trivial fibration 
	in $\mathsf{dSet}_G$ iff all maps $\tau i^{\**}\left(X^H \to Y^H\right)$ for $H\leq G$ are categorical fibrations.
\end{theorem}

\begin{proof}
	Noting that it is $S \simeq G \cdot_H S_e$ for some $S_e \in \Omega^H$, $H \leq G$, it suffices to deal with the case $S \in \Omega^G$.
	
	The proof of the main claim now follows exactly as in the proof of \cite[Thm. 6.4]{CM11} by replacing uses of \cite[Thm. 5.2]{CM11} and \cite[Thm. 4.2]{CM11} with the equivariant analogues Theorems \ref{OUTERIN THM} and 
\ref{JOINLIFT THM}.

For the ``consequently'' part, one needs only note that in the equivariant context there are now multiple $G$-trees with no vertices, namely the $G$-trees of the form $G/H \cdot \eta$.
\end{proof}

The following is the analogue of \cite[Prop. 6.7]{CM11}.

\begin{proposition}\label{KANFIB PROP}
	Let $p \colon X \to Y$ be a $G$-inner fibration between $G$-$\infty$-operads.
	If $\tau i^{\**}\left(X^H \to Y^H\right)$ is a categorical fibration for all $H \leq G$  then, for any monomorphism between normal dendroidal sets $A \to B$, the map 
	\begin{equation}\label{KISKAN EQ}
	k(B,X) \to k(B,Y) \times_{k(A,Y)} k(A,X)
	\end{equation}
is a Kan fibration between Kan complexes.
\end{proposition}

\begin{proof}
	We will mainly refer to the proof of \cite[Prop. 6.7]{CM11} while indicating the main changes. First, note that it follows 
from Theorem \ref{EXPNPROP THM}(ii)	that the map
$\underline{Hom}(B,X) \to \underline{Hom}(B,Y)
\times_{\underline{Hom}(A,Y)} \underline{Hom}(A,X)$
of simplicial mapping spaces is a $G$-inner fibration between $G$-$\infty$-categories, and thus so is (\ref{KISKAN EQ}).
	As in \cite[Prop. 6.7]{CM11}, it now suffices to check that (\ref{KISKAN EQ}) has the right lifting property against the ``left pushout products''
	\[
		\partial \Delta[n] \times \Delta[1] \cup_{\partial \Delta[n] \times \{1\}} \Delta[n]\times \{1\} 
		\to \Delta[n] \times \Delta[1]
	\]
(i.e. thanks to \cite[Lemma 6.5]{CM11} one needs only consider the case $i=1$ of (\ref{JANODEF EQ})).
A lifting problem 
	\begin{equation}\label{KANFIBLIFT EQ1}
		\begin{tikzcd}
		\partial \Delta[n] \times \Delta[1] \cup \Delta[n]\times \{1\} \ar[r] \ar[d] & k(B,X)  \ar[d] \\
		\Delta[n] \times \Delta[1] \ar[r] \ar[dashed]{ru}[swap]{h} & k(B,Y) \times_{k(A,Y)} k(A,X) \\
		\end{tikzcd}
	\end{equation}
induces a (a priori non equivalent) lifting problem
	\begin{equation}\label{KANFIBLIFT EQ2}
		\begin{tikzcd}
			u_{\**}\left(\partial \Omega[n] \otimes B \cup \Omega[n] \otimes A\right) \ar[r] \ar[d] & X^{(\Delta[1])} \ar[d] \\
			u_{\**}\left( \Omega[n] \otimes B \right) 
			\ar[r] \ar[dashed]{ru}[swap]{\bar{h}} & Y^{(\Delta[1])} \times_{u_{\**}(Y)} u_{\**}(X).
		\end{tikzcd}
	\end{equation}
That the lift $\bar{h}$ in (\ref{KANFIBLIFT EQ2}) exists  follows from Theorem \ref{XDELTA1 THM} and it hence remains to check that the adjoint map $i_{!}(\Delta[n]\times \Delta[1])\otimes B \to X$ indeed provides the map $h$ in (\ref{KANFIBLIFT EQ1}). I.e., one must check that for any element orbit $G/H \cdot b$ of $B$ the induced map 
$\Delta[n]\times \Delta[1] \to i^{\**}(X^H)$ factors through $k(i^{\**}(X^H))$ (note that the existence of $\bar{h}$ only guarantees such a factorization for the restriction along 
$\{0,1,\cdots,n\}\times \Delta[1] \subset \Delta[n] \times \Delta[1]$). The $n=0$ case is immediate and the $n>0$ case follows by arguing using the ``$2$-out-of-$3$ property'', just as in the penultimate paragraph of the proof of \cite[Prop. 6.7]{CM11}.

Standard arguments (setting $A=\emptyset$ of $Y=\**$ just as in \cite[Prop. 6.7]{CM11} or as at the end of the proof of Theorem \ref{JOINLIFT THM}) finish the proof.
\end{proof}

Combining the previous result with (\ref{CANISODSETG EQ})
now yields the following.


\begin{corollary}\label{TRIVFIB COR}
Let $p \colon X \to Y$ be a $G$-inner fibration between $G$-$\infty$-operads such that $\tau i^{\**}\left(X^H \to Y^H\right)$ is a categorical fibration for all $H \leq G$. Then, for any anodyne extension of simplicial sets $K \to L$,
\[X^{(L)} \to Y^{(L)} \times_{Y^{(K)}} X^{(K)}\]
is a trivial fibration in $\mathsf{dSet}_G$.
\end{corollary}


We now obtain our sought generalization of \cite[Thm. 6.10]{CM11}.

\begin{theorem}\label{FIBRANTOBJ THM}
A $G$-dendroidal set $X$ is $J$-fibrant iff it is a $G$-$\infty$-operad.
Further, a $G$-inner fibration $p \colon X \to Y$ between $G$-$\infty$-operads is a $J$-fibration iff 
$\tau i^{\**}\left(X^H \to Y^H\right)$ is a categorical fibration for all $H \leq G$.
\end{theorem}

\begin{proof} It suffices to prove the ``further'' claim.
	Moreover, the ``only if'' direction is a direct consequence of
	\cite[Cor. 1.6]{Joy02}.
	Unwinding definitions and adjunction properties it thus remains to show that
	\[X^{J_d} \to Y^{J_d} \times_{Y^{\{0\}}} X^{\{0\}}\]
	is a trivial fibration in $\mathsf{dSet}^G$
	if $\tau i^{\**}\left(X^H \to Y^H\right)$ is a categorical fibration for all $H \leq G$. We now note that since any map $J_d \to Z$ necessarily factors through $k(Z)$ the map
	\[u_{\**}\left(X^{J_d} \to Y^{J_d} \times_{Y^{\{0\}}} X^{\{0\}}\right)\]
	coincides with the map
	\[X^{(J_d)} \to Y^{(J_d)} \times_{Y^{(\{0\})}} X^{(\{0\})}\]
	which is a trivial fibration in $\mathsf{dSet}_G$ by Corollary \ref{TRIVFIB COR}. The result now follows since 
$\mathsf{dSet}^G$ is a reflexive subcategory of $\mathsf{dSet}_G$, so that
 $u_{\**}(f)$ is a trivial fibration iff $f$ is.
\end{proof}

The following follows exactly as in \cite[Cor. 6.11]{CM11}.

\begin{corollary}\label{WEAKEQUIVCHAR}
The weak equivalences in $\mathsf{dSet}^G$ are the smallest class containing the inner $G$-anodyne extensions, the trivial fibrations and closed under ``$2$-out-of-$3$''.
\end{corollary}

\section{Indexing system analogue results}
\label{INDEXSYS SEC}

In this section we follow the lead of \cite{BH15} and build variant model structures on $\mathsf{dSet}^G$ associated to  indexing systems, a notion originally introduced in  \cite[Def. 3.22]{BH15}, 
which we repackage (and slightly extend) in 
Definition \ref{WEAKINDEXSYS DEF}.

\begin{definition}
 	A \textit{$G$-graph subgroup} of $G\times \Sigma_n$ is a subgroup $K \leq G \times \Sigma_n$ such that $K \cap \Sigma_n = \**$.
\end{definition}

\begin{remark}
$G$-graph subgroups are graphs of homomorphisms
$G \geq H \to \Sigma_n$.
\end{remark}

\begin{definition}
A \textit{$G$-vertex family} is a collection
\[\mathcal{F} = \coprod_{n\geq 0} \mathcal{F}_{n}\]
where each $\mathcal{F}_n$ is a family of $G$-graph subgroups of $G\times \Sigma_n$ closed under subgroups and conjugation.

Further, a $H$-set $X$ for a subgroup $H\leq G$ is called a 
\textit{$\mathcal{F}$-set} if for some (and hence any) choice of isomorphism $X\simeq \{1,\cdots,n\}$ the graph subgroup of $G \times \Sigma_n$ encoding the $H$-action on $\{1,\cdots,n\}$ is in $\mathcal{F}$.
\end{definition}

\begin{definition}
	Let $\mathcal{F}$ be a $G$-vertex family.	
	A $G$-tree $T$ is called a \textit{$\mathcal{F}$-tree} if for all edges $e \in T$ with isotropy $H$ one has that the $H$-set $e^{\uparrow}$ is a $\mathcal{F}$-set.
\end{definition}

It is clear that whenever $T \to S$ is either an outer face or a quotient, $S$ being a $\mathcal{F}$-tree implies that so is $T$. However, the same is typically not true for inner faces and degeneracies.

\begin{definition}\label{WEAKINDEXSYS DEF}
A $G$-vertex family $\mathcal{F}$ is called a 
\textit{weak indexing system} if $\mathcal{F}$-trees form 
a sieve of $\Omega_G$, i.e. if for any map $T \to S$ with $S$ a
$\mathcal{F}$-tree then it is also 
$T$ a $\mathcal{F}$-tree.
In this case we denote the sieve of $\mathcal{F}$-trees by $\Omega_{\mathcal{F}} \subset \Omega_G$.

Additionally, $\mathcal{F}$ is called an \textit{indexing system} if every $\mathcal{F}_n$ contains all subgroups $H \times \** \leq G \times \Sigma_n$ for $H \leq G$, $n\geq 0$.
\end{definition}

\begin{remark}
 Closure under degeneracies is simply the statement that $\mathcal{F}_1$ contains all subgroups 
 $H =
H \times \Sigma_1 
 \leq G \times \Sigma_1$ for $H \leq G$.
\end{remark}

\begin{remark}
Since Definition \ref{WEAKINDEXSYS DEF} may at first seem to be quite different from the original \cite[Def. 3.22]{BH15}, we now address the equivalence between the two.
To a $H$-set with orbital decomposition $H/K_1 \amalg \cdots \amalg H/K_n$ one can associate the $G$-corolla with orbital representation as follows.
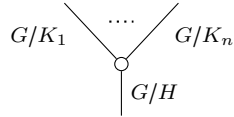
\[
\begin{tikzpicture}
[grow=up,auto,level distance=2.3em,every node/.style = {font=\footnotesize},dummy/.style={circle,draw,inner sep=0pt,minimum size=1.75mm}]
	\node at (0,0) {}
		child{node [dummy] {}
			child{
			edge from parent node [swap,near end] {$G/K_n$} node [name=Kn] {}}
			child{
			edge from parent node [near end] {$G/K_1$}
node [name=Kone,swap] {}}
		edge from parent node [swap] {$G/H$}
		};
		\draw [dotted,thick] (Kone) -- (Kn) ;
\end{tikzpicture}
\]
Note that for any of its roots $r$ one has that $r^{\uparrow}$ is a $G$-conjugate of the $H$-set $H/K_1 \amalg \cdots \amalg H/K_n$. The conditions (cf. \cite[Def. 3.22]{BH15}) that indexing systems are closed under disjoint unions \cite[Def. 3.19]{BH15}
 and sub-objects \cite[Def. 3.21]{BH15} of
$\mathcal{F}$-sets are then encoded by taking inner faces of $\mathcal{F}$-trees of the form
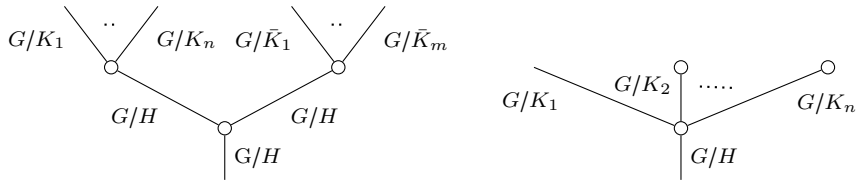
\[
\begin{tikzpicture}
[grow=up,auto,level distance=2.3em,every node/.style = {font=\footnotesize},dummy/.style={circle,draw,inner sep=0pt,minimum size=1.75mm}]
	\tikzstyle{level 2}=[sibling distance = 8.5em]
	\tikzstyle{level 3}=[sibling distance = 3.5em]
	\node at (0,0) {}
		child{node [dummy] {}
			child{node [dummy] {}
				child{
				edge from parent node [swap,near end] {$G/\bar{K}_m$} node [name=Kn] {}}
				child{
				edge from parent node [near end] {$G/\bar{K}_1$}
node [name=Kone,swap] {}}
			edge from parent node [swap] {$G/H$}}
			child{node [dummy] {}
				child{
				edge from parent node [swap,near end] {$G/K_n$} node [name=K2n] {}}
				child{
				edge from parent node [near end] {$G/K_1$}
node [name=K2one,swap] {}}
			edge from parent node {$G/H$}}
		edge from parent node [swap] {G$/H$}		
		};
		\draw [dotted,thick] (Kone) -- (Kn)	;
		\draw [dotted,thick] (K2one) -- (K2n)	;
	\begin{scope}
	\tikzstyle{level 2}=[sibling distance = 5.5em]
	\node at (6,0) {}
		child{node [dummy] {}
			child{node [dummy] {}
			edge from parent node [swap,near end] {$G/K_n$} node [name=Kn] {}}
			child{node [dummy] {}
			edge from parent node [near end] {$G/K_2$}
node [name=Kone,swap,near end] {}}
			child{
			edge from parent node [near end] {$G/K_1$}}
		edge from parent node [swap] {$G/H$}
		};
		\draw [dotted,thick] (Kone) -- (Kn) ;
	\end{scope}
\end{tikzpicture}
\]
while the closure under self-induction \cite[Def. 3.20]{BH15}
is similarly encoded by $\mathcal{F}-$trees as on the left below.
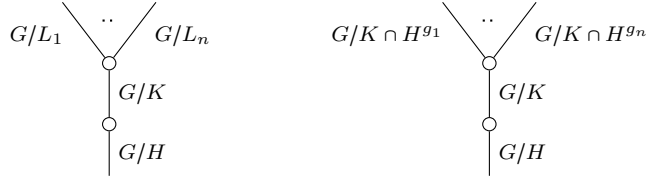
\[
\begin{tikzpicture}
[grow=up,auto,level distance=2.3em,every node/.style = {font=\footnotesize},dummy/.style={circle,draw,inner sep=0pt,minimum size=1.75mm}]
	\tikzstyle{level 2}=[sibling distance = 8.5em]
	\tikzstyle{level 3}=[sibling distance = 3.5em]
	\node at (0,0) {}
		child{node [dummy] {}
			child{node [dummy] {}
				child{
				edge from parent node [swap,near end] {$G/L_n$} node [name=Kn] {}}
				child{
				edge from parent node [near end] {$G/L_1$}
node [name=Kone,swap] {}}
			edge from parent node [swap] {$G/K$}}
		edge from parent node [swap] {$G/H$}		
		};
		\draw [dotted,thick] (Kone) -- (Kn)	;
	\node at (5,0) {}
		child{node [dummy] {}
			child{node [dummy] {}
				child{
				edge from parent node [swap,near end] {$G/K\cap H^{g_n}$} node [name=Kn] {}}
				child{
				edge from parent node [near end] {$G/K\cap H^{g_1}$}
node [name=Kone,swap] {}}
			edge from parent node [swap] {$G/K$}}
		edge from parent node [swap] {$G/H$}		
		};
		\draw [dotted,thick] (Kone) -- (Kn)	;
\end{tikzpicture}
\]
Closure under cartesian products \cite[Def. 3.22]{BH15} is in fact redundant, as the double coset formula 
$H/K \times H/L \simeq \coprod_{[g] \in L \backslash H/ K}
{H/K \cap L^g}$
allows such $H$-sets to be built using self inductions as displayed by the rightmost tree above (the case of products of sets with multiple orbits being then obtained via disjoint units).
\end{remark}

Definitions \ref{BOUNDINCG DEF}, \ref{INNERHORNG DEF}
and \ref{GINFTYOP DEF} admit weak indexing system analogues.

\begin{definition}
Let $\mathcal{F}$ be a weak indexing system.

A \textit{$\mathcal{F}$-boundary inclusion} (resp. 
\textit{$\mathcal{F}$-inner horn inclusion}) is a boundary inclusion
$\partial \Omega[T] \to \Omega[T]$ (resp. inner horn inclusion $\Lambda^{G e}[T] \to \Omega[T]$) with $T \in \Omega_{\mathcal{F}}$.

A monomorphism is called $\mathcal{F}$-normal (resp. $\mathcal{F}$-anodyne) if it is in the saturation of $\mathcal{F}$-boundary inclusions (resp. $\mathcal{F}$-inner horn inclusions) under pushouts, transfinite compositions and retracts.

Finally, a $G$-dendroidal set $X$ is called a $\mathcal{F}$-$\infty$-operad if it has the right lifting property with respect to all $\mathcal{F}$-inner horn inclusions.
	\[
		\begin{tikzcd}
		\Lambda^{Ge}[T] \ar[r] \ar[d] & X \\ 
		\Omega[T] \ar[dashed]{ur}		
		\end{tikzcd}
	\]
\end{definition}

We now list the necessary modifications to extend the results in this paper to the indexing system case.

A direct analogue of Proposition \ref{NORMCHAR PROP} yields that
$X \in \mathsf{dSet}^G$ is $\mathcal{F}$-normal (i.e. $\emptyset \to X$ is a $\mathcal{F}$-normal monomorphism) iff all dendrices $x \in X(T)$ have $\mathcal{F}$-isotropy, i.e. isotropies $\Gamma \leq G \times \Sigma_T$ that are graph subgroups for  partial homomorphisms
$G \geq H \to \Sigma_T$ such that the induced $G$-tree $G \cdot_H T$ is a $\mathcal{F}$-tree.

It then follows that, much like normal dendroidal sets, $\mathcal{F}$-normal dendroidal sets form a sieve, i.e., for any map $X \to Y$ with $Y$ a $\mathcal{F}$-normal dendroidal set then so is $X$. 

Noting that the subtrees of $S \otimes T$ are $\mathcal{F}$-trees whenever $S$, $T$ are $\mathcal{F}$-trees (since the generating vertices/broad relations of $S \otimes T$ are induced from those of $S$, $T$),
it follows that $\Omega[S] \otimes \Omega[T]$ is then $\mathcal{F}$-normal so that the sieve condition implies that  
Proposition \ref{MONOPUSHOUTPROD PROP} generalizes to the $\mathcal{F}$-normal case.

Likewise, the key results Theorem \ref{EXPNPROP THM} and 
\ref{OUTERIN THM} immediately generalize by replacing the terms 
``$G$-tree'' and ``$G$-anodyne'' with ``$\mathcal{F}$-tree'' and ``$\mathcal{F}$-anodyne''. This is because their proofs, while long, ultimately amount to identifying suitable edge orbits of suitable subtrees of $S\otimes T$ and then attaching the corresponding equivariant horns.

Likewise, Proposition \ref{JOINPUSHOUT PROPOSITION} generalizes to the $\mathcal{F}$ case for the same reason, and hence so does Theorem \ref{JOINLIFT THM}, since its proof is an application of
Proposition \ref{JOINPUSHOUT PROPOSITION}.

We can now prove Theorem \ref{INDEXSYSMAIN THM}.

\begin{proof}[proof of Theorem \ref{INDEXSYSMAIN THM}]
	The proof of the existence of the model structure follows just as in \S \ref{EXIST SEC}. The only notable changes are as follows:
	in defining $J_{\mathcal{F}}$-anodyne extensions one uses only $\mathcal{F}$-inner horns and those maps in (\ref{JANODEF EQ})
	for $T \in \Omega_{\mathcal{F}}$;
	the term ``normal'' is replaced with ``$\mathcal{F}$-normal'' throughout
	(note that any monomorphism over a $\mathcal{F}$-normal dendroidal set is a $\mathcal{F}$-normal monomorphism).
	

The characterization of the $J_\mathcal{F}$-fibrant objects as being the $\mathcal{F}$-$\infty$-operads follows by repeating the arguments in \S \ref{FIBOBJCHAR SEC}, though some care is needed when adapting the definitions preceding 
Theorem \ref{XDELTA1 THM}. Firstly, letting 
$\Omega_{\mathcal{F}} \subset \Omega_{G}$ denote the sieve of $\mathcal{F}$-trees, one sets 
$\mathsf{dSet}_{\mathcal{F}} = \mathsf{dSet}^{\Omega^{op}_{\mathcal{F}}}$, leading to an adjunction
\[
u^{\**} \colon \mathsf{dSet}_{\mathcal{F}}
	\leftrightarrows
\mathsf{dSet}^G \colon u_{\**}
\]
allowing for the $\mathcal{F}$-normal monomorphisms of 
$\mathsf{dSet}_{\mathcal{F}}$
to be defined from the $\mathcal{F}$-normal monomorphisms in $\mathsf{dSet}^G$ just as in Definition \ref{NORMALMONDSETLOWG DEF}.

For a $\mathcal{F}$-$\infty$-operad and simplicial set $K$,
one defines $X^{(K)} \in \mathsf{dSet}_{\mathcal{F}}$
just as in Notation \ref{XTOTHEPARX NOT}
while for $A \in \mathsf{dSet}^{G}$ a $\mathcal{F}$-normal dendroidal set and $\mathcal{F}$-$\infty$-operad X one defines
$k(A,X) \in \mathsf{sSet}$ just as in Notation \ref{KAX NOT}.

The proofs of Theorems \ref{XDELTA1 THM}, 
Proposition \ref{KANFIB PROP} and Theorem \ref{FIBRANTOBJ THM}
now extend mutatis mutandis by using the $\mathcal{F}$ versions of Theorems \ref{OUTERIN THM} and \ref{JOINLIFT THM}.
\end{proof}

\end{document}